\documentclass{aims}

\pdfoutput=1

\usepackage{amsmath}
  \usepackage{paralist}
  \usepackage{graphics} 
  \usepackage{epsfig} 
\usepackage{graphicx}
 \usepackage[colorlinks=true]{hyperref}
 \hypersetup{urlcolor=blue, citecolor=red}
\usepackage{wrapfig}
 \usepackage{amsfonts}
\usepackage{amssymb}
\hypersetup{urlcolor=blue, citecolor=red}

\def\currentvolume{X}
 \def\currentissue{X}
  \def\currentyear{200X}
   \def\currentmonth{XX}
    \def\ppages{X--XX}
     \def\DOI{10.3934/xx.xx.xx.xx}

\newtheorem{theorem}{Theorem}[section]

\theoremstyle{definition}

\numberwithin{equation}{section}

\title[Two-Scale numerical simulation of sand transport problems]
      {Two-Scale numerical simulation of sand transport problems}

\author[Ibrahima Faye, Emmanuel Fr\'enod and Diaraf Seck]{}


 \email{ibrahima.faye@uadb.edu.sn}
 \email{emmanuel.frenod@univ-ubs.fr}
 \email{diaraf.seck@ucad.edu.sn}

\newcommand{\ds}{\displaystyle}

\newcommand{\N}{\mathbb{N}}
\newcommand{\R}{\mathbb{R}}
\newcommand{\torus}{{\mathbb{T}}}

\begin{document}
\maketitle

\centerline{\scshape Ibrahima Faye}
\medskip
{\footnotesize
 \centerline{Universit\'e de Bambey,UFR S.A.T.I.C, BP 30 Bambey (S\'en\'egal),}
\centerline{  Ecole Doctorale de
                      Math\'ematiques et Informatique. }
   \centerline{Laboratoire de Math\'ematiques de la D\'ecision et d'Analyse Num\'erique}
\centerline{ (L.M.D.A.N) F.A.S.E.G)/F.S.T. }
} 

\medskip

\centerline{\scshape Emmanuel Fr\'enod and Diaraf Seck}
\medskip
{\footnotesize
 \centerline{Universit\'e Europ\'eenne de Bretagne, LMBA(UMR6205)}
   \centerline{ Universit\'e de Bretagne-Sud, Centre Yves Coppens,}
  \centerline{ Campus de Tohannic, F-56017,
       Vannes Cedex, France}
  \centerline{ET}
   \centerline{Projet INRIA Calvi, Universit\'{e} de Strasbourg, IRMA,}
   \centerline{7 rue Ren\'e Descartes, F-67084 Strasbourg Cedex, France}
 }
\medskip
{\footnotesize
 \centerline{Universit\'e Cheikh Anta Diop de Dakar, BP 16889 Dakar Fann,}
  \centerline{  Ecole Doctorale de
                      Math\'ematiques et Informatique. }
   \centerline{Laboratoire de Math\'ematiques de la D\'ecision et d'Analyse Num\'erique}
    \centerline{ (L.M.D.A.N) F.A.S.E.G)/F.S.T. }
      \centerline{ET}
   \centerline{UMMISCO, UMI 209, IRD, France}
\bigskip

 \centerline{(Communicated by the associate editor name)}
\footnote{\thanks{This work is supported by NLAGA(Non Linear Analysis, Geometry and Application Project).}
}

\pagestyle{myheadings}
 \renewcommand{\sectionmark}[1]{\markboth{#1}{}}
\renewcommand{\sectionmark}[1]{\markright{\thesection\ #1}}
%

\begin{abstract} In this paper we consider the model built in \cite{FaFreSeN} for short term dynamics of dunes in tidal area. 
We construct a Two-Scale Numerical Method based on the fact that the solution of the equation which has oscillations Two-Scale converges to the solution of a well-posed problem.
This numerical method uses on Fourier series.
\end{abstract}


\section{Introduction}       
This paper deals with numerical simulations of sand transport problems. Its goal is to build a Two-Scale Numerical Method to simulate dynamics of dunes in tidal area.\\
This paper enters a work program concerning the development of Two-Scale Numerical
Methods to solve PDEs with oscillatory singular perturbations linked with physical phenomena.
In Ailliot, Fr\'enod and Monbet\,\cite{AiFreMon}, such a method is used to manage the tide oscillation for long term
drift forecast of objects in coastal ocean waters.
Fr\'enod, Mouton and Sonnendr\"{u}cker\,\cite{FrMoutonSon}  made simulations of the 1D Euler equation using a Two-Scale Numerical Method. In Fr\'enod, Salvarani and Sonnendr\"{u}cker\,\cite{FrSalSon}, such a
method is used to simulate a charged particle beam in a periodic focusing channel.
Mouton \cite{MoutT, MoutA} developped a Two-Scale Semi Lagrangian Method for beam and plasma applications.

~

We consider the following model, valid for short-term dynamics of dunes, built and studied in \cite{FaFreSeN}:

\begin{equation}\label{3.5n}\left\{\begin{array}{cc}
    \ds \frac{\displaystyle\partial z^{\epsilon}}{\partial t}-\frac{1}{\epsilon}\nabla\cdot(\mathcal{A}^{\epsilon}\nabla z^{\epsilon})=\frac{1}{\epsilon}\nabla\cdot\mathcal{C}^{\epsilon},\\
    \ds z^{\epsilon}_{|t=0}=z_{0},
\end {array}\right.\end{equation}
where $z^\epsilon=z^\epsilon(t,x)$ is the dimensionless seabed altitude. For a given $T,\,\,t\in (0,T)$ stands for the dimensionless time and
$x\in \torus^{2},\,\,\torus^2$ being the two dimensional torus $\mathbb R^2/ \mathbb Z^2,$   stands for the dimensionless position and $\mathcal{A}^{\epsilon},\,\,\mathcal{C}^{\epsilon}$ are given by
\begin{equation}\label{3.1n}
    \mathcal{A}^{\epsilon}(t,x)=\widetilde{\mathcal{A}}^{\epsilon}(t,x)+\epsilon\widetilde{\mathcal{A}}_{1}^{\epsilon}(t,x),
\end{equation}
and
\begin{equation}\label{3.3n}
    \mathcal{C}^{\epsilon}(t,x)=\widetilde{\mathcal{C}}^\epsilon(t,x)+\epsilon\widetilde{\mathcal{C}}_1^\epsilon(t,x),
\end{equation}
where, for three positive constants $a$, $b$ and $c$,
\begin{equation}\label{3.6n}\widetilde{\mathcal{A}}^{\epsilon}(t,x)
=\widetilde{\mathcal{A}}(t,\frac{t}{\epsilon},x)=a\,g_a(|\mathcal{U}(t,\frac{t}{\epsilon},x)|),
\end{equation}
\begin{equation}\label{3.7n} \widetilde{\mathcal{C}}^{\epsilon}(t,x)=\widetilde{\mathcal{C}}(t,\frac{t}{\epsilon},x)=
c\,g_c(|\mathcal{U}(t,\frac{t}{\epsilon},x)|)\,
\frac{\mathcal{U}(t,\frac{t}{\epsilon},x)}{|\mathcal{U}(t,\frac{t}{\epsilon},x)|},
\end{equation}
and
\begin{equation}\label{3.7n1}\widetilde{\mathcal{A}}_{1}^{\epsilon}(t,x)=\widetilde{\mathcal{A}}_{1}(t,\frac{t}{\epsilon},x),\,\,
\widetilde{\mathcal{C}}_{1}^{\epsilon}(t,x)=\widetilde{\mathcal{C}}_{1}(t,\frac{t}{\epsilon},x),
\end{equation}
with
\begin{equation}
\widetilde{\mathcal{A}}_{1}(t,\theta,x)=-ab\mathcal{M}(t,\theta,x)\,g_a(|\mathcal{U}(t,\theta,x)|)\,\,
\textrm{and}\,\,\,\widetilde{\mathcal{C}}_{1}(t,\theta,x)=
-cb\mathcal{M}(t,\theta,x)\,g_c(|\mathcal{U}(t,\theta,x)|)\,\frac{\mathcal{U}(t,\theta,x)}{|\mathcal{U}(t,\theta,x)|}.
\end{equation}
$\mathcal U$ and $\mathcal M$ are the dimensionless water velocity and height.\\
The small parameter $\epsilon$ involved in the model is the ratio between the main tide period $\frac{1}{\bar\omega}=13$ hours and an observation time which is about three months i.e. $\epsilon=\frac{1}{\bar t\bar\omega}=\frac{1}{200}.$\\
The following hypotheses on $g_a,\,\,g_c,\,\,\mathcal U$ and  $\mathcal M$ given in (\ref{eq2n}) and (\ref{eq8n})
are technical assumptions and are needed to prove Theorem\,\ref{thHomSecHomn}.
Functions $g_{a}$ and $g_{c}$ are regular functions
on $\R^+$ and satisfy
\begin{equation}\label{eq2n}\left\{ \begin{array}{ccc}\vspace{0.25cm}
\hspace{-4cm}g_{a}\geq g_{c}\geq0,\,\,g_{c}(0)=g'_{c}(0)=0,\\
\vspace{0.25cm}\hspace{-2cm}\exists\, d\geq0,\sup_{u\in\mathbb{R}^{+}}|g_{a}(u)|+\sup_{u\in\mathbb{R}^{+}}|g'_{a}(u)|\leq d,\\
   \vspace{0.25cm} \hspace{-3.15cm}\sup_{u\in\mathbb{R}^{+}}|g_{c}(u)|+\sup_{u\in\mathbb{R}^{+}}|g'_{c}(u)|\leq d,\\
   \vspace{0.25cm} \exists\, U_{thr}\geq0,\,\,\exists\, G_{thr}>0,\,\,\textrm{such that}\,\, u\geq U_{thr}\Longrightarrow g_{a}(u)\geq G_{thr}.\end{array}\right.
\end{equation}

Functions $\mathcal{U}$ and $\mathcal{M}$ are regular
and satisfy:
\begin{equation}\label{eq8n}\left\{\begin{array}{ccc}\vspace{0.25cm}
\hspace{-4cm}\ds\theta\longmapsto(\mathcal{U},\mathcal{M})\,\,\textrm{is periodic of period 1,}\\
\vspace{0.25cm}
\hspace{-3cm}\ds |\mathcal{U}|,\,\,|\frac{\partial \mathcal{U}}{\partial t}|,\,\,|\frac{\partial \mathcal{U}}{\partial \theta}|,\,\,|\nabla \mathcal{U}|,
\hspace{3cm}\\ \hspace{3cm}
\vspace{0.25cm}
\hspace{-5.7cm}\ds |\mathcal{M}|,\,\,|\frac{\partial \mathcal{M}}{\partial t}|,\,\,|\frac{\partial \mathcal{M}}{\partial \theta}|,\,\,|\nabla \mathcal{M}| \,\,\textrm{are bounded by}\,\,d,\\
\vspace{0.25cm}
            \ds\forall\, (t,\theta,x)\in\mathbb{R}^{+}\times\mathbb{R}\times\torus^{2},\,\,|\mathcal{U}(t,\theta,x)|\leq U_{thr}\Longrightarrow
 \hspace{2.5cm}\\  \hspace{0.5cm}
 \vspace{0.25cm}
 \hspace{-5cm}\Big(\ds\frac{\partial\mathcal{U}}{\partial t} (t,\theta,x) =0,\,\,\nabla\mathcal{U}(t,\theta,x)=0,
 \\ \hspace{2.5cm}
 \vspace{0.25cm}
 \hspace{-6cm}\ds\frac{\partial\mathcal{M}}{\partial t} (t,\theta,x) =0,\,\,\textrm{ and}\,\,\nabla\mathcal{M}(t,\theta,x)=0\Big),\\
            \exists\, \theta_{\alpha}<\theta_{\omega}\in[0,1]\,\,\textrm{such that}\,\, \forall\,\,\theta\in [\theta_{\alpha},\theta_{\omega}]\Longrightarrow|\mathcal{U}(t,\theta,x)|\geq U_{thr}.
\end{array}\right.\end{equation}
To develop the Two-Scale Numerical Method, we use that in \cite{FaFreSeN} we proved that under assumptions (\ref{eq2n}) and (\ref{eq8n}) the solution $z^\epsilon$ of (\ref{3.5n}) exists, is unique and moreover asymptotically behaves, as $\epsilon\rightarrow0,$ the way given by the following theorem.
\begin{theorem}\label{thHomSecHomn}
Under assumptions (\ref{eq2n}) and (\ref{eq8n}), for any $T,$ not depending on $\epsilon,$  the sequence $(z^{\epsilon})$ of solutions to (\ref{3.5n}), with coefficients given by (\ref{3.1n}) coupled with (\ref{3.6n}) and (\ref{3.3n}), (\ref{3.7n}) and (\ref{3.7n1}), Two-Scale converges to the profile
$Z\in L^{\infty}([0,T],L^{\infty}_\#(\R,L^2(\torus^{2})))$ solution to
\begin{equation}
\label{ee179n} 
\frac{\partial Z}{\partial\theta}
-\nabla\cdot(\widetilde{\mathcal{A}}\nabla Z)=\nabla \cdot\widetilde{\mathcal{C}},
\end{equation} where $\widetilde{\mathcal{A}}$ and $\widetilde{\mathcal{C}}$ are given by
\begin{equation}\label{EqHomcofn}
\widetilde{\mathcal{A}}(t,\theta,x)=a\,g_a(|\mathcal{U}(t,\theta,x)|) \,\, \textrm{and}\,\,\,\,\widetilde{\mathcal{C}}(t,\theta,x)=c\,g_c(|\mathcal{U}(t,\theta,x)|)\,\frac{\mathcal{U}(t,\theta,x)}{|\mathcal{U}(t,\theta,x)|}.
\end{equation}
Futhermore,
if the supplementary assumption
\begin{gather}
\label{201308151611}  
U_{thr}=0,
\end{gather}
is done, we have
\begin{equation}\label{EqRajn}\widetilde{\mathcal{A}}(t,\theta,x)\geq \widetilde G_{thr}\,\,\textrm{for any}\,\,t,\theta,x
\in [0,T]\times\R\times\torus^2,
\end{equation}
and, defining $Z^{\epsilon}=Z^{\epsilon}(t,x)=Z(t,\frac{t}{\epsilon},x)$, the following estimate holds for $z^{\epsilon}-Z^{\epsilon}$
\begin{equation}\Big\|\frac{z^{\epsilon}-Z^{\epsilon}}{\epsilon}\Big\|_{  L^{\infty}([0,T), L^{2}(\torus^{2}))}\leq\alpha,
\end{equation} where $\alpha$ is a constant not depending on $\epsilon.$
\end{theorem}

Because of assumptions (\ref{eq2n}) and (\ref{eq8n}),
\begin{gather}
\label{totitotin}
\widetilde{\mathcal{A}},\,\, \widetilde{\mathcal{C}},\,\,
\widetilde{\mathcal{A}}_1,\,\, \widetilde{\mathcal{C}}_1,\,\,
 \widetilde{\mathcal{A}}^\epsilon, \,\,
\widetilde{\mathcal{A}}_1^\epsilon,\,\, \widetilde{\mathcal{C}}^\epsilon, \text{ and }
\widetilde{\mathcal{C}}_1^\epsilon
\text{ are regular and bounded. }
\end{gather}%
%

\section{Two-Scale Numerical Method Building}
\label{SecTSNMB}
In this section, we develop the Two-Scale Numerical Method in order to approach the solution $z^\epsilon$ of (\ref{3.5n}).  The idea is to get a good approximation of $z^\epsilon(t,x)$ seeing Theorem\,\ref{thHomSecHomn} content as $z^\epsilon(t,x)\sim Z(t,\frac{t}{\epsilon},x).$

The strategy is to consider a Fourier expansion of $Z$ solution to (\ref{ee179n}). In this equation, $t$ is only a parameter.

The Fourier expansion of $Z$ is given as follows:


\begin{equation}
\label{F01}
 Z(t,\theta,x)=\sum_{l,m,n}Z_{l,m,n}(t)\,\,e^{2i\pi(l\theta+mx_1+nx_2)},
\end{equation}
where $Z_{l,m,n}(t),\,\,l=0,1,2,\ldots$, $m=0,1,2,\ldots$, $n=0,1,2,\ldots,$ are the unknown complex coefficients of the Fourier expansion of $Z.$
Using (\ref{F01}), the Fourier expansion of $\frac{\partial Z}{\partial \theta}$ is given by
\begin{equation}\label{F011}
 \frac{\partial Z}{\partial \theta}(t,\theta,x)=\sum_{l,m,n}2i\pi\,l\,Z_{l,m,n}(t)\,\,e^{2i\pi(l\theta+mx_1+nx_2)}.
\end{equation}
To obtain the system satisfied by the Fourier expansion (\ref{F01}) of $Z$, it is necessary to compute the Fourier expansions of
$\nabla\cdot(\widetilde{\mathcal{A}}\nabla Z)$ and $\nabla\cdot\widetilde{\mathcal{C}}.$
As  $\nabla\cdot(\widetilde{\mathcal{A}}\nabla Z)= \nabla\widetilde{\mathcal{A}}\cdot\nabla Z+\widetilde{\mathcal{A}}\cdot\Delta Z,$
let
\begin{equation}\label{F012}
\sum_{l,m,n}\widetilde{\mathcal{A}}_{l,m,n}(t)\,e^{2i\pi(l\theta+mx_1+nx_2)},
\end{equation}
and
\begin{equation}\label{F013}
\sum_{l,m,n}\widetilde{\mathcal{A}}_{l,m,n}^{grad}(t)\,e^{2i\pi(l\theta+mx_1+nx_2)},
\end{equation}
 be respectively the Fourier expansions of $\widetilde{\mathcal{A}}\,\,\text{and}\,\,\nabla\widetilde{\mathcal{A}},$
  where $\widetilde{\mathcal{A}}_{l,m,n}^{grad}(t)=2i\pi\widetilde{\mathcal{A}}_{l,m,n}\left(\begin{array}{ccc} m\\
 n\end{array}\right)$
and  then the Fourier expansions of $\nabla Z$ and $\Delta Z$ are respectively given by
\begin{equation}\label{F014}
\sum_{l,m,n}2i\pi\left(\begin{array}{ccc}m\\
n\end{array}\right)Z_{l,m,n}(t)\,e^{2i\pi(l\theta+mx_1+nx_2)},
\end{equation}
and
\begin{equation}\label{F015}
-\sum_{l,m,n}4\pi^2(m^2+n^2)Z_{l,m,n}(t)\,e^{2i\pi(l\theta+mx_1+nx_2)}.
\end{equation}
In the same way the Fourier expansion of  $\nabla\cdot\widetilde{\mathcal{C}}$ is given by
\begin{equation}
\label{201308150932}  
\sum_{l,m,n}\widetilde{\mathcal{C}}_{l,m,n}e^{2i\pi(l\theta+mx_1+nx_2)}.
\end{equation}
Using (\ref{F01}), (\ref{F011}), (\ref{F012}), (\ref{F013}), (\ref{F014}), (\ref{F015}) and (\ref{201308150932}), equation (\ref{ee179n}) becomes
\begin{equation}
\sum_{l,m,n}2i\pi\,l\,Z_{l,m,n}(t)\,\,e^{2i\pi(l\theta+mx_1+nx_2)} $$
$$-\Big(\sum_{l,m,n}\widetilde{\mathcal{A}}_{l,m,n}^{grad}(t)\,e^{2i\pi(l\theta+mx_1+nx_2)}\Big)\cdot\Big(\sum_{l,m,n}2i\pi\left(\begin{array}{ccc}m\\
n\end{array}\right)Z_{l,m,n}(t)\,e^{2i\pi(l\theta+mx_1+nx_2)}\Big)$$
$$+\Big(\sum_{l,m,n}\widetilde{\mathcal{A}}_{l,m,n}(t)\,e^{2i\pi(l\theta+mx_1+nx_2)}\Big)\Big( \sum_{l,m,n}4\pi^2(m^2+n^2)Z_{l,m,n}(t)\,e^{2i\pi(l\theta+mx_1+nx_2)}\Big)=$$
$$\sum_{l,m,n}\widetilde{\mathcal{C}}_{l,m,n}(t)\,e^{2i\pi(l\theta+mx_1+nx_2)},
\end{equation}
which gives after identification, the following algebraic system for $(Z_{l,m,n})$:
\begin{eqnarray}\label{Fa02}
{2i\pi\,l\,Z_{l,m,n}(t)-\sum_{i,j,k}2i\pi\widetilde{\mathcal{A}}_{i,j,k}^{grad}(t)\cdot\left(\begin{array}{ccc}m-j\\
n-k\end{array}\right)Z_{l-i,m-j,n-k}(t) {} }~~~~~~~~
\nonumber\\
+4\pi^2\sum_{i,j,k}\widetilde{\mathcal{A}}_{i,j,k}(t)((m-j)^2+(n-k)^2)Z_{l-i,m-j,n-k}(t)= \widetilde{\mathcal{C}}_{l,m,n}(t).
\end{eqnarray}
In formula (\ref{F01}), the integers $m, n$ and $l$ vary from $-\infty$ to $+\infty.$ But in practice, we will consider the truncated
Fourier series of order $P\in\mathbb N$ defined by
\begin{equation}\label{Fa01}  
 Z_P(t,\theta,x)=\sum_{0\leq l\leq P,0\leq m\leq P,0\leq n\leq P}Z_{l,m,n}(t)\,\,e^{2i\pi(l\theta+mx_1+nx_2)}.
\end{equation}
Using (\ref{Fa01}), formula (\ref{Fa02}) becomes:
\begin{eqnarray}
{2i\pi\,l\,Z_{l,m,n}(t)-\sum_{0\leq i\leq P,\,1\leq j\leq P,\,\,0\leq k\leq P}2i\pi\widetilde{\mathcal{A}}_{i,j,k}^{grad}(t)\cdot\left(\begin{array}{ccc}m-j\\
n-k\end{array}\right)Z_{l-i,m-j,n-k}(t) {} }~~~~~~~~
\nonumber\\
+4\pi^2\sum_{0\leq i\leq P,\,\,0\leq j\leq P,\,\,0\leq k\leq P}\widetilde{\mathcal{A}}_{i,j,k}(t)((m-j)^2+(n-k)^2)Z_{l-i,m-j,n-k}(t)= \widetilde{\mathcal{C}}_{l,m,n}(t).
\end{eqnarray}
\section{Convergence result}
\begin{proof} of Theorem\,\ref{thHomSecHomn}. For self-containedness, we recall the proof of Theorem\,\ref{thHomSecHomn}. Firstly, we obtain an estimate leading to that $z^\epsilon$ is bounded in $L^\infty((0,T);L^2(\torus^2)).$ Secondly,
defining test function $\psi^{\epsilon}(t,x)=\psi(t,\frac{t}{\epsilon},x)$ for any $\psi(t,\theta,x)$,
regular with a compact support over $[0,T)\times\torus^{2}$ and 1-periodic in $\theta,$ 
multiplying (\ref{3.5n}) by $\psi^{\epsilon}$  and integrating over $[0,T)\times\torus^{2}$ gives
\begin{equation}\label{H3}\int_{\torus^{2}}\int_{0}^{T}\frac{\partial z^{\epsilon}}{\partial t}\psi^{\epsilon}dtdx-
    \frac{1}{\epsilon}\int_{\torus^{2}}\int_{0}^{T}\nabla\cdot(\mathcal{A}^{\epsilon}\nabla z^{\epsilon})\psi^{\epsilon}dtdx=\frac{1}{\epsilon}\int_{\torus^{2}}\int_{0}^{T}\nabla\cdot\mathcal{C}^{\epsilon}\psi^{\epsilon}dtdx.
\end{equation}
Then integrating by parts in the first integral over $[0,T)$ and using the Green formula in $\torus^{2}$ in the second integral  we have
\begin{eqnarray}\label{H4}
    \lefteqn{\hspace{0.80cm}-\int_{\torus^{2}}z_{0}(x)\psi(0,0,x)dx-\int_{\torus^{2}}\int_{0}^{T}\frac{\partial\psi^{\epsilon}}{\partial t}z^{\epsilon}dtdx {} }
                       \nonumber\\
    & & {}+\frac{1}{\epsilon}\int_{\torus^{2}}\int_{0}^{T}\mathcal{A}^{\epsilon}\nabla z^{\epsilon}\nabla\psi^{\epsilon}dtdx=\frac{1}{\epsilon}\int_{\torus^{2}}\int_{0}^{T}
    \nabla\cdot\mathcal{C}^{\epsilon}\psi^{\epsilon}dtdx.
\end{eqnarray}
Again using the Green formula in the third integral we obtain
\begin{eqnarray}\label{H4.111}
    \lefteqn{\hspace{0.80cm}-\int_{\torus^{2}}z_{0}(x)\psi(0,0,x)\,dx-\int_{\torus^{2}}\int_{0}^{T}\frac{\partial\psi^{\epsilon}}{\partial t}z^{\epsilon}dtdx {} }
                       \nonumber\\
    & & {}-\frac{1}{\epsilon}\int_{\torus^{2}}\int_{0}^{T}z^{\epsilon}\,\nabla\cdot(\mathcal{A}^{\epsilon}\nabla\psi^{\epsilon})\,dtdx=
    \frac{1}{\epsilon}\int_{\torus^{2}}\int_{0}^{T}
    \nabla\cdot\mathcal{C}^{\epsilon}\psi^{\epsilon}dtdx.
\end{eqnarray}
But
\begin{equation}\label{H5}
 \frac{\partial \psi^{\epsilon}}{\partial t}
 =\left(\frac{\partial \psi}{\partial t}\right)^{\epsilon}
 +\frac{1}{\epsilon}\left(\frac{\partial\psi}{\partial \theta}\right)^{\epsilon},
\end{equation}
where
\begin{gather}\label{H4a}
   \left(\frac{\partial\psi}{\partial t}\right)^{\epsilon}(t,x)
   =\frac{\partial\psi}{\partial t}(t,\frac{t}{\epsilon},x)\,\,\textrm{ and }
     \left(\frac{\partial\psi}{\partial\theta}\right)^{\epsilon}(t,x)=\frac{\partial\psi}{\partial\theta}(t,\frac{t}{\epsilon},x),
\end{gather}
then we have
\begin{eqnarray}\label{H6s}
    \lefteqn{\hspace{0.80cm} \int_{\torus^{2}}\int_{0}^{T}z^{\epsilon}\left( \left(\frac{\partial \psi}{\partial t}\right)^{\epsilon}+\frac{1}{\epsilon}\left(\frac{\partial\psi}{\partial \theta}\right)^{\epsilon}+\frac{1}{\epsilon}\nabla\cdot(\mathcal{A}^{\epsilon}\nabla\psi^{\epsilon})\right)dxdt {}}~~~~~~~~
    \nonumber\\
    & &+
\frac{1}{\epsilon}\int_{\torus^{2}}\int_{0}^{T}
    \nabla\cdot\mathcal{C}^{\epsilon}\psi^{\epsilon}dtdx=-\int_{\torus^{2}}z_{0}(x)\psi(0,0,x)\,dx.
    \end{eqnarray}
\\
Using the Two-Scale convergence due to Nguetseng \cite{nguetseng:1989} and
Allaire \cite{allaire:1992} (see also Fr\'enod Raviart and Sonnendr\"{u}cker \cite{FRS:1999}),  since $z^\epsilon$ is bounded in $L^\infty([0,T),L^{2}(\torus^{2})),$  there exists a profile
$Z(t,\theta,x)$, periodic of period 1 with respect to $\theta$, such that for all
$\psi(t,\theta,x),$ regular with a compact support with respect to $(t,x)$ and  1-periodic with respect to $\theta$, we have
 \begin{equation}\label{H7}
    \int_{\torus^{2}}\int_{0}^{T}z^{\epsilon}\psi^{\epsilon}dtdx
    \longrightarrow\int_{\torus^{2}}\int_{0}^{T}\int_{0}^{1}Z\psi \,d\theta dtdx,\,\,\text{as}\,\,\epsilon\,\,\text{tends to zero},
 \end{equation}
 for a subsequence extracted from $(z^{\epsilon})$.
 \\
Multiplying (\ref{H6s}) by $\epsilon,$ passing to the limit as $\epsilon\rightarrow 0$ and using (\ref{H7})
we have
    \begin{equation}\label{H8}
\int_{\torus^{2}}\int_{0}^{T}\int_{0}^{1}Z\frac{\partial \psi}{\partial\theta}\,d\theta dtdx+\lim_{\epsilon\rightarrow0}\int_{\torus^{2}}\int_{0}^{T}z^{\epsilon}\nabla\cdot(\mathcal{A}^{\epsilon}\nabla\psi^{\epsilon})\,dtdx
 =\lim_{\epsilon\rightarrow 0}\int_{\torus^{2}}\int_{0}^{T}\mathcal{C}^{\epsilon}\cdot\nabla\psi^{\epsilon} dtdx,
    \end{equation}
    for an extracted subsequence. As $\mathcal{A}^{\epsilon}$ and $\mathcal{C}^{\epsilon}$ are bounded   and $\psi^{\epsilon}$ is a regular function, $\mathcal{A}^{\epsilon}\nabla\psi^{\epsilon}$ and $\nabla\psi^{\epsilon}$ can be considered as test functions.
 Using (\ref{H7}) we have
\begin{equation}
   \int_{\torus^{2}}\int_{0}^{T}z^{\epsilon}\,\nabla\cdot(\mathcal{A}^{\epsilon}\nabla\psi^{\epsilon})dtdx\longrightarrow
    \int_{\torus^{2}}\int_{0}^{T}\int_{0}^{1}Z\nabla\cdot(\widetilde{\mathcal{A}}\nabla\psi)\,d\theta dtdx,
\end{equation}
and
\begin{equation}\label{H9}
\int_{\torus^{2}}\int_{0}^{T}\mathcal{C}^{\epsilon}\cdot\nabla\psi^{\epsilon}dtdx
\,\,\textrm{Two-Scale converges to} \,\,\int_{\torus^{2}}\int_{0}^{T}\int_{0}^{1}\widetilde{\mathcal{C}}\cdot\nabla\psi \,d\theta dtdx.
\end{equation}
Passing to the limit as $\epsilon\rightarrow 0$ we obtain from (\ref{H8}) a weak formulation of the equation (\ref{ee179n}) satisfied by $Z$.

Using (\ref{3.1n}) and (\ref{3.3n}) equation (\ref{3.5n}) becomes
\begin{equation}\label{eq5.1}\frac{\partial z^{\epsilon}}{\partial t}-\frac{1}{\epsilon}\nabla\cdot(\widetilde{\mathcal{A}}^{\epsilon}\nabla z^{\epsilon})=\frac{1}{\epsilon}\nabla\cdot\widetilde{\mathcal{C}}^{\epsilon}+\nabla\cdot(\widetilde{\mathcal{A}}_{1}^{\epsilon}\nabla z^{\epsilon})+\nabla\cdot\widetilde{\mathcal{C}}_{1}^{\epsilon}.
\end{equation}
For $Z^{\epsilon}$, we have
\begin{equation}\frac{\partial Z^{\epsilon}}{\partial t}=\left(\frac{\partial Z}{\partial t}\right)^{\epsilon}+\frac{1}{\epsilon}\left(\frac{\partial Z}{\partial \theta}\right)^{\epsilon},
\end{equation}
where
\begin{gather}\label{H4acopie}
   \left(\frac{\partial Z}{\partial t}\right)^{\epsilon}(t,x)
   =\frac{\partial Z}{\partial t}(t,\frac{t}{\epsilon},x)\,\,\textrm{ and }
     \left(\frac{\partial Z}{\partial\theta}\right)^{\epsilon}(t,x)=\frac{\partial Z}{\partial\theta}(t,\frac{t}{\epsilon},x).
\end{gather}
Using (\ref{ee179n}), $Z^{\epsilon}$ is solution to
\begin{equation}\label{eq5.2}
\frac{\partial Z^{\epsilon}}{\partial t}-\frac{1}{\epsilon}\nabla\cdot\left(\widetilde{\mathcal{A}}^{\epsilon}\nabla Z^{\epsilon}\right)=\frac{1}{\epsilon}\nabla\cdot\widetilde{\mathcal{C}}^{\epsilon}+\left(\frac{\partial Z}{\partial t}\right)^{\epsilon}.
\end{equation}
Formulas (\ref{eq5.1}) and (\ref{eq5.2}) give
\begin{equation}\label{eq5.3}
\frac{\partial (z^{\epsilon}-Z^{\epsilon})}{\partial t}-\frac{1}{\epsilon}\nabla\cdot\left(\widetilde{\mathcal{A}}^{\epsilon}\nabla (z^{\epsilon}- Z^{\epsilon})\right)=\nabla\cdot\widetilde{\mathcal{C}}_{1}^{\epsilon}+\left(\frac{\partial Z}{\partial t}\right)^{\epsilon}+\nabla\cdot(\widetilde{\mathcal{A}}_{1}^{\epsilon}\nabla z^{\epsilon}).
\end{equation}
Multiplying equation (\ref{eq5.3}) by $\frac{1}{\epsilon}$ and using the fact that $z^{\epsilon}=z^{\epsilon}-Z^{\epsilon}+Z^{\epsilon}$ in the right hand side of equation (\ref{eq5.3})$,\frac{z^{\epsilon}-Z^{\epsilon}}{\epsilon}$ is solution to:
\begin{equation}\label{eq5.4n}\frac{\ds\partial \Big(\frac{z^{\epsilon}-Z^{\epsilon}}{\epsilon}\Big)}{\partial t}-\frac{1}{\epsilon}\nabla\cdot\Big((\widetilde{\mathcal{A}}^{\epsilon}+\epsilon\widetilde{\mathcal{A}}_{1}^{\epsilon})\nabla (\frac{z^{\epsilon}- Z^{\epsilon}}{\epsilon})\Big)
=\frac{1}{\epsilon}\Big(\nabla\cdot\widetilde{\mathcal{C}}_{1}^{\epsilon}+(\frac{\partial Z}{\partial t})^{\epsilon}+\nabla\cdot(\widetilde{\mathcal{A}}_{1}^{\epsilon}\nabla Z^{\epsilon})\Big).
\end{equation}
Our aim here is to prove that $\frac{z^{\epsilon}-Z^{\epsilon}}{\epsilon}$ is bounded by a constant $\alpha$ not depending on $\epsilon.$ For this let us use that
$\widetilde{\mathcal{A}}^{\epsilon},\,\,\widetilde{\mathcal{A}}_{1}^{\epsilon},\,\,
\widetilde{\mathcal{C}}^{\epsilon}\,\,\textrm{and}\,\,\widetilde{\mathcal{C}}_{1}^{\epsilon}$ are regular and bounded coefficients (see (\ref{totitotin})) and that $\widetilde{\mathcal{A}}^{\epsilon}\geq G_{thr}$
(see (\ref{EqRajn})).
Hence, $\nabla\cdot\widetilde{\mathcal{C}}_{1}^{\epsilon}$ is bounded, $\nabla\cdot(\widetilde{\mathcal{A}}_{1}^{\epsilon}\nabla Z^{\epsilon})$ is also bounded. Since $Z^{\epsilon}$ is solution to (\ref{eq5.2}),
$\frac{\partial Z}{\partial t}$ satisfies the following equation
\begin{equation}\label{eq5.5a}\frac{\ds\partial \left(\frac{\partial Z}{\partial t}\right)}{\partial\theta}-\nabla\cdot\left(\widetilde{\mathcal{A}}\nabla\frac{\partial Z}{\partial t}\right)=\frac{\partial\nabla\cdot\widetilde{\mathcal{C}}}{\partial t}+\nabla\cdot\left(\frac{\partial \widetilde{\mathcal{A}}}{\partial t}\nabla Z\right).
\end{equation}
Equation (\ref{eq5.5a}) is linear with regular and bounded coefficients. Using a result of Ladyzenskaja, Solonnikov and Ural'Ceva \cite{LSU}, $\frac{\partial Z}{\partial t}$ is regular
and bounded and so the coefficients of equations (\ref{eq5.4n}) are regular and bounded.
Then, using the same arguments as in the proof of Theorem 1.1 in    \,\cite{FaFreSeN} we obtain that
$\left(\frac{z^{\epsilon}-Z^{\epsilon}}{\epsilon}\right)$  is bounded.\\
To determine the value of the constant $\alpha$, we proceed in the same way as in the proof of Theorem 3.16 of \cite{FaFreSeN}.
Since the coefficients $\Big(\widetilde{\mathcal{A}}^{\epsilon},\,\,\widetilde{\mathcal{A}}_{1}^{\epsilon},\,\,
\widetilde{\mathcal{C}}^{\epsilon}\,\,\textrm{and}\,\,\widetilde{\mathcal{C}}_{1}^{\epsilon},\,\,\nabla\cdot\widetilde{\mathcal{C}}_{1}^{\epsilon},\,\,\nabla\cdot(\widetilde{\mathcal{A}}_{1}^{\epsilon}\nabla Z^{\epsilon}),\,\,\text{and}\,\,\frac{\partial Z}{\partial t}\Big)$ are bounded by constants, let $\beta$ denotes the maximum between all these constants. Then we use the same argument as in the proof of Theorems 1.1 and 3.16 and we get:
\begin{equation}\label{CVR1}
\Big\|\frac{z^{\epsilon}-Z^{\epsilon}}{\epsilon}\Big\|_{L^\infty([0,T),L^2(\torus^2))}\leq \|z_0(\cdot)-Z(0,0,\cdot)\|_2\sqrt{\frac{\beta+\beta^3}{\sqrt{\widetilde{G}_{thr}}}+2\beta}\;T.
\end{equation}
\end{proof}
\begin{theorem}\label{therem3.1}
Let $\epsilon$ be a positive real, $z^\epsilon$ be the solution to (\ref{3.5n}),  $Z_P$ be the truncated Fourier series (defined by (\ref{Fa01})) of $Z$
solution to (\ref{ee179n}) and $Z_P^\epsilon$ defined by $Z_P^\epsilon(t,x)=Z_P(t,\frac t\epsilon,x)$. Then,
under assumptions (\ref{eq2n}), (\ref{eq8n}) and (\ref{201308151611}),
 $z^\epsilon-Z_P^\epsilon$ satisfies the following estimate:
\begin{equation} \label{CVR2}
\|z^\epsilon-Z_P^\epsilon\|_{L^\infty([0,T),L^2(\torus^2))}\leq \epsilon\|z_0(\cdot)-Z(0,0,\cdot)\|_2\sqrt{\frac{\beta+\beta^3}{\sqrt{\widetilde{G}_{thr}}}+2\beta}\;T+f(P),
\end{equation}
where $f$ is a non-negative function of $P$ not depending on $\epsilon$ and satisfying $\lim_{P\rightarrow+\infty}f(P)=0.$
\end{theorem}
\begin{proof} We can write\,:
\begin{eqnarray}\label{CVR3}  
{\|z^\epsilon-Z_P^\epsilon\|_{L^\infty([0,T),L^2(\torus^2))}=
\|z^\epsilon-Z^\epsilon+Z^\epsilon-Z_p^\epsilon\|_{L^\infty([0,T),L^2(\torus^2))}{} }~~~~~~~~
\nonumber\\
\leq
\|z^\epsilon-Z^\epsilon\|_{L^\infty([0,T),L^2(\torus^2))}+\|Z^\epsilon-Z_p^\epsilon\|_{L^\infty([0,T),L^2(\torus^2))}.
\end{eqnarray}
Using (\ref{CVR1}), the first term in the right hand side of (\ref{CVR3}) is bounded by
\begin{equation}\label{CVR4}
\|z^\epsilon-Z^\epsilon\|_{L^\infty([0,T),L^2(\torus^2))}\leq\epsilon\|z_0(\cdot)-Z(0,0,\cdot)\|_2\sqrt{
\frac{\beta+\beta^3}{\sqrt{\widetilde{G}_{thr}}}+2\beta}T.
\end{equation}
For the second term of (\ref{CVR3}), using classical results of Fourier series theory, since $Z-Z_P$  is nothing but the rest of the
Fourier series of order $P$ of $Z$ and since $Z$ is regular (because it is the solution of (\ref{ee179n}) which has regular coefficients), the non-negative function  $f$ satisfying $\lim_{P\rightarrow+\infty}f(P)=0$ such that
\begin{gather}
\label{201308151641}  
\|Z-Z_p\|_{L^{\infty}([0,T],L^{\infty}_\#(\R,L^2(\torus^{2})))} \leq f(P),
\end{gather}
exists. From this last inequality,
\begin{equation}\label{CVR5}
\|Z^\epsilon-Z_p^\epsilon\|_{L^\infty([0,T),L^2(\torus^2))}\leq f(P),
\end{equation}
follows
and coupling this with (\ref{CVR4}) and (\ref{CVR3})  gives inequality (\ref{CVR2}).
\end{proof}
\section{Numerical illustration of Theorem\,\ref{therem3.1}.}
\subsection{Reference solution}\label{grand}
Having Fourier coefficients of $Z$ on hand, we will do the same for function $z^\epsilon(t,x)$ solution to (\ref{3.5n}) in order to compare it to the profile $Z$ for a given $\epsilon,$ in a fixed time.
The Fourier expansion of $z^{\epsilon}$ is given by
\begin{equation}\label{eqalgeb0}
z^{\epsilon}(t,x_1,x_2)=\sum_{m,n}z_{m,n}(t)\,\,e^{2\pi i(mx_1+nx_2)},
\end{equation} where $m=0,1,2,\dots$ and $n=0,1,2,\ldots,$
then the Fourier expansion of $\frac{\partial z^\epsilon}{\partial t}$ is
\begin{equation}
\frac{\partial z^\epsilon}{\partial t}=\sum_{m,n}\dot{z}_{m,n}(t)\,\,e^{2\pi i(mx_1+nx_2)}.
\end{equation}
Using the same idea as in the Fourier expansion of $Z,$ we obtain the following infinite system of Ordinary Differential Equations
\begin{eqnarray}\label{geb0}
{\frac{\partial z_{m,n}}{\partial t}(t)-\frac{1}{\epsilon}\sum_{i,j}2i\pi\mathcal{A}_{i,j}^{grad}(t)\cdot\left(\begin{array}{ccc}m-i\\
n-j\end{array}\right)z_{m-i,n-j}(t){} }~~~~~~~~
\nonumber\\
\hspace{1,5cm}+\frac{1}{\epsilon}4\pi^2\sum_{i,j}\mathcal{A}_{i,j}(t)((m-i)^2+(n-j)^2)z_{m-i,n-j}(t)=\frac{1}{\epsilon}\mathcal{C}_{m,n}(t),
\end{eqnarray}
where $\mathcal{A}_{i,j}^{grad}(t),\,\, \mathcal{A}_{i,j}(t)\,\,\text{and}\,\, \mathcal{C}_{m,n}(t)$ are respectively the Fourier coefficients of $\nabla\mathcal{A}^\epsilon,\,\,\mathcal{A}^\epsilon$ and $\nabla\cdot\mathcal{C}^\epsilon.$\\
In the same way, the truncated Fourier series of order $P\in \N$ of $z^\epsilon$ is  given by
\begin{equation}
z^{\epsilon}_{P}(t,x_1,x_2)=\sum_{m,n=0}^{P}z_{m,n}(t)\,\,e^{2\pi i(mx_1+nx_2)},
\end{equation}
which gives from (\ref{geb0}) the following system Ordinary Differential Equations
\begin{eqnarray}\label{algeb}
{\frac{\partial z_{m,n}}{\partial t}(t)-\frac{1}{\epsilon}\sum_{i,j=0}^{P}2i\pi\mathcal{A}_{i,j}^{grad}(t)\cdot\left(\begin{array}{ccc}m-i\\
n-j\end{array}\right)z_{m-i,n-j}(t){} }~~~~~~~~
\nonumber\\
+\frac{1}{\epsilon}4\pi^2\sum_{i,j=0}^{P}\mathcal{A}_{i,j}(t)((m-i)^2+(n-j)^2)z_{m-i,n-j}(t)=\frac{1}{\epsilon}\mathcal{C}_{m,n}(t).
\end{eqnarray}
In (\ref{algeb}), we will use an initial condition $z_{m,n}(0,x).$ To solve (\ref{algeb}) we use, for the discretization in time, a Runge-Kutta method (ode45).
\subsection{Comparison Two-Scale Numerical Solution and reference solution}
In this paragraph, we consider the truncated solution $z^\epsilon_{P}(t,x_1,x_2)$ and $Z_P(t,\frac{t}{\epsilon},x_1,x_2).$
The objective here is to compare for a fixed $\epsilon$ and a given time, the quantity $|z^\epsilon_{P}(t,x_1,x_2)-Z_P(t,\frac{t}{\epsilon},x_1,x_2)|$ when the water velocity $\mathcal U$ is given.

\subsubsection{Comparisons of $z_{P}^\epsilon(t,x)$ and $Z_P(t,\frac{t}{\epsilon},x)$ with  $\mathcal{U}$ given by (\ref{wat0}).\label{sectionavant}}
For the numerical simulations, concerning $z^{\epsilon},\,$ we take $z_0(x_1,x_2)=\,\cos2\pi x_1 +\,\,\cos4\pi x_1 $  and $z_0(x_1,x_2)=Z(0,0,x_1,x_2)$.
In what concerns the water velocity field, we consider the function
\begin{equation}\label{wat0}
\mathcal{U}(t,\theta,x_1,x_2)=\sin\pi x_1\sin2\pi\theta\,\textbf e_1,
\end{equation}
where $\textbf e_1$ and $\textbf{e}_2$ are respectively
the first and the second vector
of the canonical basis of $\mathbb{R}^{2}$ and $x_1,\,\,x_2$ are the first and the second components of $x.$\\
In  Figure \ref{UA2x0} , we can see the space distribution of the first component of the velocity $\mathcal U$ for a given time $t=1$ and for various
values of $\theta$: $0.3$, $0.55$, and $0.7$.
\begin{figure}[htbp]
\begin{center}
\includegraphics[angle=0, width=5cm]{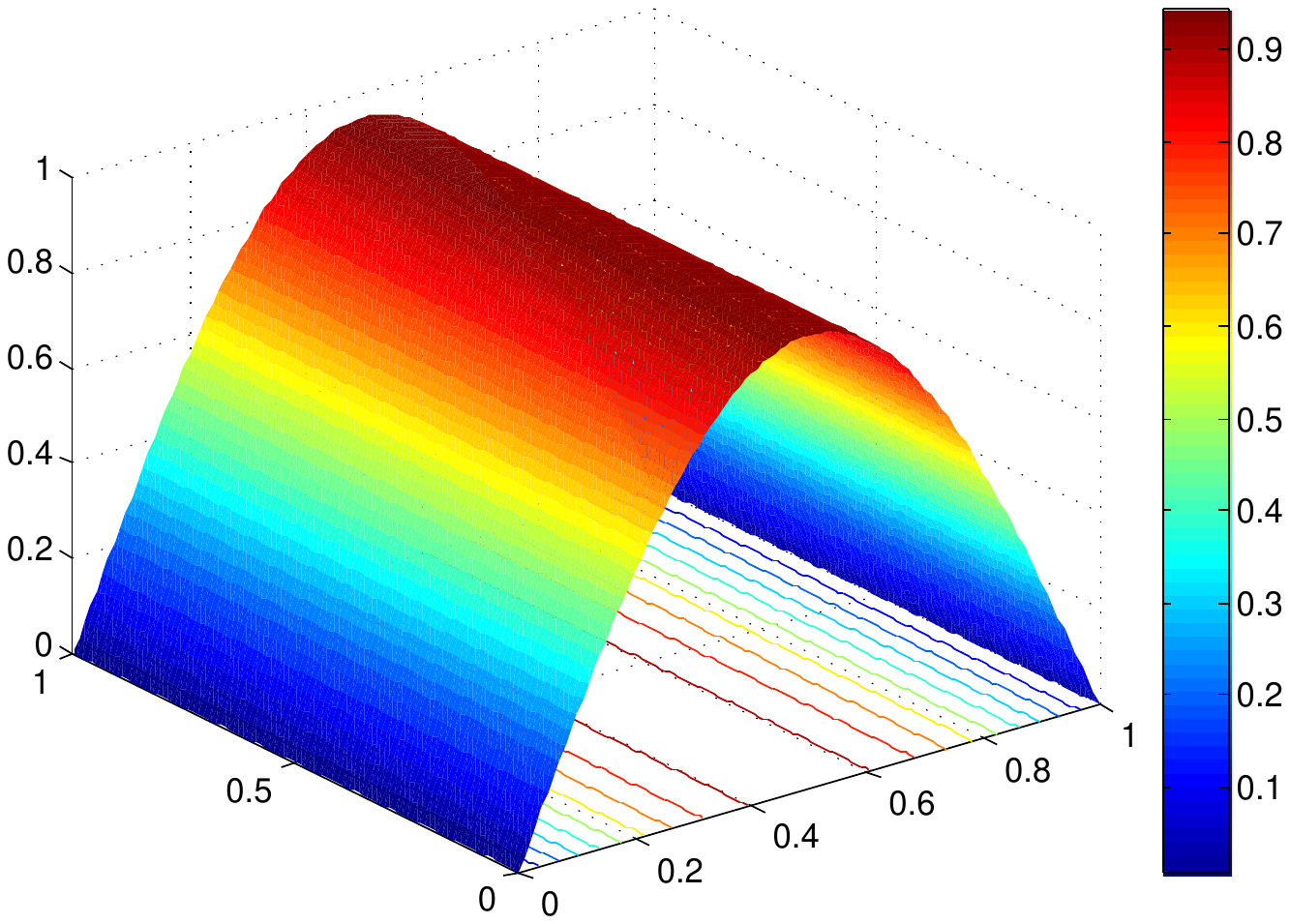}\includegraphics[angle=0, width=5cm]{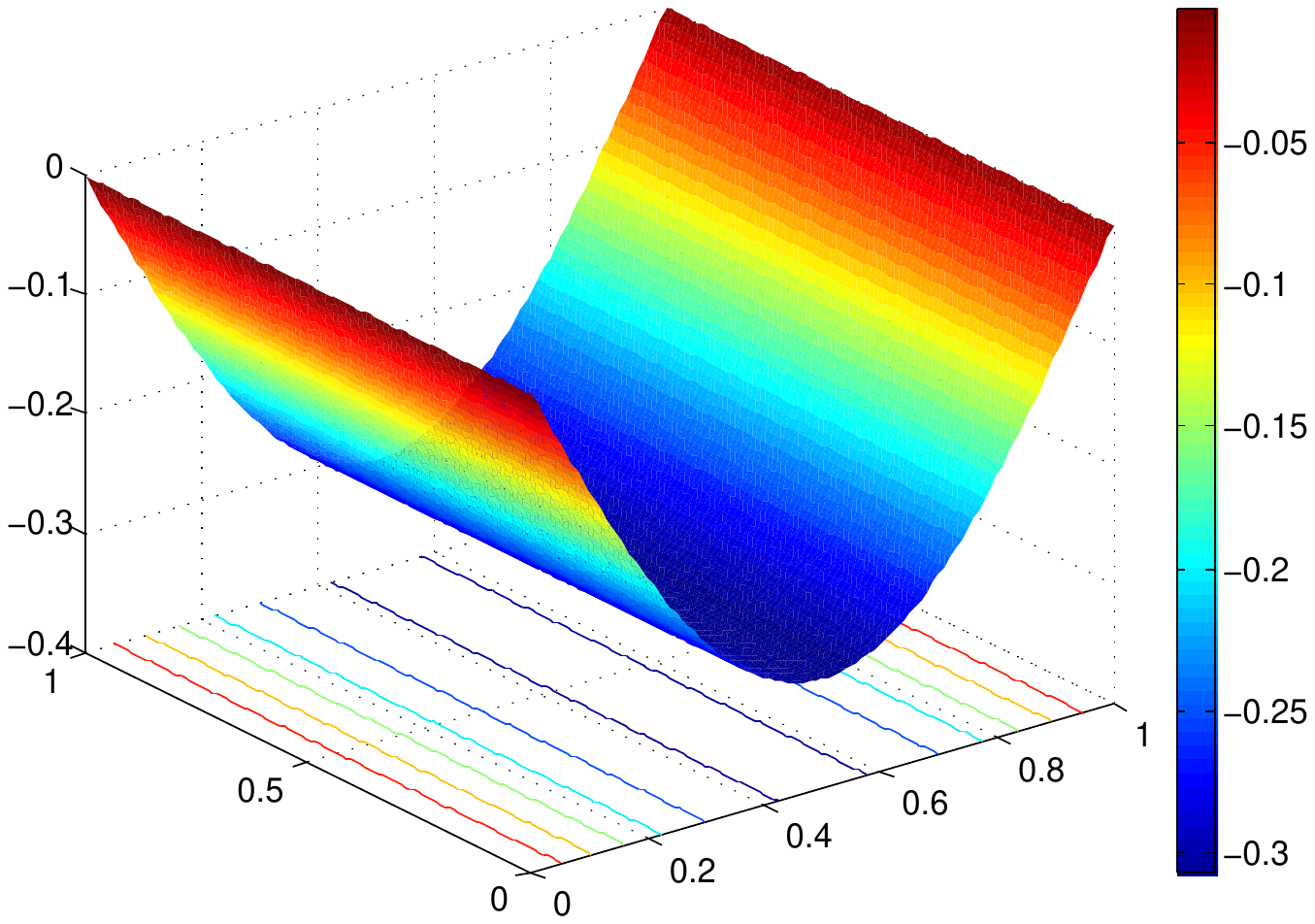}\\
\includegraphics[angle=0, width=5cm]{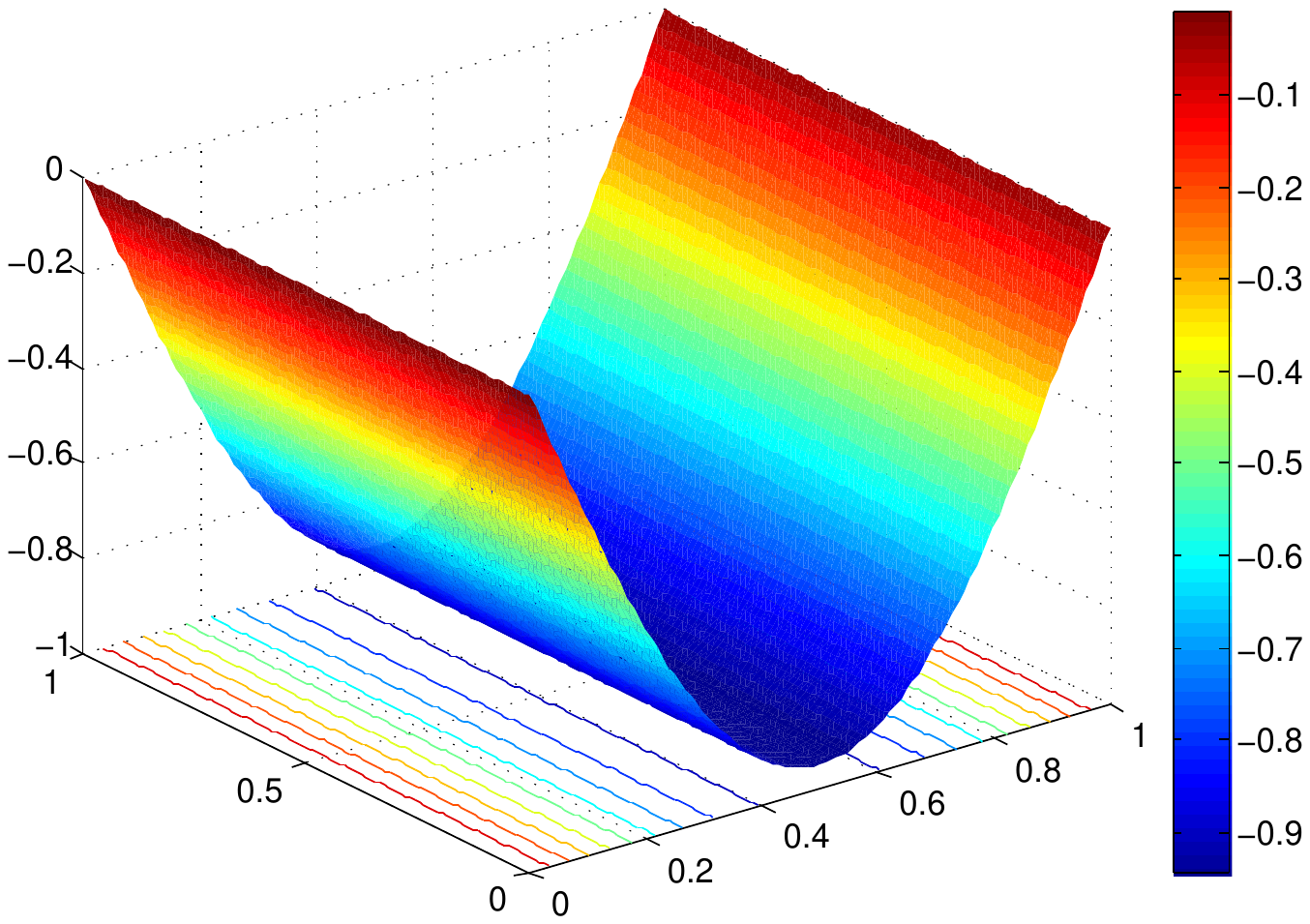}
\end{center}
\caption{Space distribution of the first component of $\mathcal{U}(1, 0.3,(x_1,x_2))$, $\mathcal{U}(1, 0.55,(x_1,x_2))$
and $\mathcal{U}(1, 0.7,(x_1,x_2))$ when $\mathcal{U}$ is given by (\ref{wat0}).} \label{UA2x0}
\end{figure}
In Figure\,\ref{gra}, we see, for a fixed point $x=(x_1,x_2)$, how the water velocity $\widetilde{\mathcal U}(\theta)$ evolves with respect to $\theta.$
\begin{figure}[htbp]
\begin{center}
\includegraphics[angle=0,width=8cm]{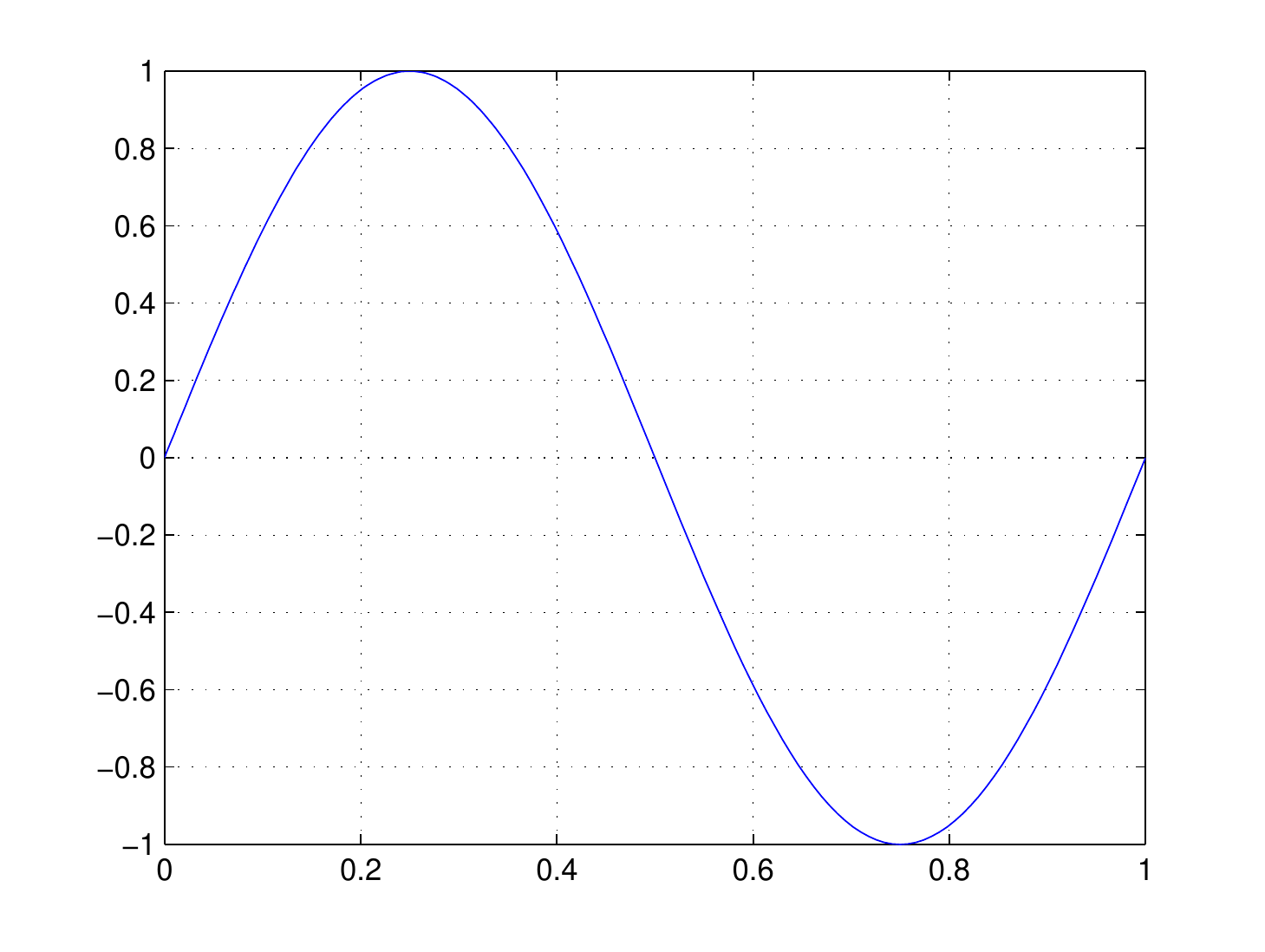}\includegraphics[angle=0, width=8cm]{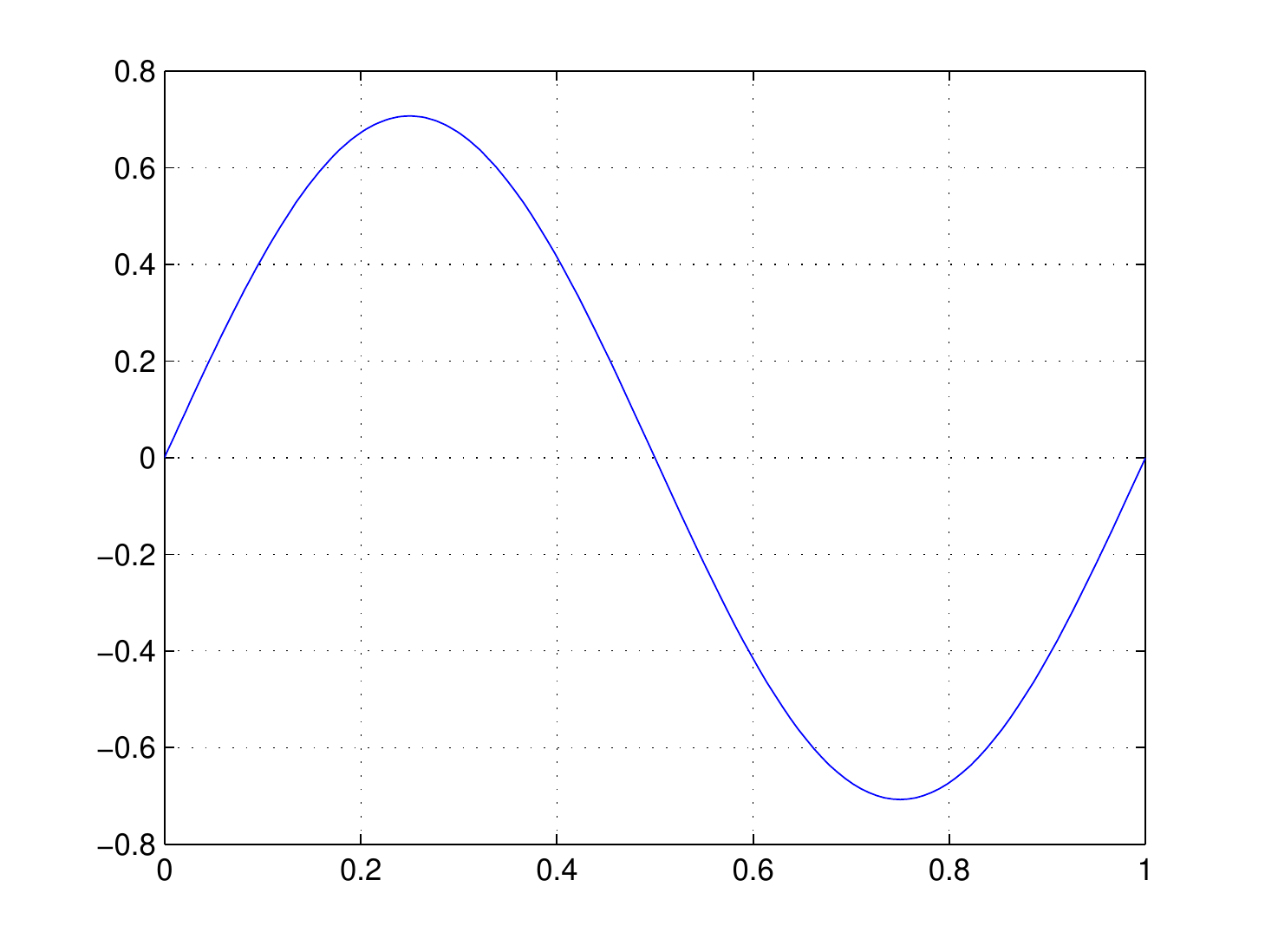}\\
\caption{$\theta$-evolution of  $\widetilde{\mathcal{U}}(\theta,(1/2,0))$ and $\widetilde{\mathcal{U}}(\theta,(1/4,0))$ when
$\mathcal{U}$ is given by (\ref{wat0})}\label{gra}
\end{center}
\end{figure}
%
\noindent In Figure\,\ref{grap}, the $\theta$-evolution of $\widetilde{\mathcal A}(\theta)$ is also given in various points
$(x_1,x_2)\in\mathbb R^2.$
\begin{figure}[htbp]
\includegraphics[angle=0, width=8cm]{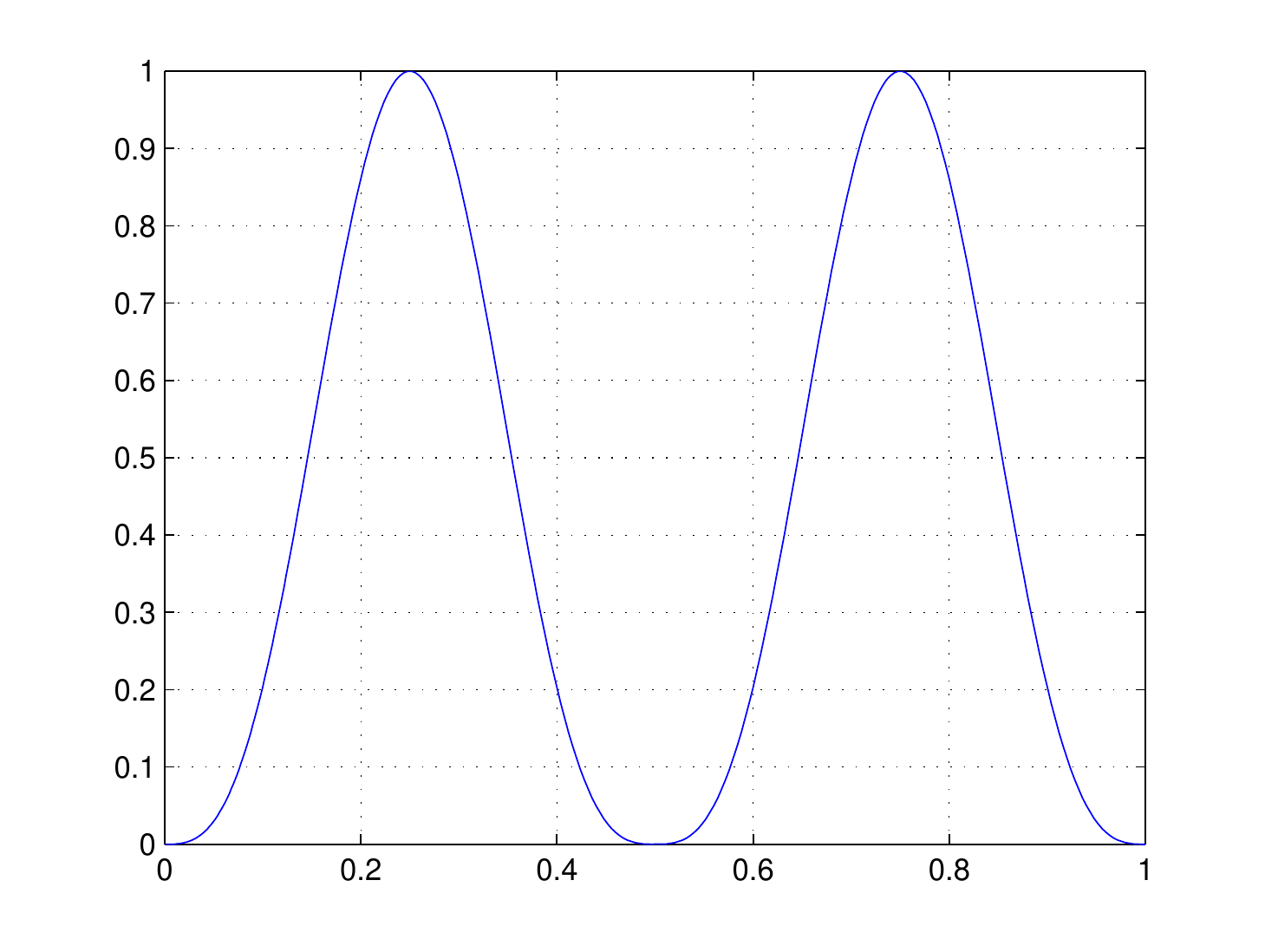},\includegraphics[angle=0, width=8cm]{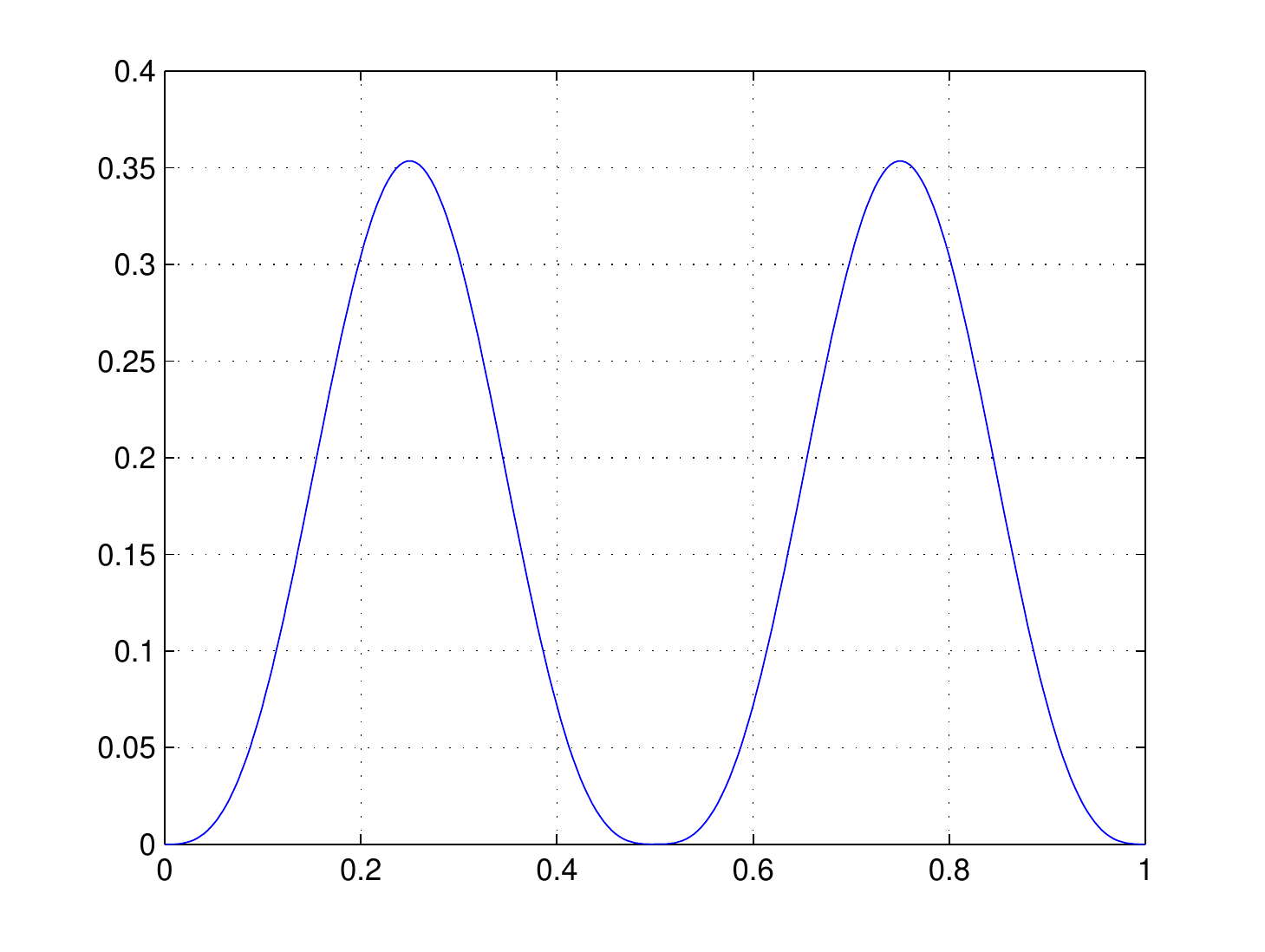}\\
\caption{$\theta$-evolution of  $\widetilde{\mathcal{A}}(\theta,(1/2,0))$ and $\widetilde{\mathcal{A}}(\theta,(1/4,0))$ when
$\mathcal{U}$ is given by (\ref{wat0})}\label{grap}
\end{figure}
\newpage
\noindent In this paragraph, we present numerical simulations in order to validate the Two-Scale convergence presented in Theorem \ref{thHomSecHomn}. For a given $\epsilon,$ we compare $Z_P(t,\frac{t}{\epsilon},x),$ where $Z_P$ is the Fourier expansion of order $P$ of the solution to (\ref{ee179n}) and $z_{P}^\epsilon(t,x)$ the Fourier expansion of order $P$ of the solution to the reference problem. The simulations presented are given for $P=4$.
\noindent The calculation of $z^\epsilon_{P}(t,x)$ implies knowledge of $z_0(x).$ For an initial condition $z_0(x)$
well prepared  and equal to $Z(0,0, x)$, we obtain the results of Figure \ref{U11unistra} and  we remark that the results obtained are the same for
$z_{P}^\epsilon(t,x)$ and $Z_P(t,\frac{t}{\epsilon},x).$\\
\begin{figure}[htbp]
  \includegraphics[angle=0, width=8cm]{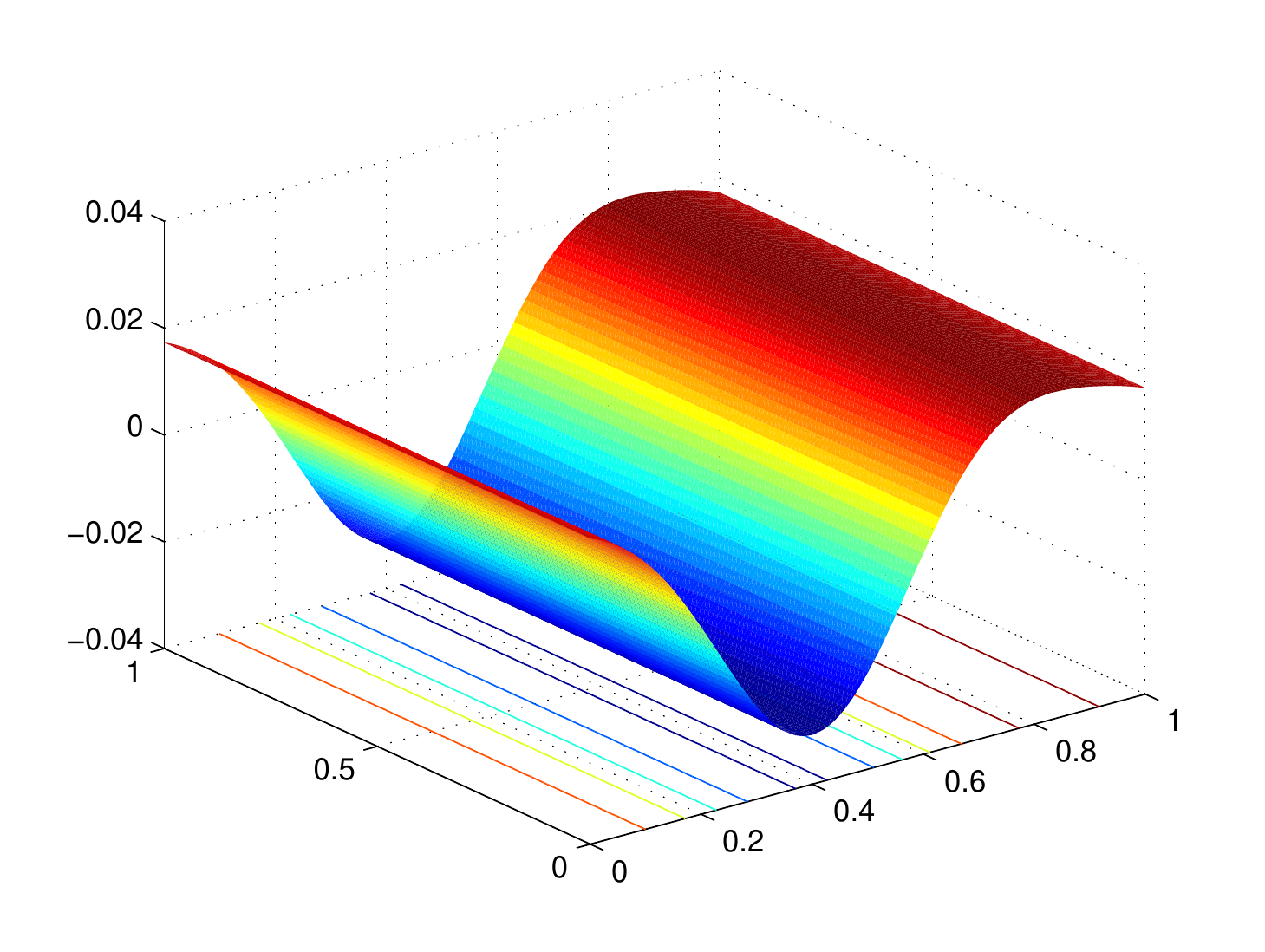}\includegraphics[angle=0, width=8cm]{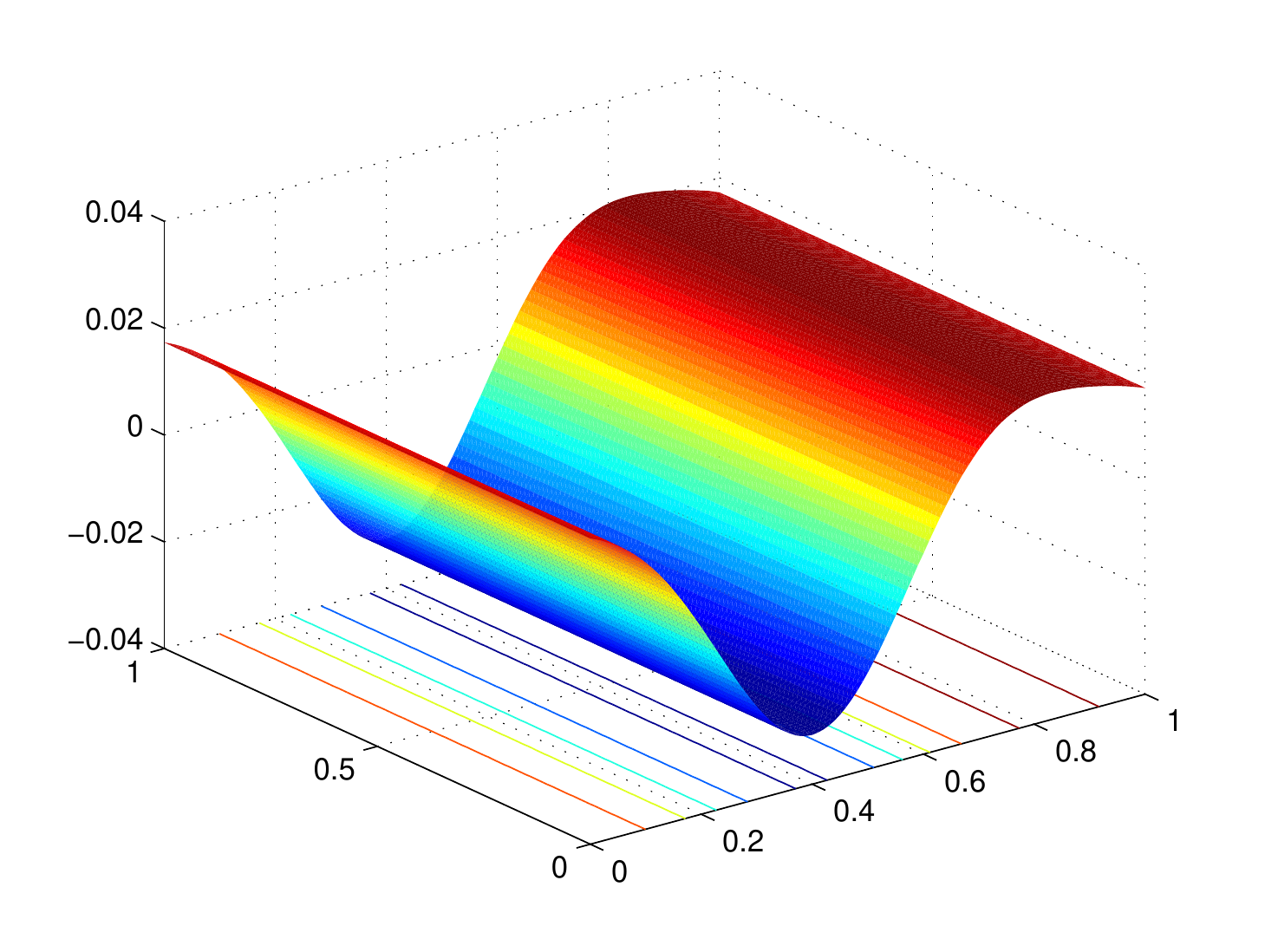}\\
  \caption{Comparison of $z_P^{\epsilon}(t,\cdot)$ and $Z_P(t,\frac{t}{\epsilon},\cdot)$, $P=4$, at time $t=1,\,\,\epsilon=0.001,\,\text{when}\,\,\mathcal{U}$ is given by (\ref{wat0}) and when $z_0(\cdot)=Z(0,0,\cdot).$ On the left $z_P^\epsilon(t,\cdot),$ on the right $Z_P(t,\frac{t}{\epsilon},\cdot).$}\label{U11unistra}
\end{figure}

\newpage
In practice, the solution $Z_P,\,P\in\mathbb N$ evolves according to $P.$ For the simulations, we made the value of the integer $P$ vary and
we saw that this variation is very low from $P\geq6.$

\noindent To better show that $Z_P(t,\frac{t}{\epsilon},x_1,x_2)$ is close to the reference solution $z_{P}^\epsilon(t,x_1,x_2),$ we plot and  compare  $Z_P(t,\frac{t}{\epsilon},x_1,0)$  and $z_{P}^\epsilon(t,x_1,0),$ at different times $t.$ In these comparisons the initial condition $z_0(x_1,x_2)=\cos2\pi x_1+\cos 4\pi x_1$ is different from  $Z(0,0,x_1,x_2)$. The results are shown in Figure\,\ref{compare1} and Figure\,\ref{compare2}.
We see in these figures that the solution $z_{P}^\epsilon(t,x)$ get closer and closer to $Z_P(t,\frac t \epsilon,x)$ with time of order
$\epsilon.$\\
\begin{figure}[htbp]
\begin{center}
  \includegraphics[angle=0, width=5cm]{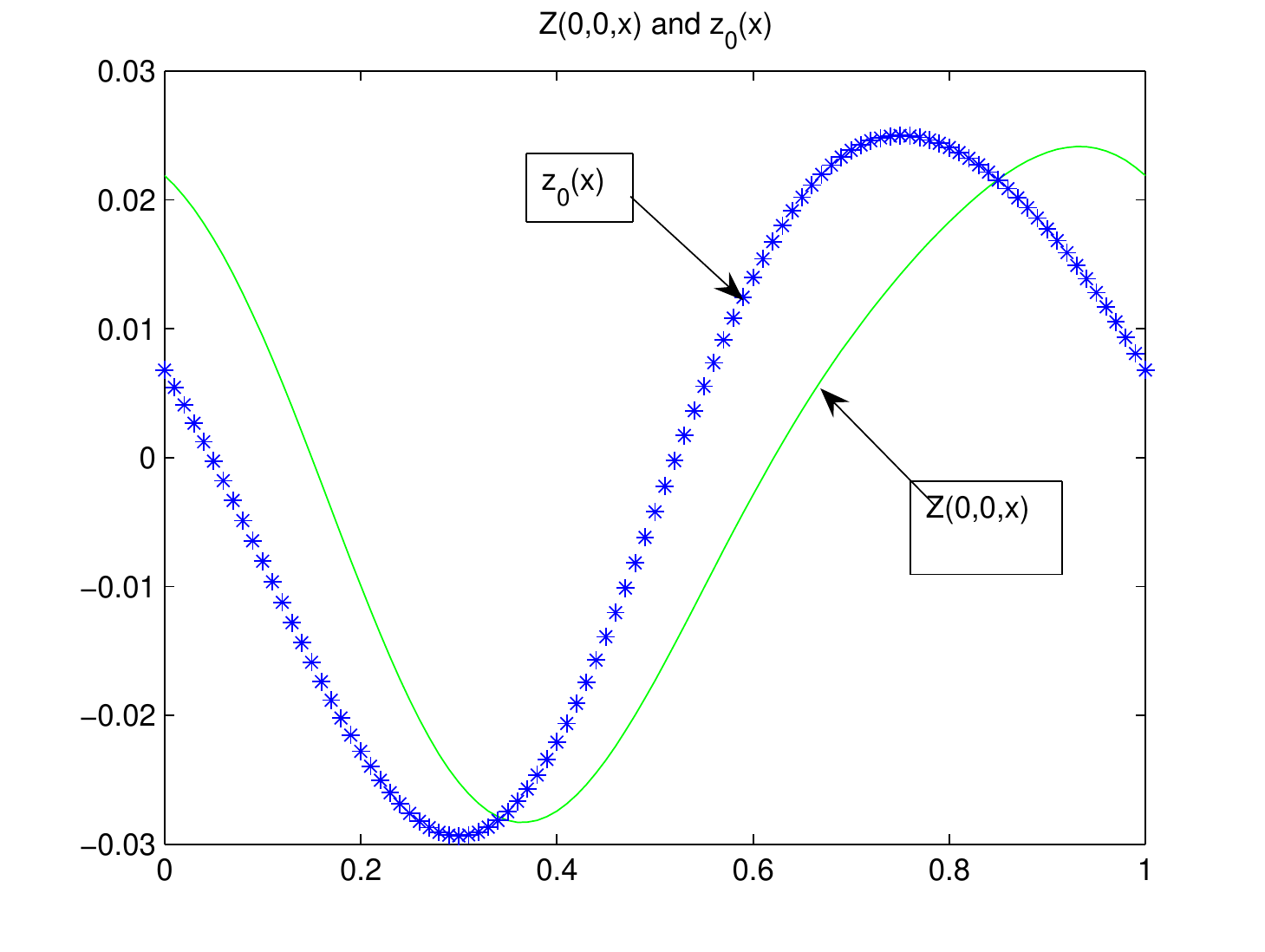}\includegraphics[angle=0, width=5cm]{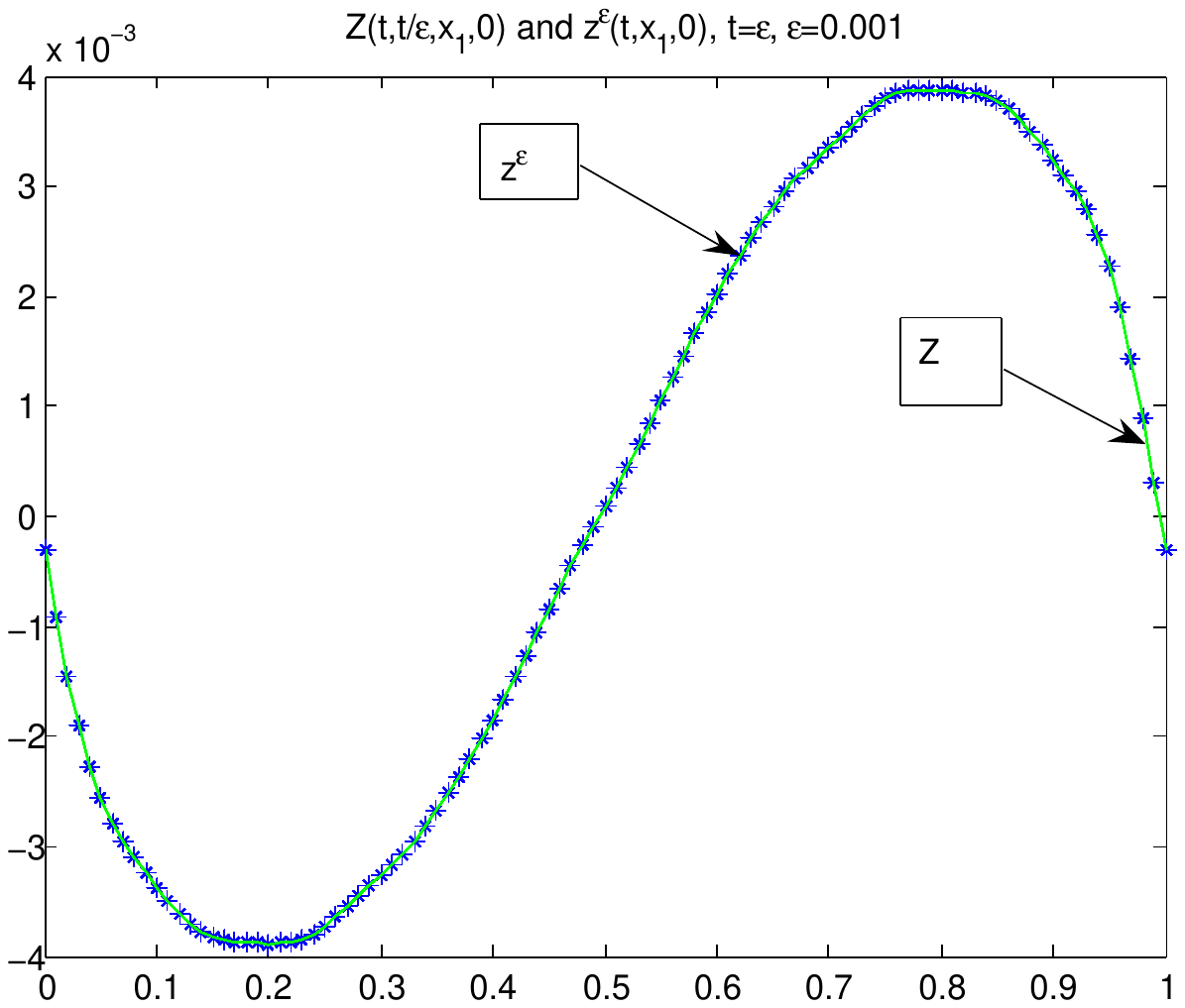}\includegraphics[angle=0, width=5cm]{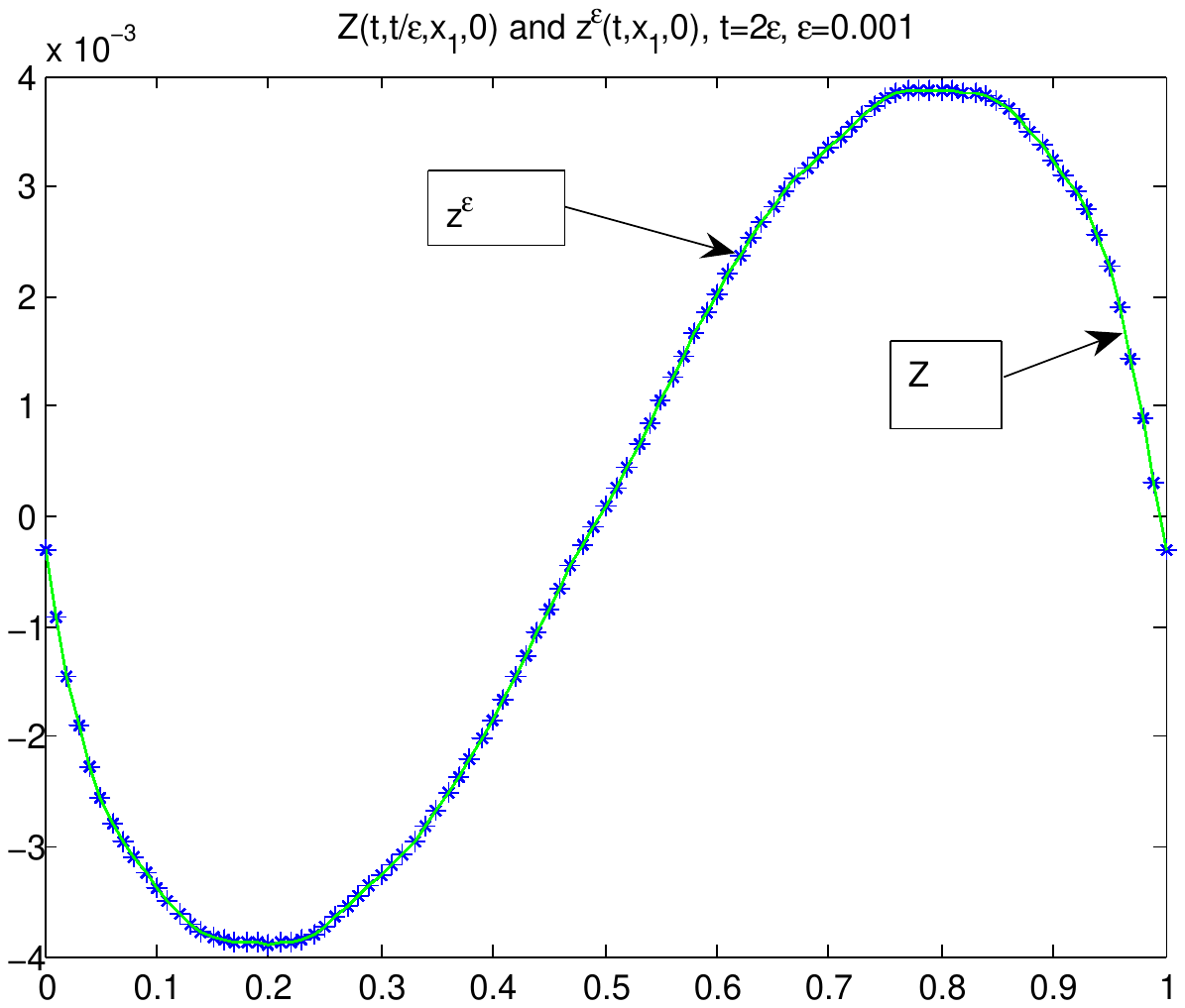}\\\end{center}
  \caption{\label{compare1}Comparison  of $z_P^{\epsilon}(t,x_1,0)$  and $Z_P(t,\frac{t}{\epsilon},x_1,0))$, $P=4$. On the left $t=0$, in the middle $t=\epsilon$ and $t=2\epsilon$ on the right,\,$\epsilon=0.001.$}
\end{figure}
\begin{figure}[htbp]
\begin{center}
  \includegraphics[angle=0, width=5cm]{compar00.pdf}\includegraphics[angle=0, width=5cm]{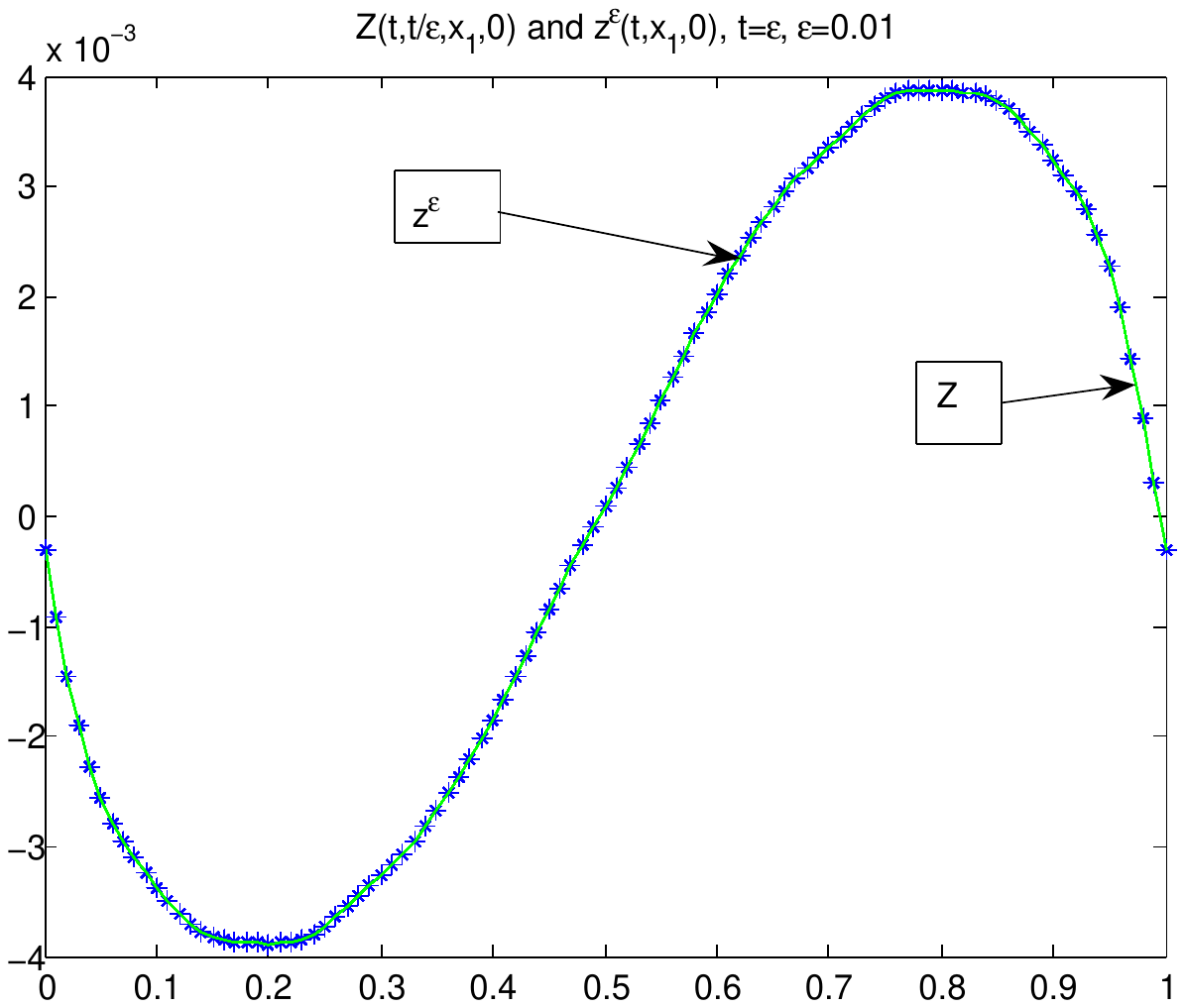}\includegraphics[angle=0, width=5cm]{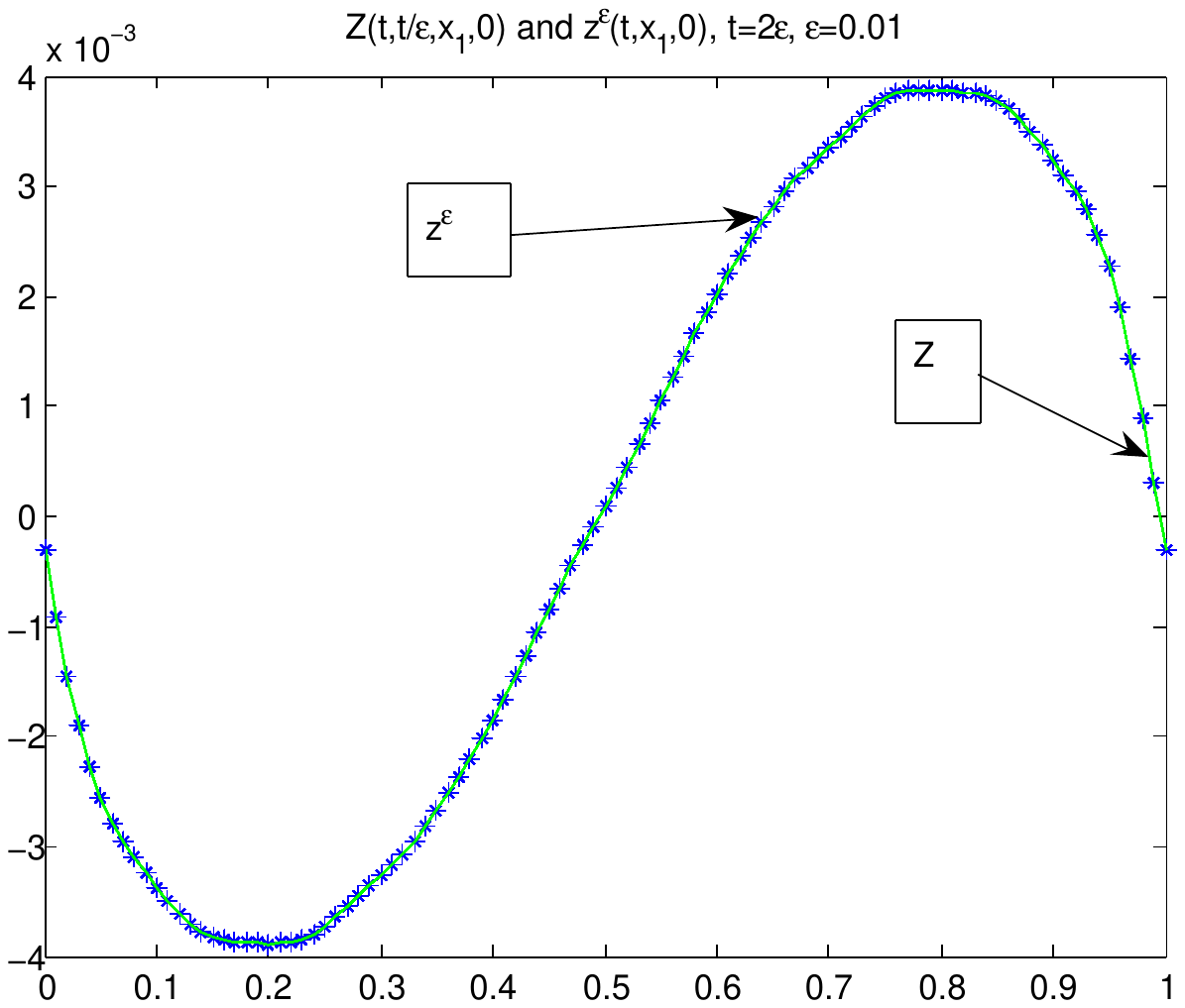}\\\end{center}
  \caption{\label{compare2}Comparison  of $z_P^{\epsilon}(t,x_1,0)$  and $Z_P(t,\frac{t}{\epsilon},x_1,0)).$ On the left $t=0$,
  in the middle $t=\epsilon$ and $t=2\epsilon$ on the right, $\epsilon=0.01.$}
\end{figure}
\newpage
So we can see from these figures that the solution $Z$ of the Two-Scale limit problem is such that $Z(t,\frac t\epsilon, \cdot,\cdot)$ is close
 to the solution $z^\epsilon(t,\cdot,\cdot)$ of the reference problem. 
In the presently considered case where the initial condition for $z^\epsilon$ is not $Z(0,0,\cdot,\cdot),$ we saw in Figure \ref{compare1} and Figure\,\ref{compare2} that $z_P^\epsilon$ tends to reach a steady state. This steady state is an oscillatory one in the sense that for large $t,\,\,z_P^\epsilon(t,\cdot,\cdot)$ behaves like $Z_P(t,\frac t\epsilon,\cdot,\cdot).$ This is illustrated by Figure\,\ref{imper} where $z_P^\epsilon(t,x_1,0)$ and $Z_P(t,\frac t\epsilon,x_1,0)$ are given for various value of $t$ in a period of lenght $\epsilon.$\\
More precisely, in this figure we see that within a period of time of lenght $\epsilon,\,z_{P}^\epsilon(t,\cdot,\cdot)$ and $Z_P(t,\frac t\epsilon,\cdot,\cdot)$ do not glue together completly. Nevertheless, despite this phenomenon which is linked with the fact that the Two-Scale approximation of $z^\epsilon(t,\cdot,\cdot)$ by $Z(t,\frac t\epsilon,\cdot,\cdot)$ is only of order 1 in $\epsilon,$ the two solutions re-glue well together at the end of the period.
\begin{figure}[htbp]
 \begin{center} \includegraphics[angle=0, width=7cm]{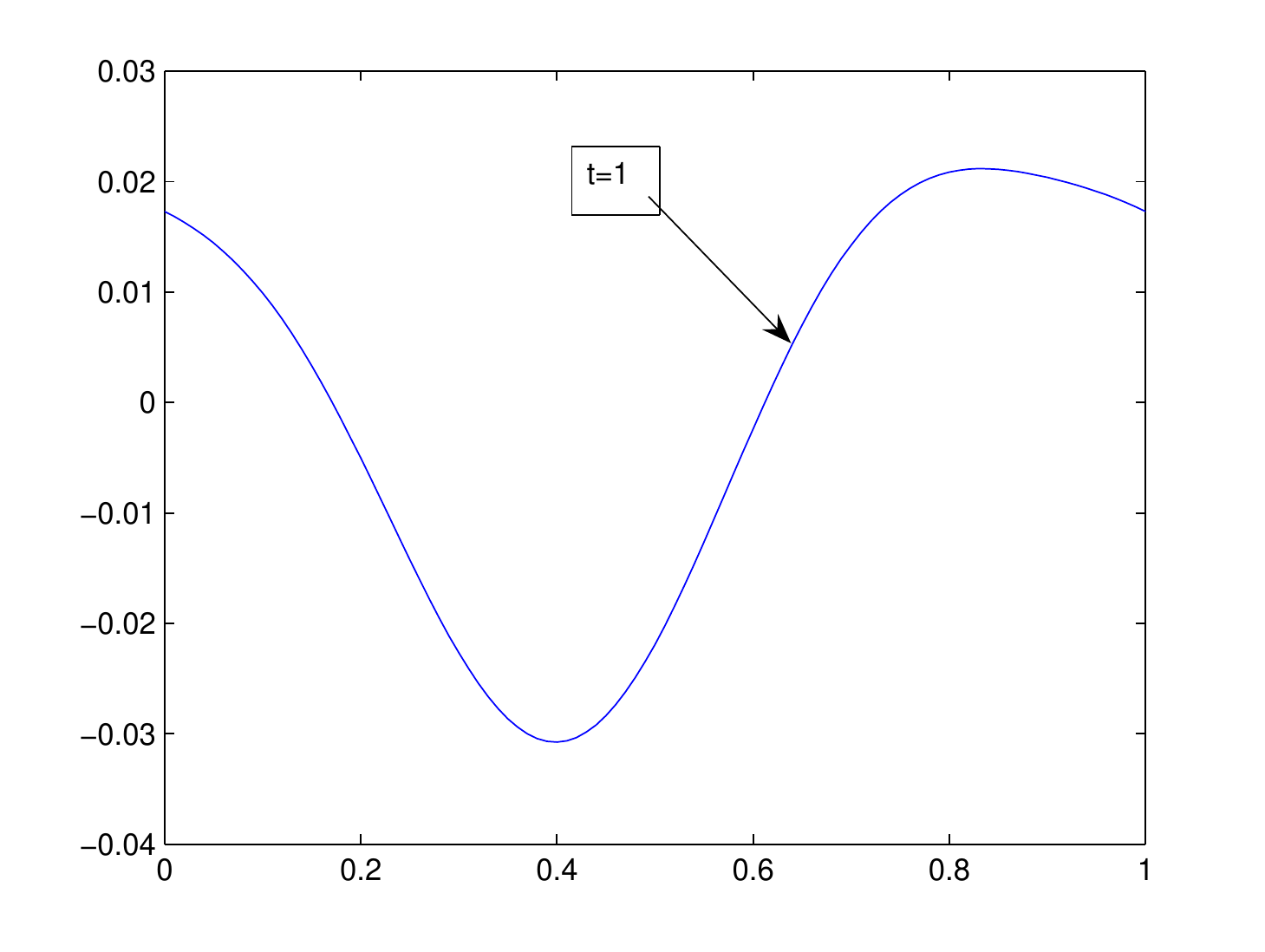}\includegraphics[angle=0, width=7cm]{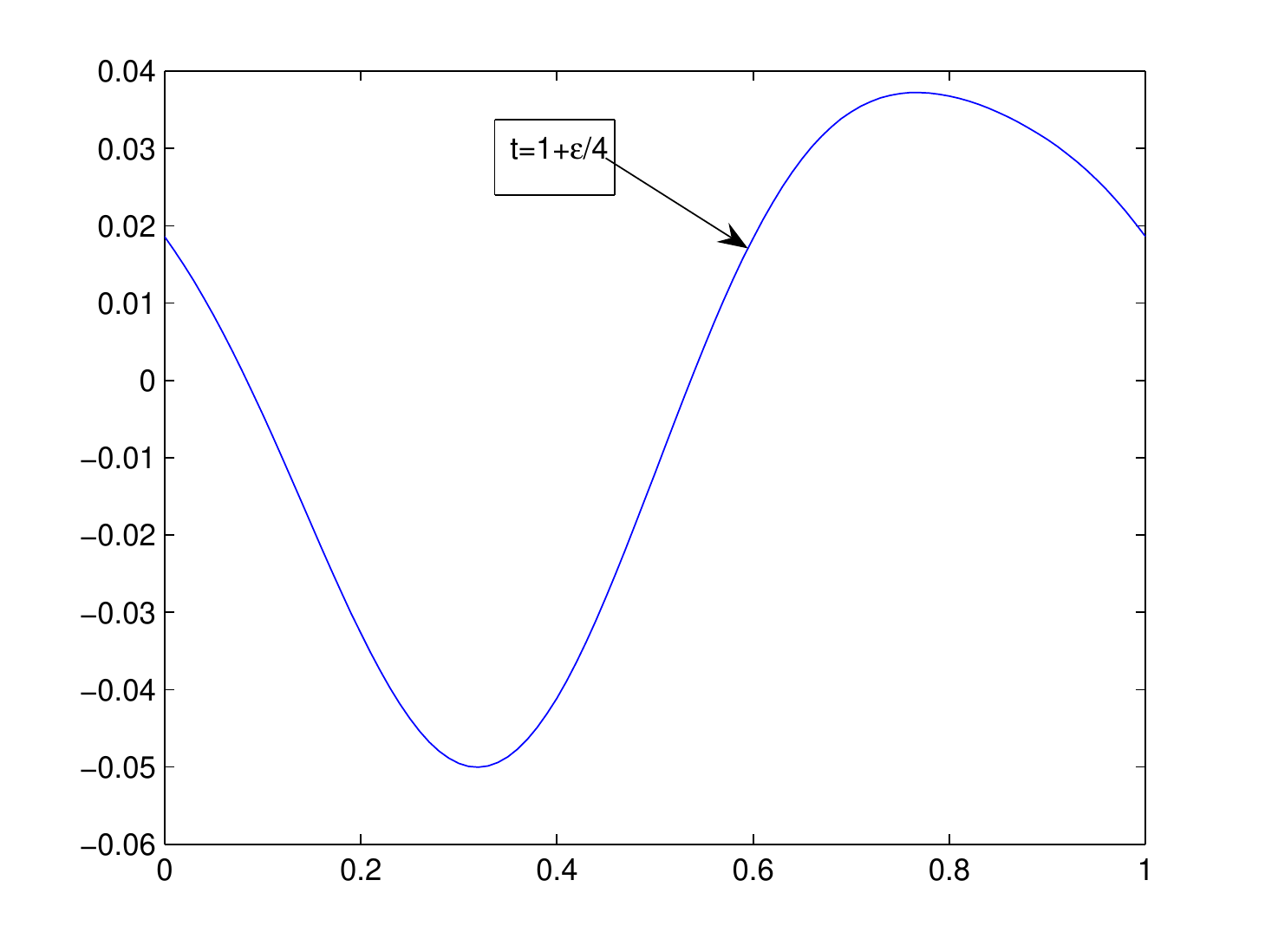}\\\includegraphics[angle=0, width=7cm]{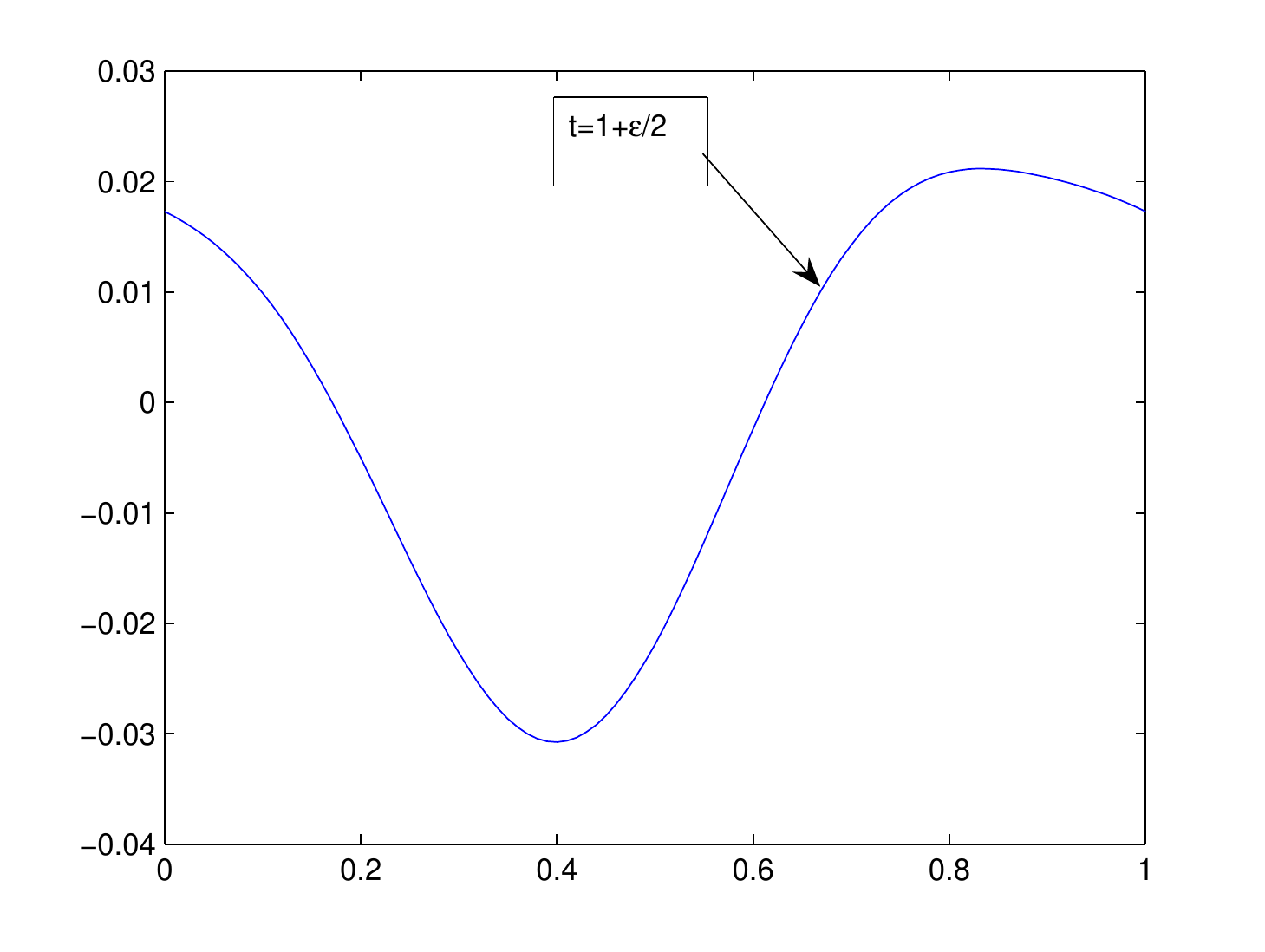}\includegraphics[angle=0, width=7cm]{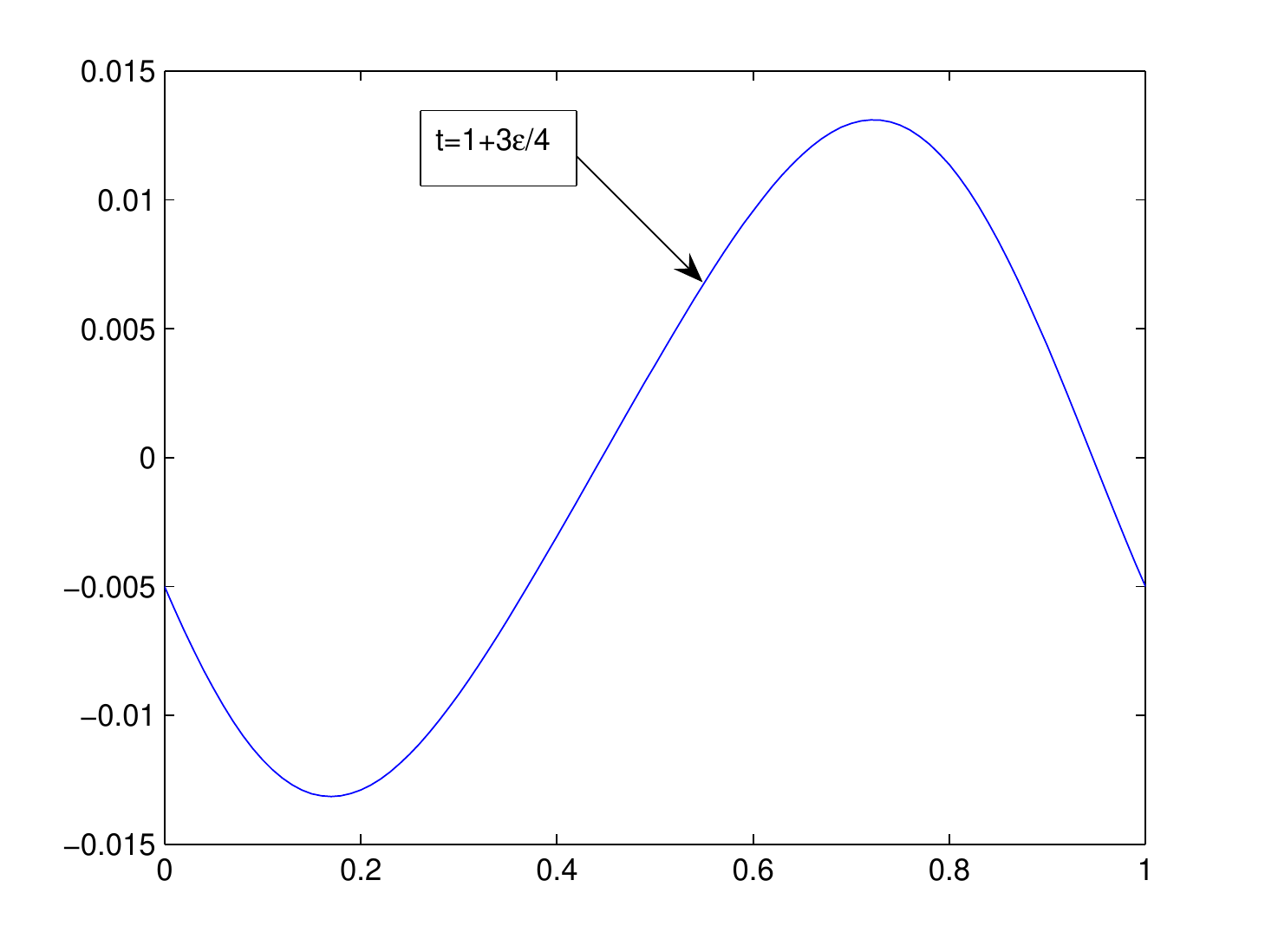}\\
 \includegraphics[angle=0, width=7cm]{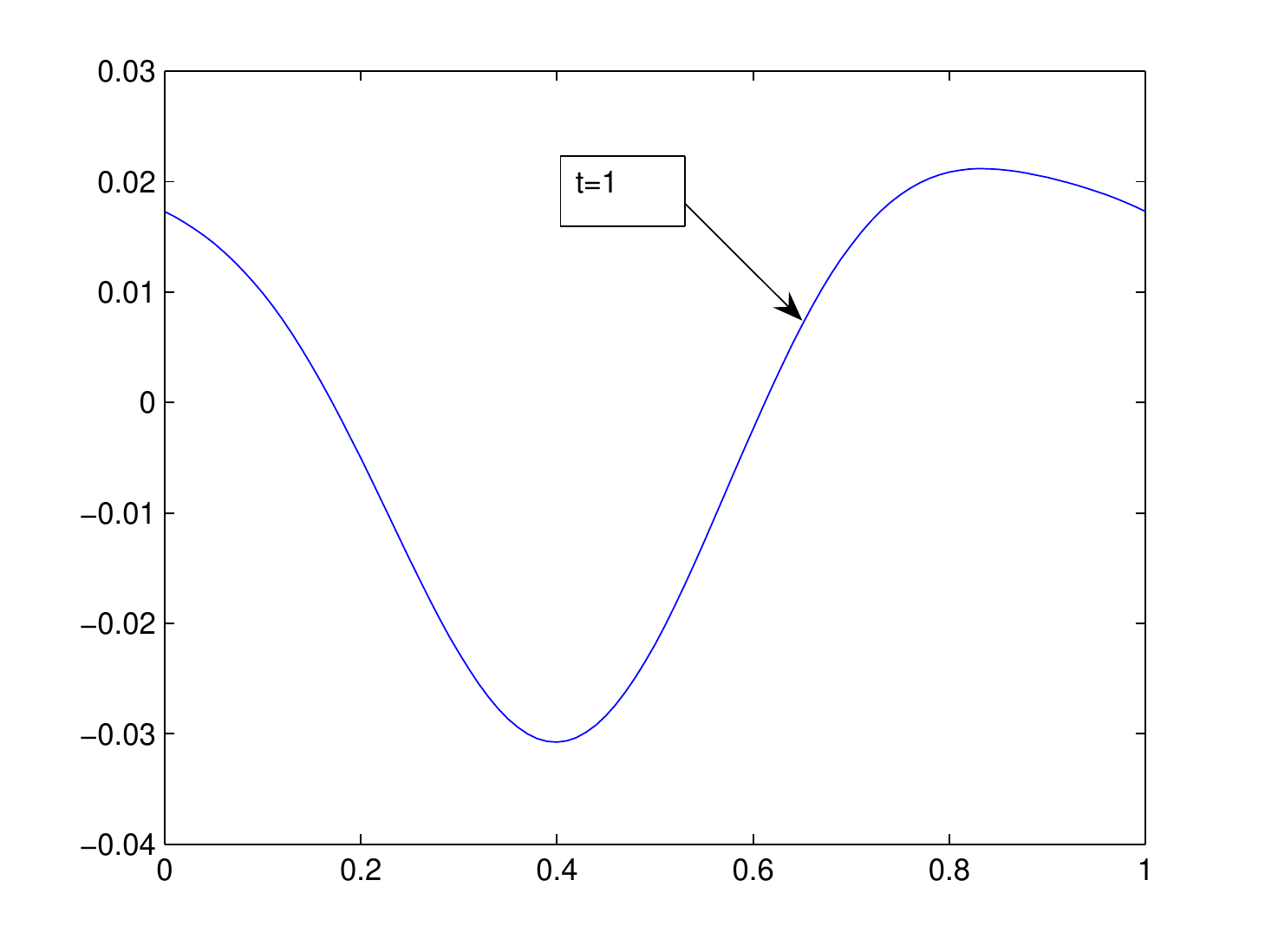}\includegraphics[angle=0, width=7cm]{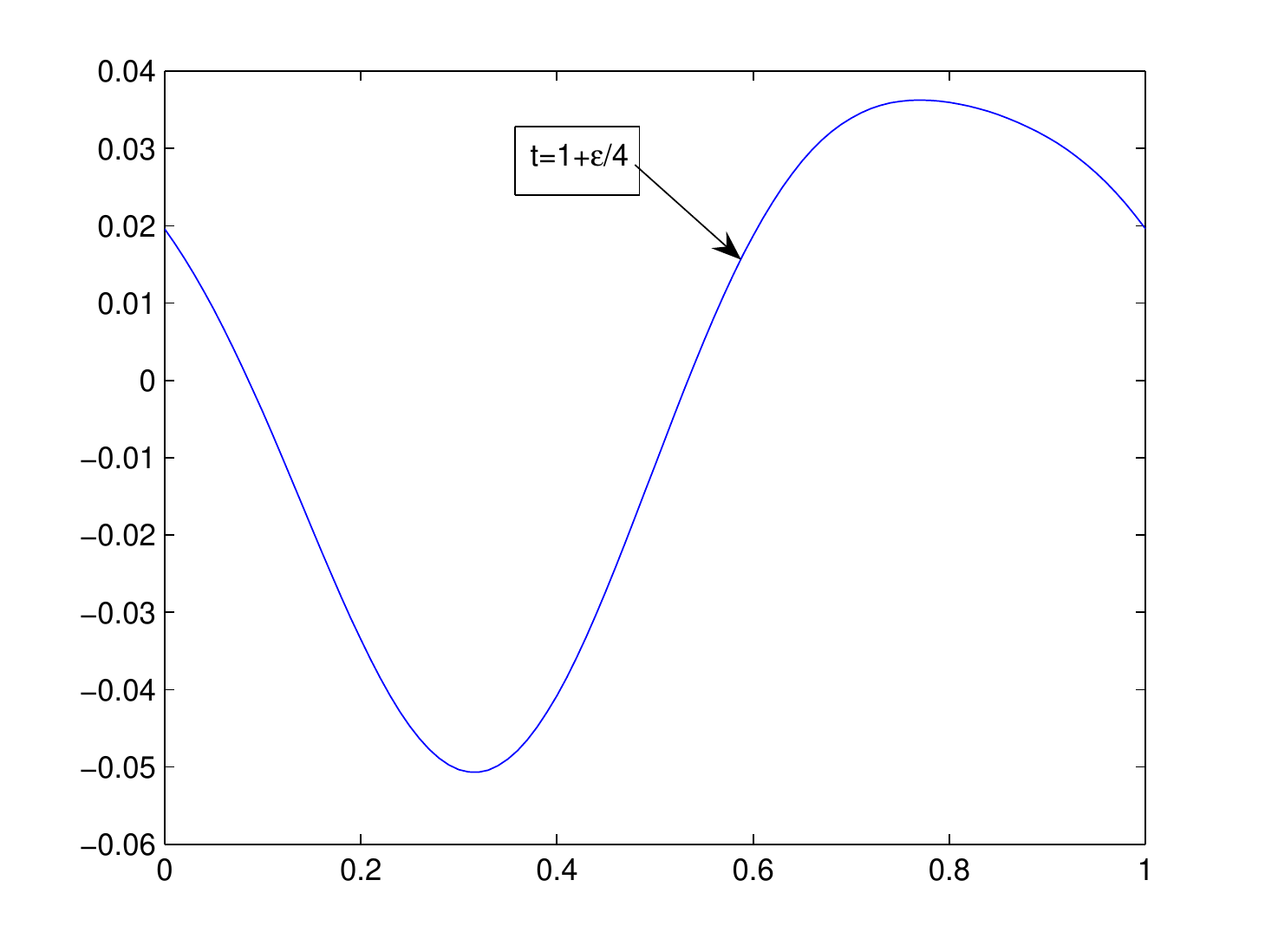}\\\includegraphics[angle=0, width=7cm]{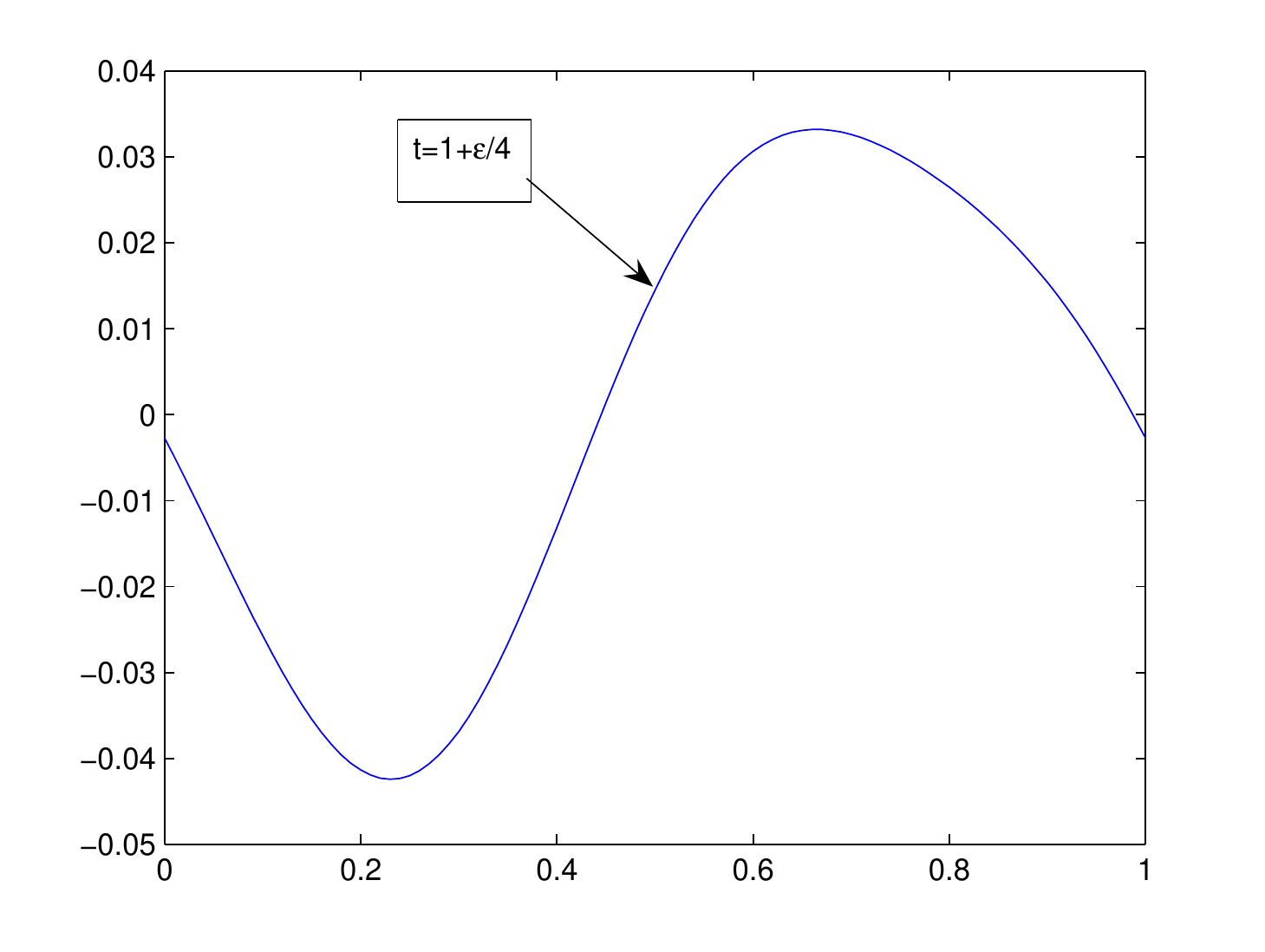}\includegraphics[angle=0, width=7cm]{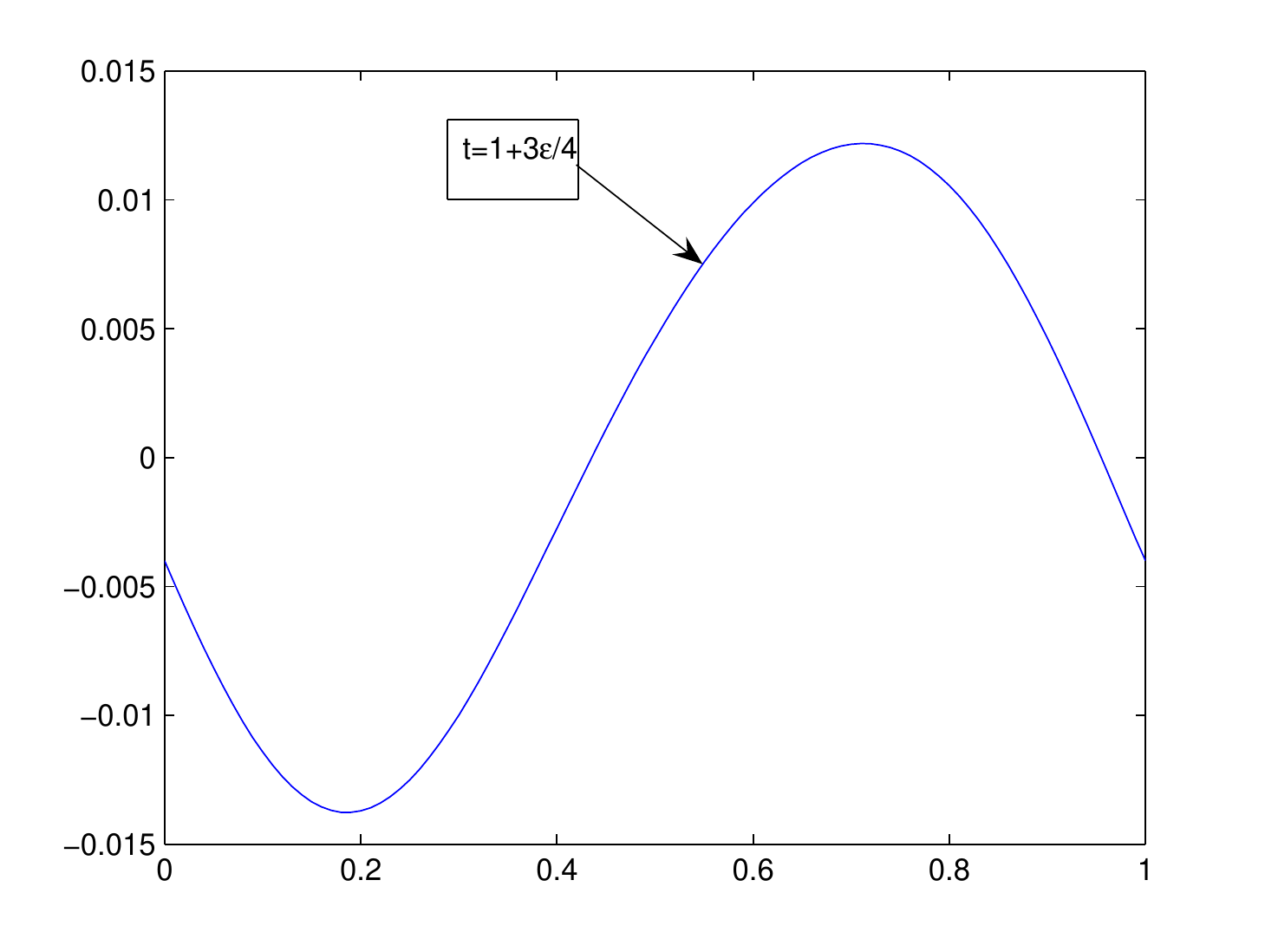}\\\end{center}
  \caption{\label{imper}Evolution of $Z_P(t,\frac{t}{\epsilon},x_1,0)$ in the top and $z_{P}^\epsilon(t,x_1,0)$ in the bottom, $t=1+\frac{n\epsilon}{4},\,\, n=0,1,2,3.$ }
\end{figure}
\newpage
\subsubsection{Comparisons of $z^\epsilon(t,x)$ and $Z(t,\frac t \epsilon,x)$ with $\mathcal{U}$ is given by (\ref{wat}).  }
In this subsection,  we do the same as in the precedent one, but when the velocity fields $\mathcal{U}$ given by (\ref{wat}). The results are all identical to the precedent one i.e.
the Two-Scale limit $Z_P(t,\frac{t}{\epsilon},x_1,x_2)$ is very close to the solution $z_P^\epsilon(t,x_1,x_2)$  to the reference problem when $P=4$.
The initial condition $z_0(x_1,x_2)\neq Z(0,0,x_1,x_2)$ and is the same as in the subsection \ref{sectionavant}.
The results are given for $\epsilon=0.1$ and $\epsilon=0.005$ and for various time $t.$ We notice that $z^\epsilon$ comes very close to
$Z(t,\frac{t}{\epsilon},x_1,x_2)$ when $\epsilon$ is very small. We begin by giving the space distribution of $\mathcal U$ at various time and
the $\theta-$evolution of $\mathcal U$ and $\widetilde{\mathcal{A}}.$
The second  velocity fields is given by 
\begin{equation}\label{wat}\mathcal{U}(t,\theta,x_1,x_2)=\mathcal{U}(t,\theta,x)=\left\{\begin{array}{ccc}\vspace{0.25cm}0\,\,\text{in}\,\,[0,\theta_1],\\
\vspace{0.25cm}
\frac{\theta-\theta_1}{\theta_2-\theta_1}U_{thr} \mathbf e_2\,\,\text{in}\,\,[\theta_1,\theta_2],\\
\vspace{0.25cm}
U_{thr}\mathbf e_2+\phi(\frac{\theta-\theta_2}{\theta_3-\theta_2})\psi(t,x)\,\,\text{in}\,\,[\theta_2,\theta_3],\\
\vspace{0.25cm}
\frac{\theta-\theta_3}{\theta_4-\theta_3}U_{thr} \mathbf e_2\,\,\text{in}\,\,[\theta_3,\theta_4],\\
\vspace{0.25cm}
0\,\,\text{in}\,\,[\theta_4,\theta_5],\\
\vspace{0.25cm}
\frac{\theta-\theta_5}{\theta_6-\theta_5}U_{thr} \mathbf e_2\,\,\text{in}\,\,[\theta_5,\theta_6],\\
\vspace{0.25cm}
-U_{thr} \mathbf e_2-\phi(\frac{\theta-\theta_6}{\theta_7-\theta_6})\psi(t,x)\,\,\text{in}\,\,[\theta_6,\theta_7],\\
\vspace{0.25cm}
-\frac{\theta-\theta_7}{\theta_8-\theta_7}U_{thr} \mathbf e_2\,\,\text{in}\,\,[\theta_7,\theta_8],\\
\vspace{0.25cm}
0\,\,\text{in}\,\,[\theta_8,1],\\
\end{array}\right.\end{equation}
where $U_{thr}>0,\,\,\phi$ is a regular positive function satisfying $\phi(s)=s(1-s)$ and 
$\psi(t,x_1)=(1+\sin\frac{\pi}{30}t)(U_{thr}\mathbf e_2+\frac{1}{10}(1+\sin2\pi x_1)\mathbf e_1),\,\,\theta_{i}=\frac{i+1}{10},\,\,i=1,\ldots,8.$\\
The $\theta$-evolution of $\mathcal{U},$ given by (\ref{wat}), is given in Figure\,\ref{UA1} for various position in $[0,1]^2.$\\
Function $g_a(\mathbf u)=g_c(\mathbf u)=|\mathbf u|^3, a=c=1$ and $\mathcal{M}(t,\theta,x)=0$ which yields a $\theta$-evolution of
$\widetilde{\mathcal{A}}(\theta)$ which is drawn for various positions in Figure\,\ref{UA2}.\\
\vspace{9cm}
\begin{figure}[htbp]
\begin{center}
  \includegraphics[angle=0, width=5cm]{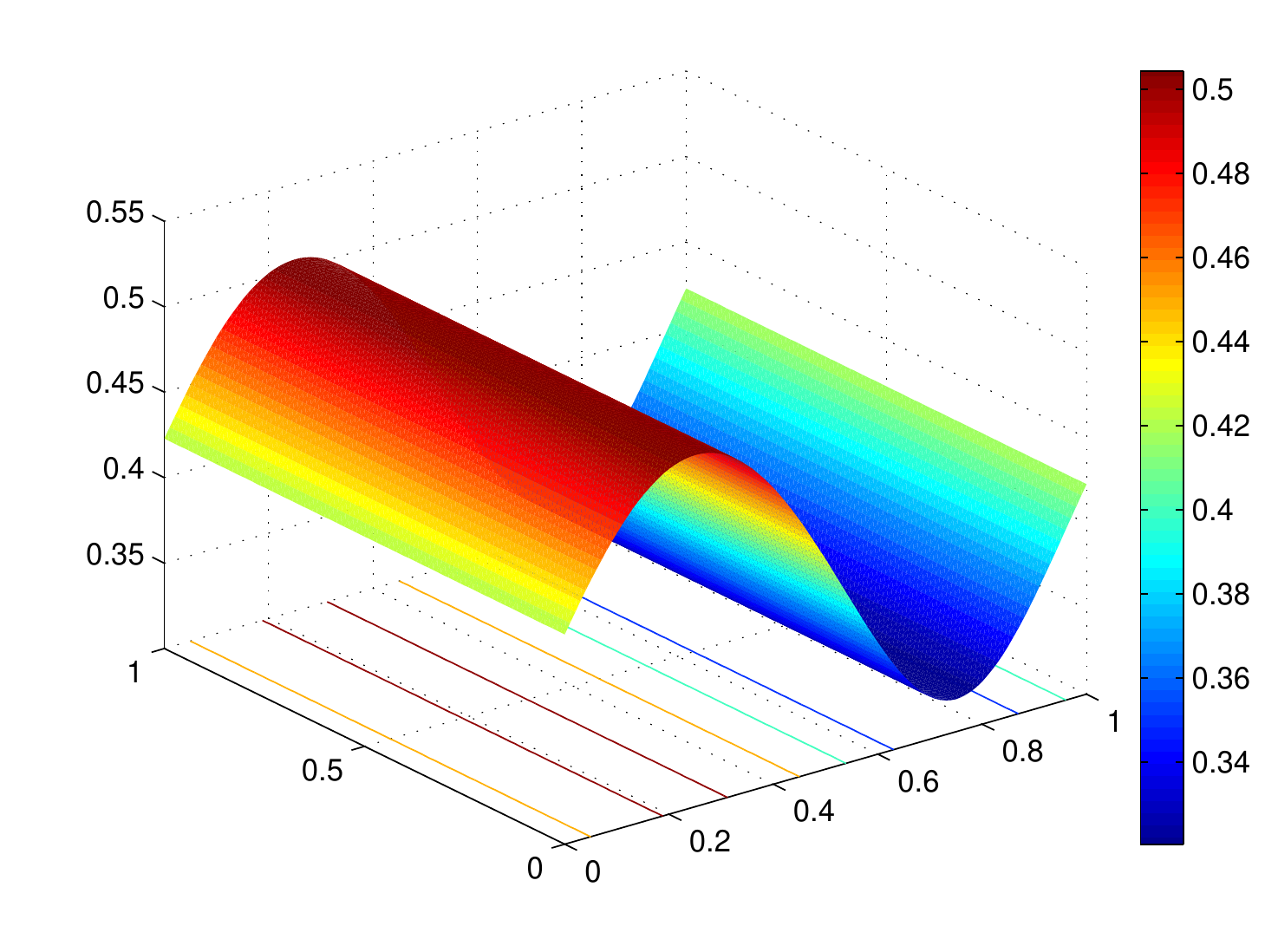}\includegraphics[angle=0, width=5cm]{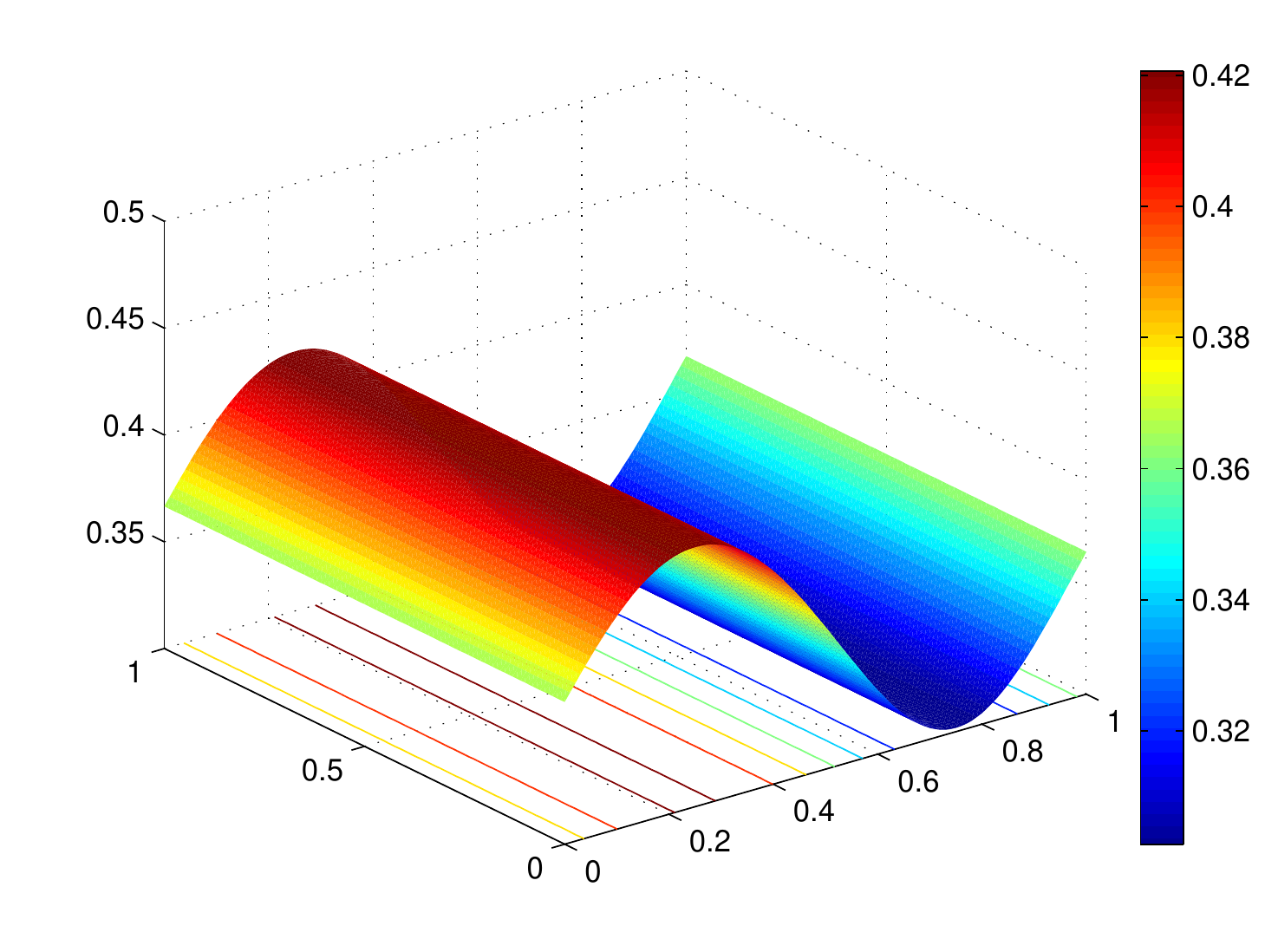}\\
\includegraphics[angle=0, width=5cm]{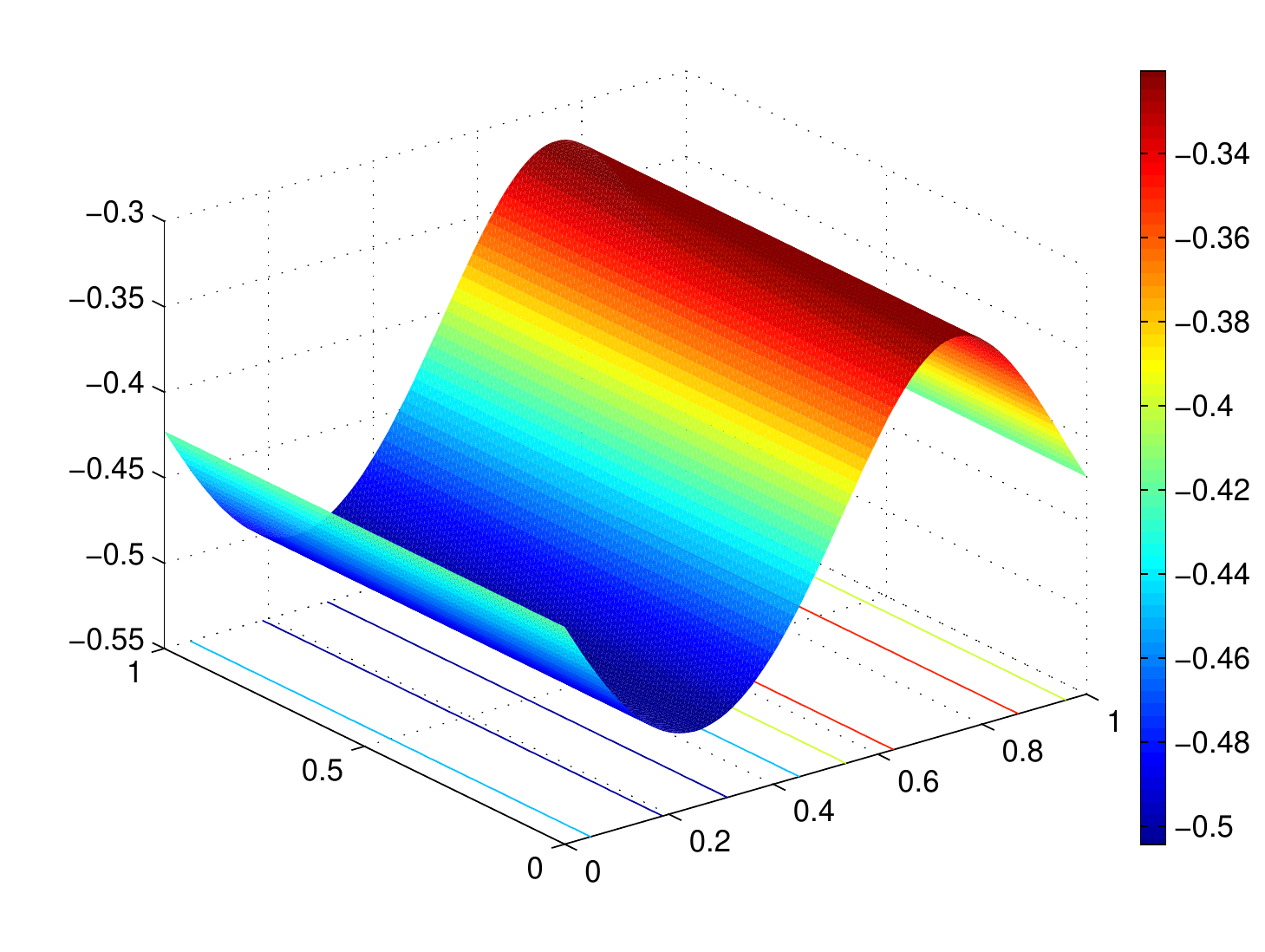}\end{center}
\caption{Space distribution of the first component of $\mathcal{U}(1, 0.25,(x_1,x_2))$, $\mathcal{U}(1, 0.275,(x_1,x_2))$
 and $\mathcal{U}(1, 0.75,(x_1,x_2))$ when $\mathcal{U}$ is given by (\ref{wat}).}\label{UA2x}
\end{figure}
\begin{figure}[!]
\begin{center}
  \includegraphics[angle=0, width=6cm]{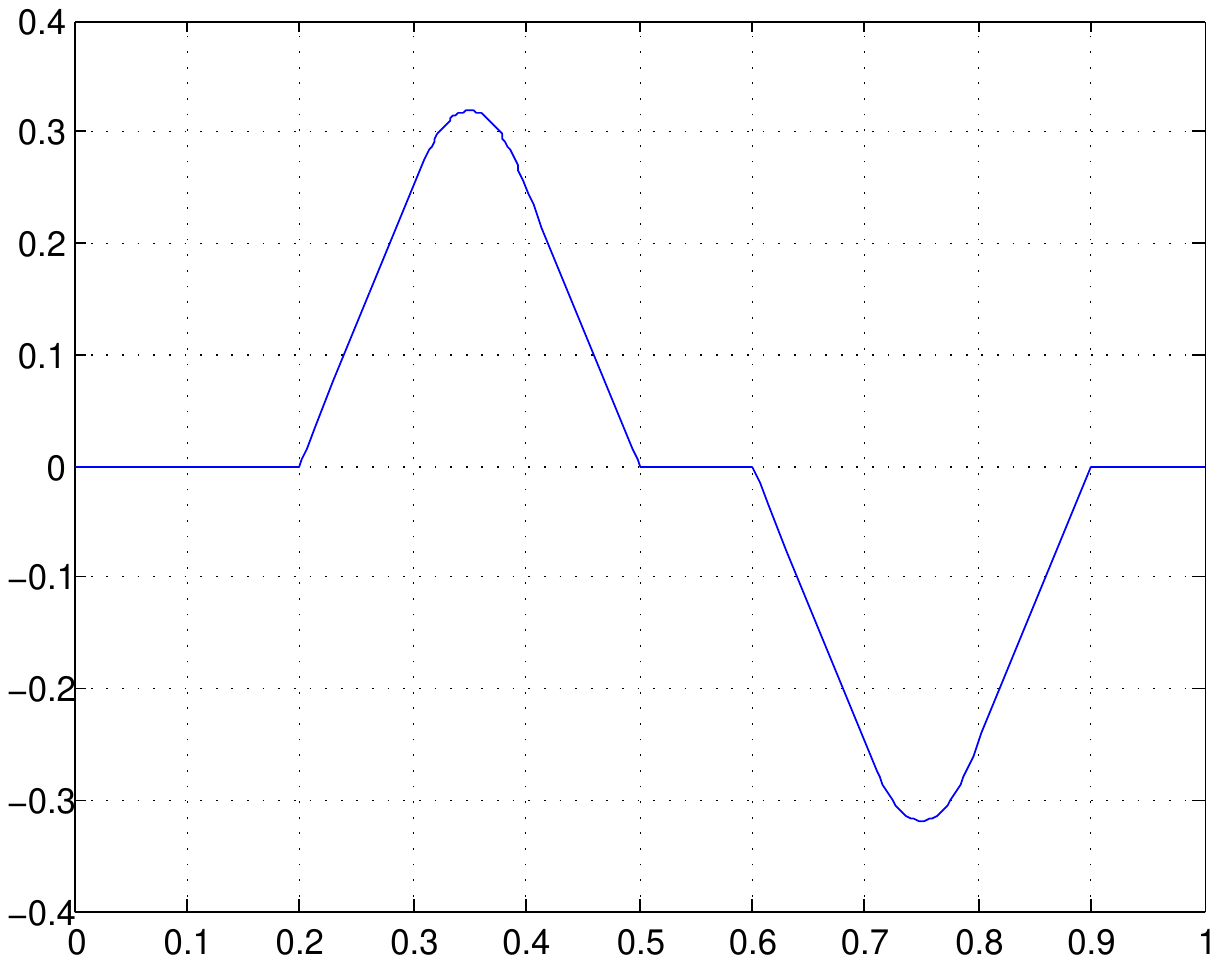}\includegraphics[angle=0, width=6cm]{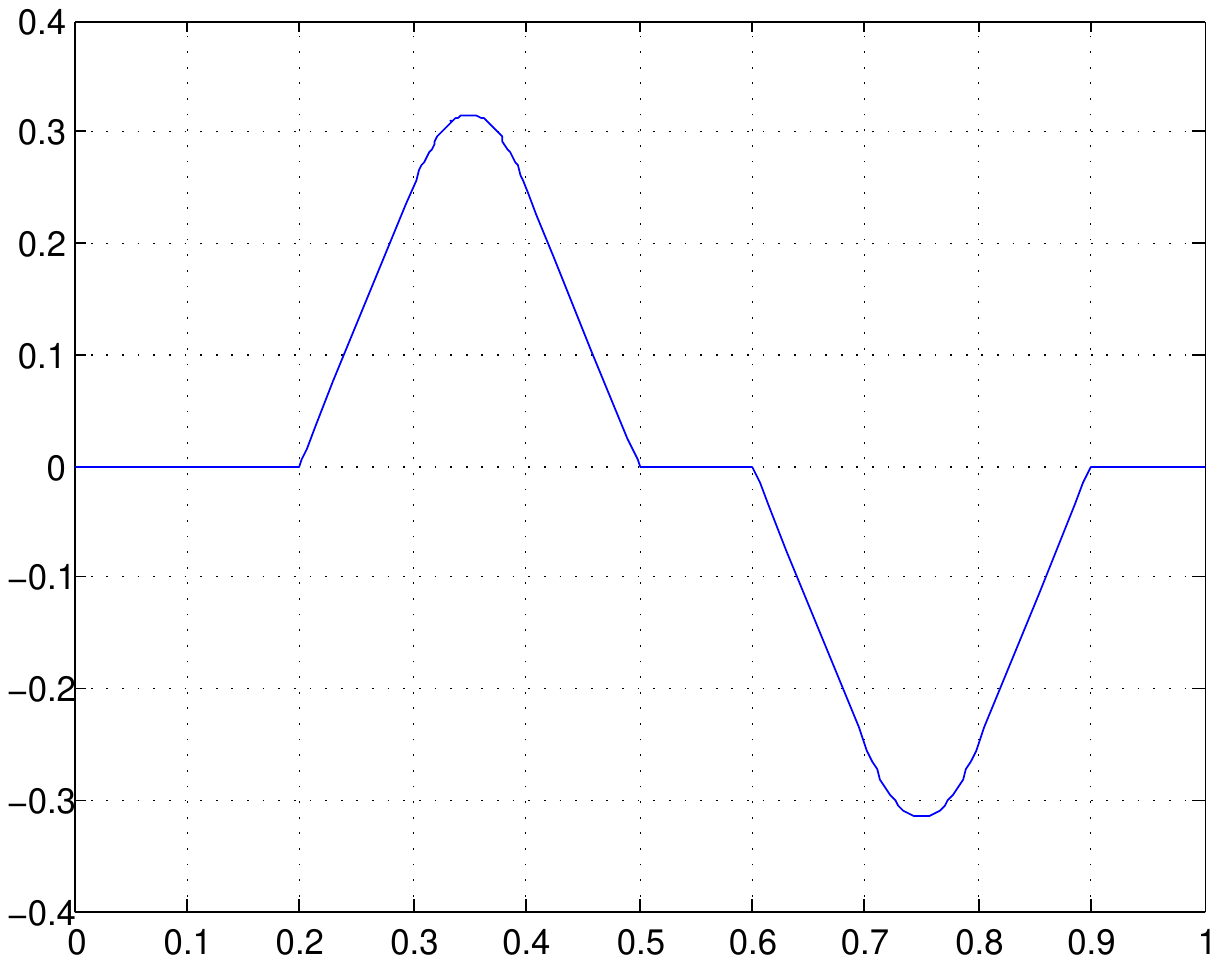}\\
\includegraphics[angle=0, width=6cm]{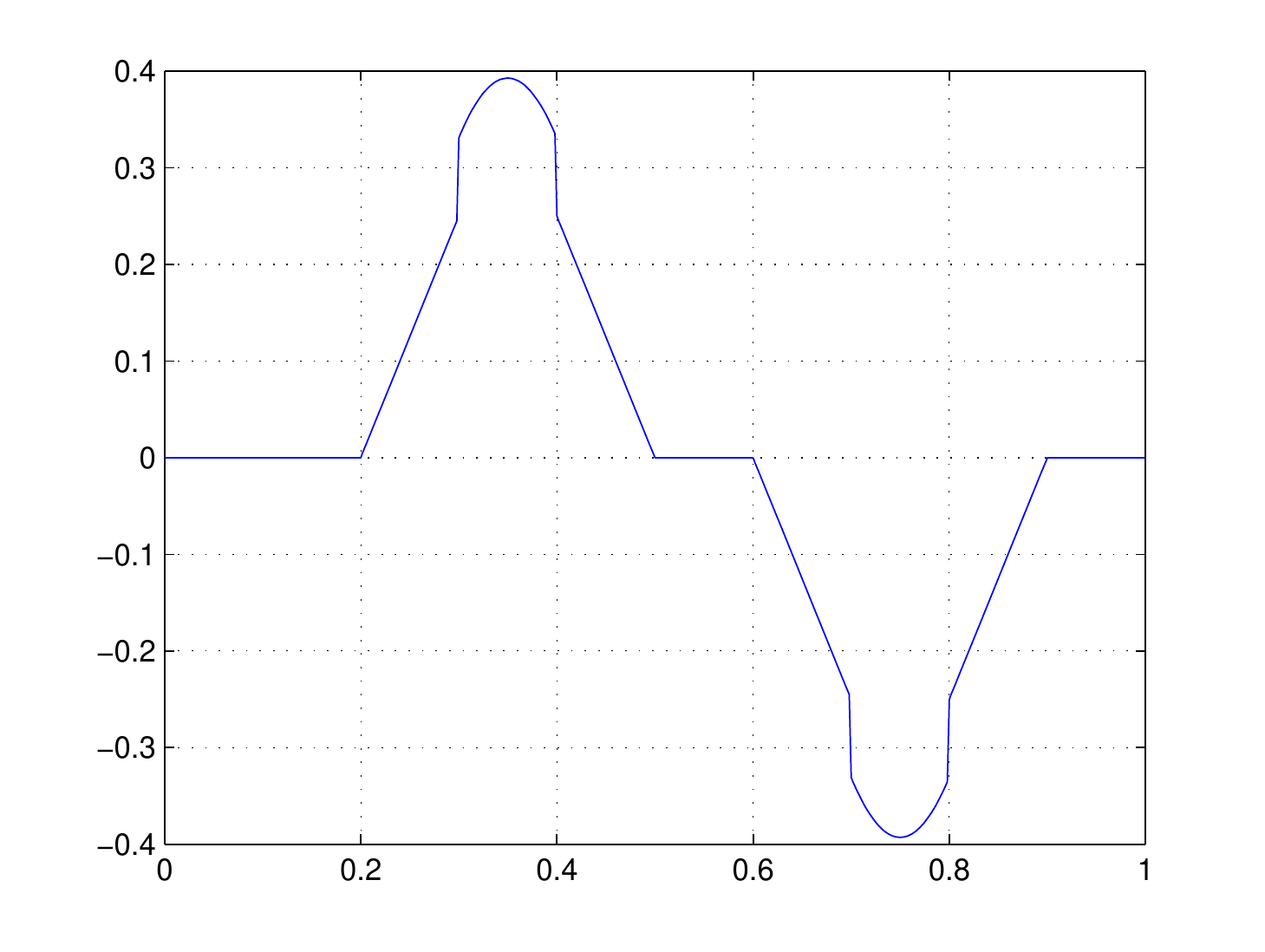}\\
\end{center}
\caption{\text{\label{UA1}$\theta$-evolution of} $\mathcal{U}(1, \theta,(1,0)),\,\,\mathcal{U}(1,\theta,(4,0))$
\text{and}
$\mathcal{U}(1,\theta,(1/3,1/3))$ \text{when} $\mathcal{U}$ \text{is given by} (\ref{wat}). }
\end{figure}
\begin{figure}[htbp]
\begin{center}  
  \includegraphics[angle=0, width=6cm]{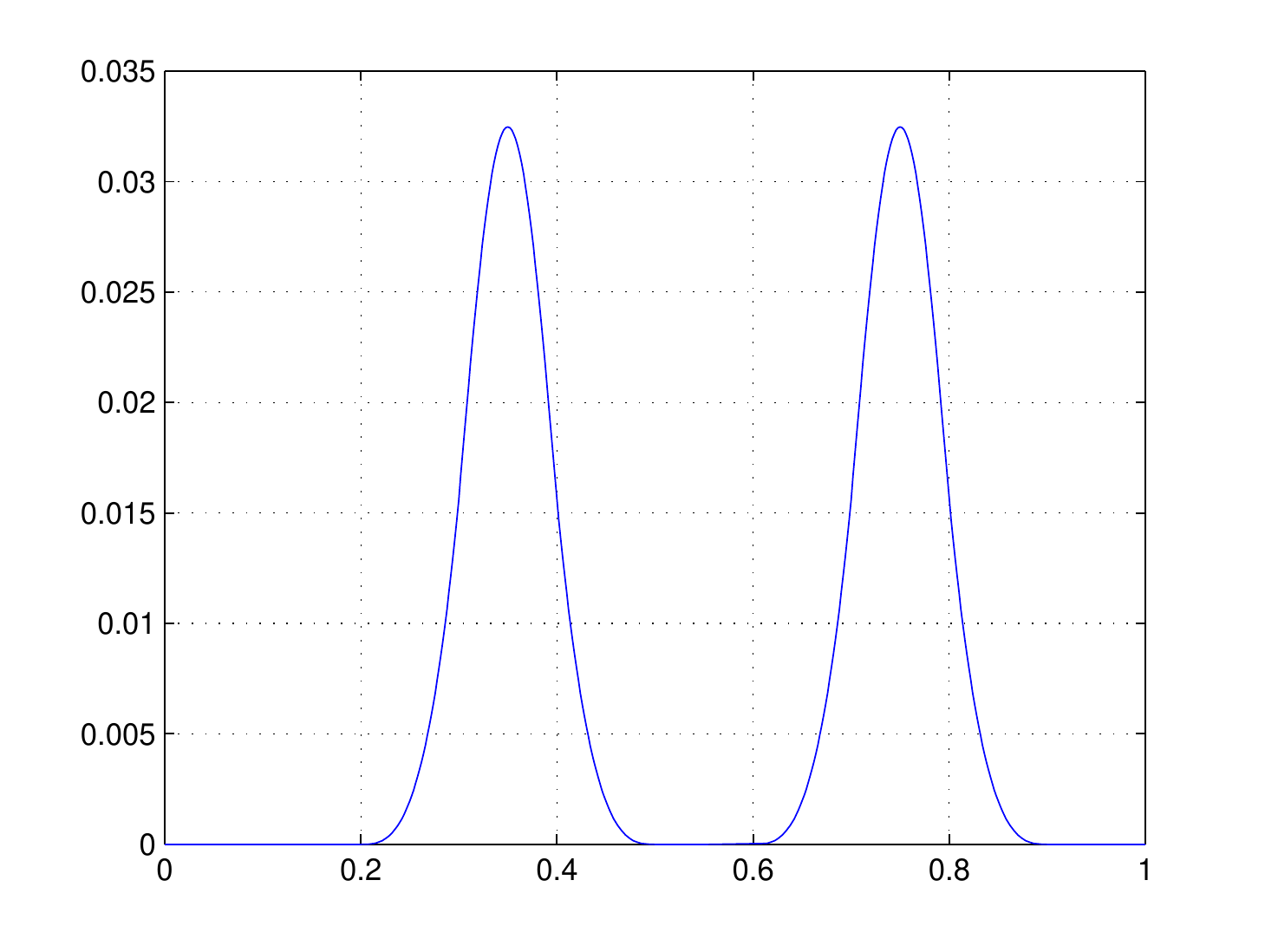}\includegraphics[angle=0, width=6cm]{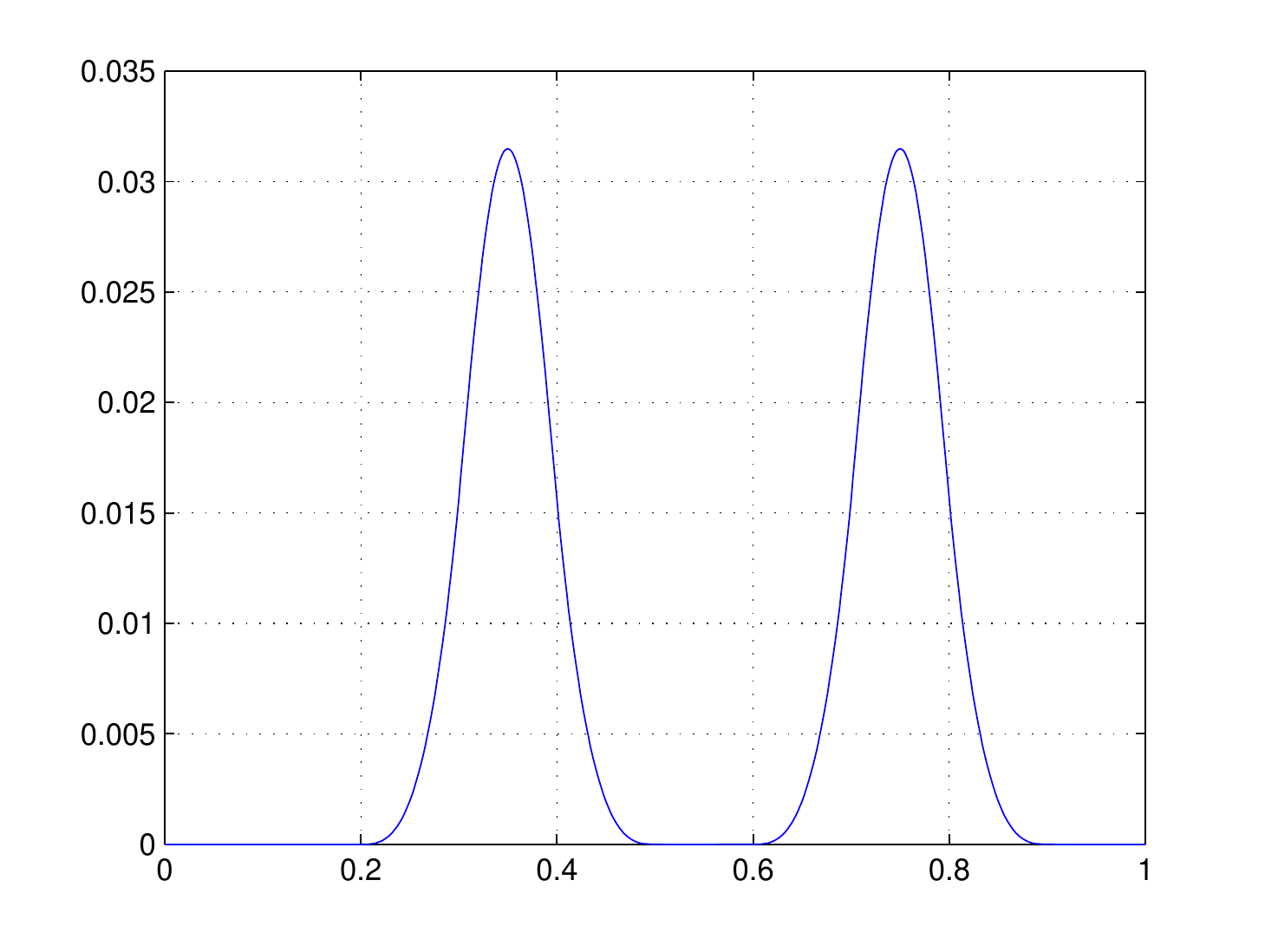}\\
\includegraphics[angle=0, width=6cm]{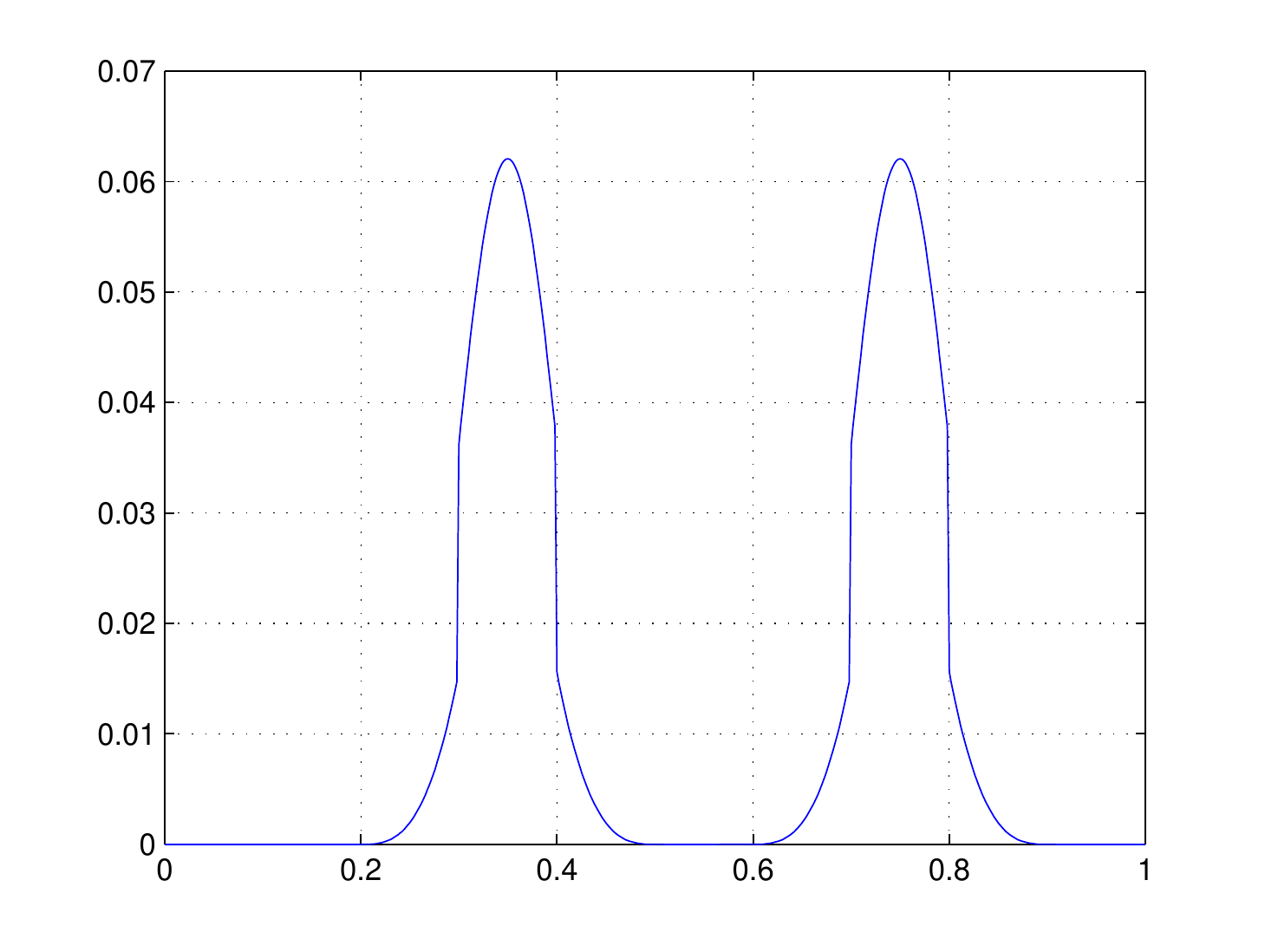}\\
\caption{$\theta$-evolution of $\widetilde{\mathcal{A}}(1, \theta,(1,0))$, $\widetilde{\mathcal{A}}(1,\theta,(4,0))$ and
 $\widetilde{\mathcal{A}}(1,\theta,(1/3,1/3))$ when $\mathcal{U}$ is given by (\ref{wat}).} \label{UA2}
 \end{center}
 \end{figure}

\newpage\noindent
Using this, we compute $Z_P(t,\frac{t}{\epsilon},x_1,x_2)$ and $z_{P}^\epsilon(t,x)$ for $P=4$.
To compute $z_{P}^\epsilon(t,x)$  we take $z_0(x_1,x_2)=\cos2\pi x_1+\cos 4\pi x_1$ which is not $Z(0,0,x_1,x_2).$
First we study  the errors $Z_P(t,\frac{t}{\epsilon},x_1,x_2)-z_{P}^\epsilon(t,x)$ at $t=1$. This quantity
decreases when $\epsilon$ decreases as illustrated in the following tabular.\\
\begin{center}
\begin{tabular}{|c|c|c|c|}
\hline
value of $\epsilon$ & norm $L^1$ & norm $L^2$ & norm $L^\infty$\\
\hline
0.01 &  0.012212 & 0.00048013 & 0.003376\\
0.03 &  0.019082     &  0.0005753     &  0.0017347         \\
0.05 &  0.030769      &  0.01348     &  0.0069818         \\
0.07 &  0.045123      &  0.029055      &  0.009         \\
0.09 &  0.17067      &   0.10562      &   0.038790        \\
0.1 &   0.3053      &    0.10562      &   0.04878        \\
\hline
\end{tabular}\\
Table: Errors norm $Z_P(t,\frac{t}{\epsilon},x_1,x_2)-z_{\tilde P}^\epsilon(t,x_1,x_2),\,\,\tilde P=(4,4),\,P=(4,4,4),\,\,t=1.$ \\

\end{center}
The results given in this table show that, at time $t=1,\,\, z^\epsilon(t,x)$ is closer to $Z(t,\frac{t}{\epsilon},x)$ when $\epsilon$ is very small. These results validate the results obtained in Theorem\,\ref{thHomSecHomn}.

In Figures \ref{mard3} and \ref{mard4},  we present simulations  at times $t=0.75$
and $t=0.775$. We see that $Z_P(t,\frac{t}{\epsilon},x_1,x_2)$ is close to $z_{P}^\epsilon(t,x_1,x_2).$ The numerical results shown in these figures
  are made with $\epsilon=0.1.$

\begin{figure}[htbp]
\begin{center}
 \includegraphics[angle=0, width=8cm]{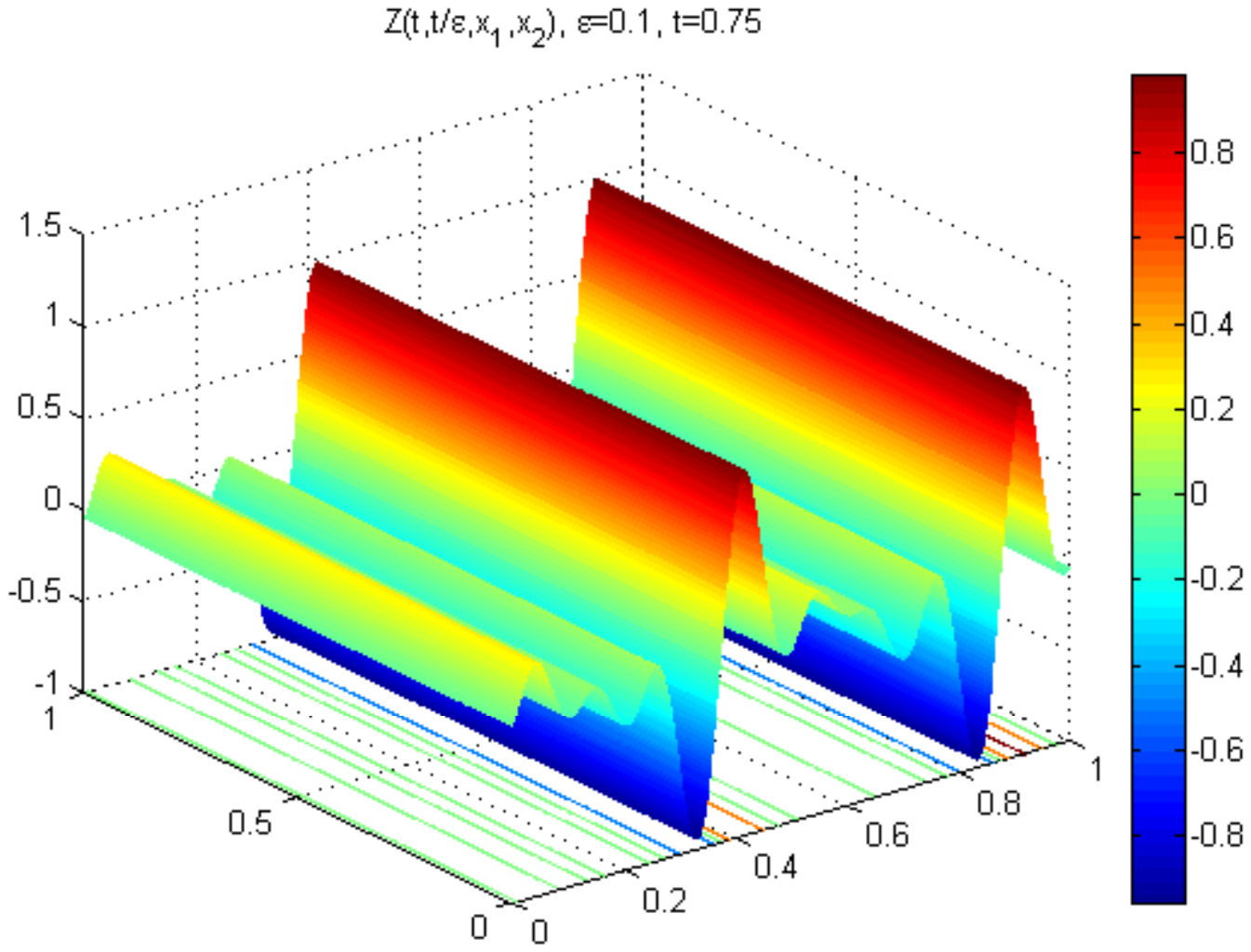}\includegraphics[angle=0, width=8cm]{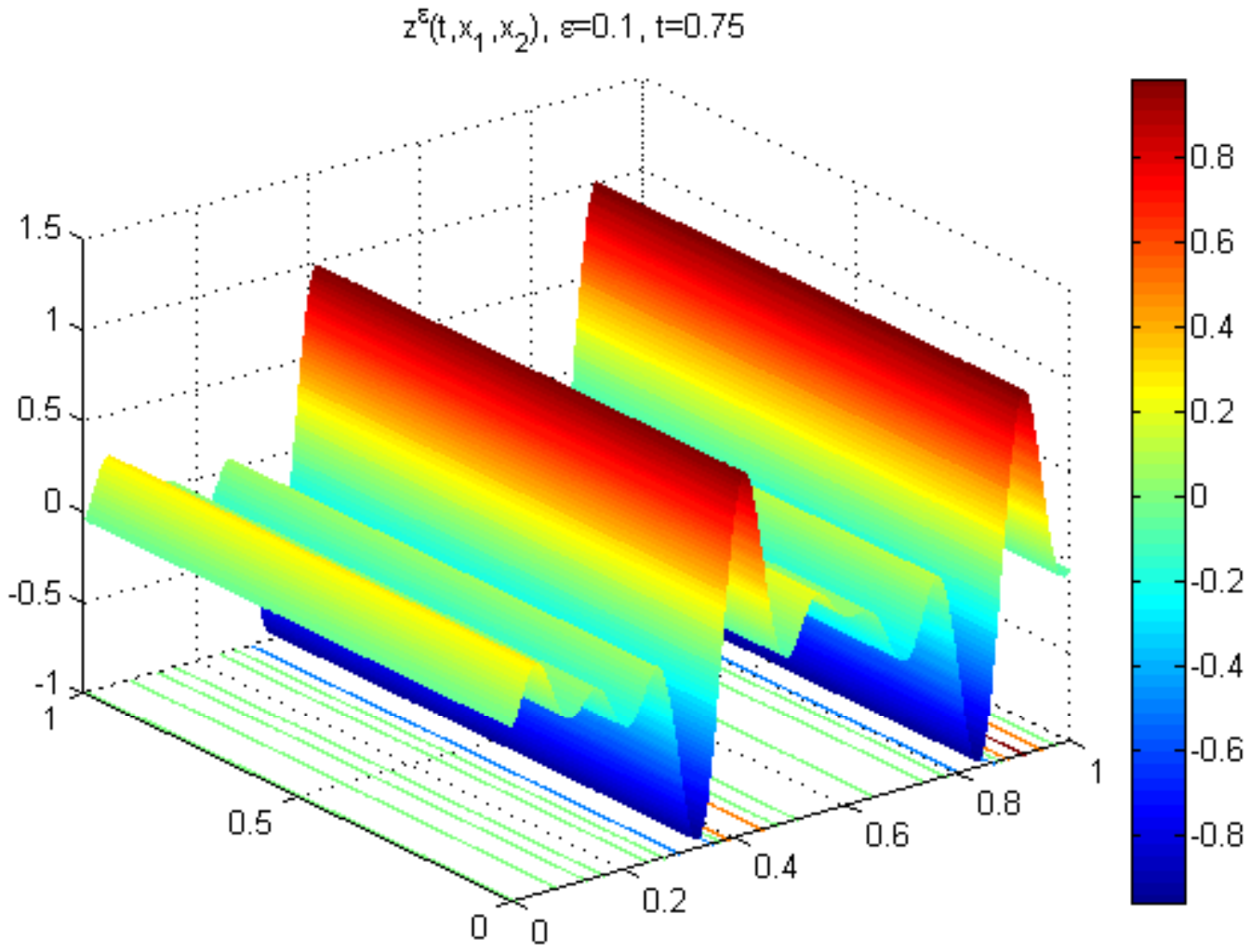}\\\end{center}
 \caption{\label{mard3}Comparison of $z_P^{\epsilon}(t,x_1,x_2)$  and
 $Z_P(t,\frac{t}{\epsilon},x_1,x_2)$,  $P=4$; $t=0.75,$ $\epsilon=0.1,\,\,z_0(x_1,x_2)=\cos2\pi x_1+\cos 4\pi x_1.$
 On the left $Z_P(t,\frac{t}{\epsilon},x_1,x_2)$, on the right $z_{P}^\epsilon(t,x_1,x_2)$.}
\begin{center}
 \includegraphics[angle=0, width=8cm]{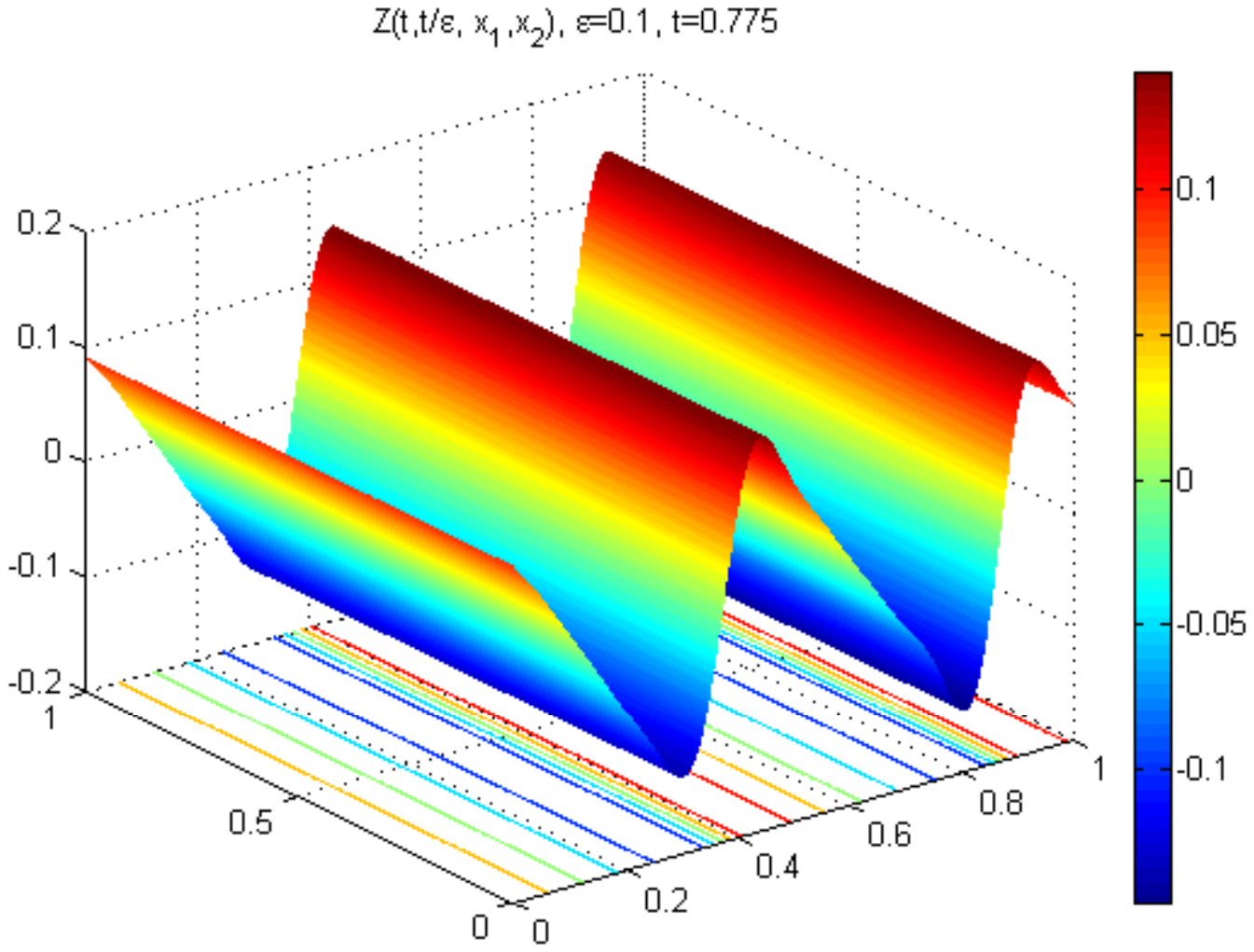}\includegraphics[angle=0, width=8cm]{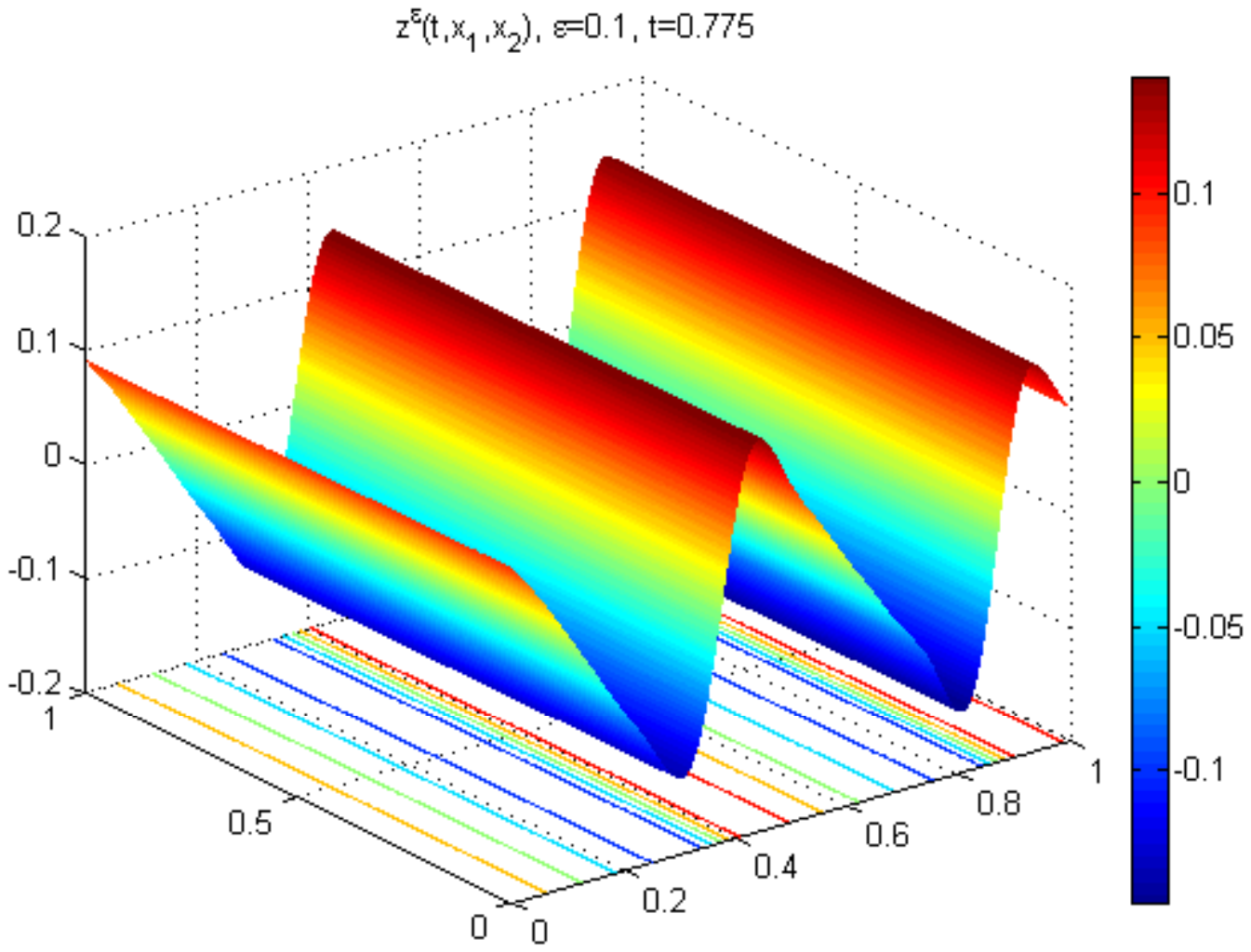}\\\end{center}
 \caption{\label{mard4}Comparison  of $z_{P}^{\epsilon}(t,x_1,x_2)$  and $Z_P(t,\frac{t}{\epsilon},x_1,x_2)$, $P=4$; $ t=0.775,$ $\epsilon=0.1,\,\,z_0(x_1,x_2)=\cos2\pi x_1+\cos 4\pi x_1$. On the left $Z_P(t,\frac{t}{\epsilon},x_1,x_2)$,
 on the right $z_{P}^\epsilon(t,x_1,x_2)$.}
\end{figure}
\newpage
In Figure\,\ref{mard13} and \ref{mard14}, we do the same but for  $\epsilon=0.005.$ The numerical results show that $z_{P}^\epsilon(t,x)$ is
also very close to $Z_P(t,\frac{t}{\epsilon},x_1,x_2).$
\newpage
\begin{figure}[htbp]
\begin{center}
 \includegraphics[angle=0, width=8cm]{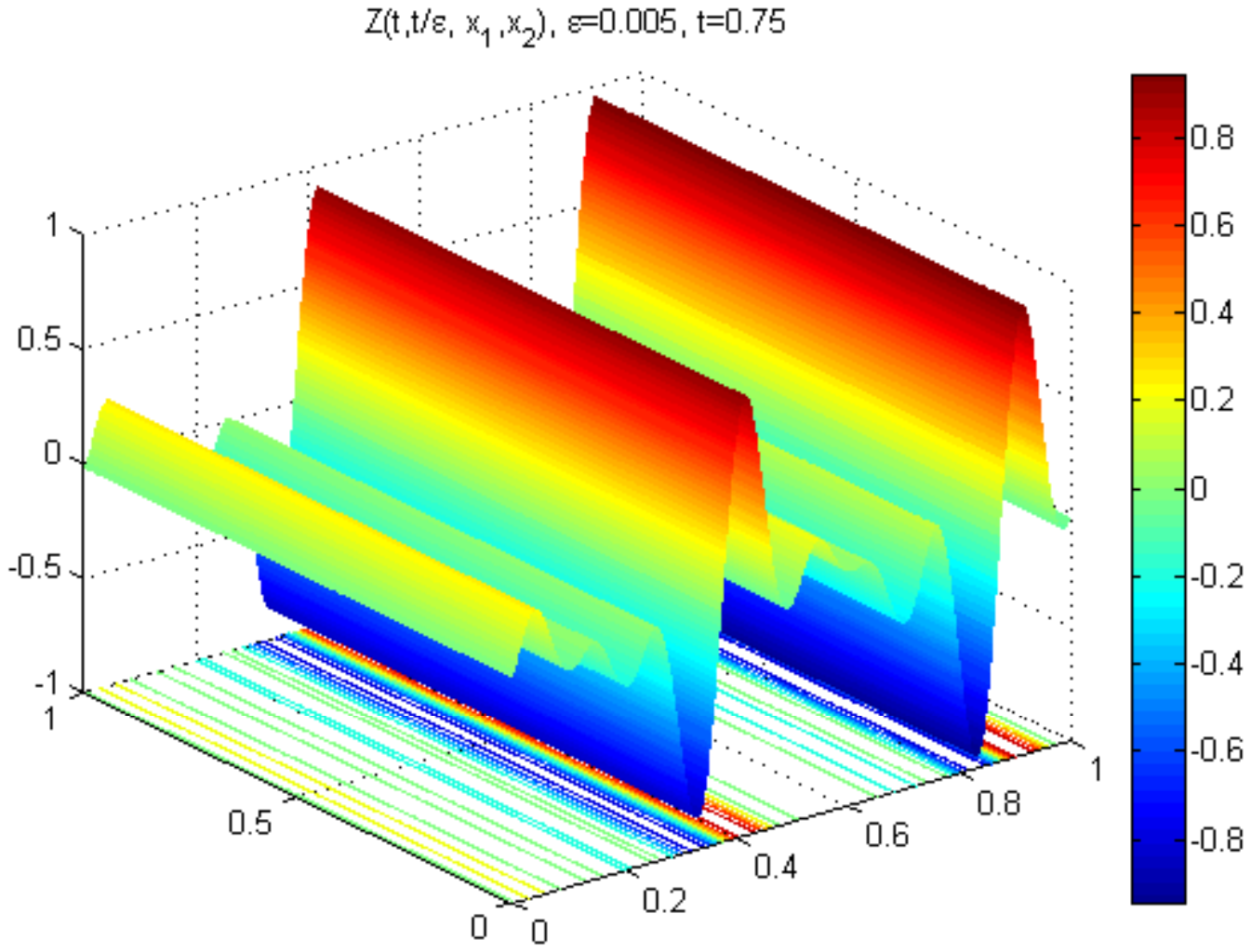}\includegraphics[angle=0, width=8cm]{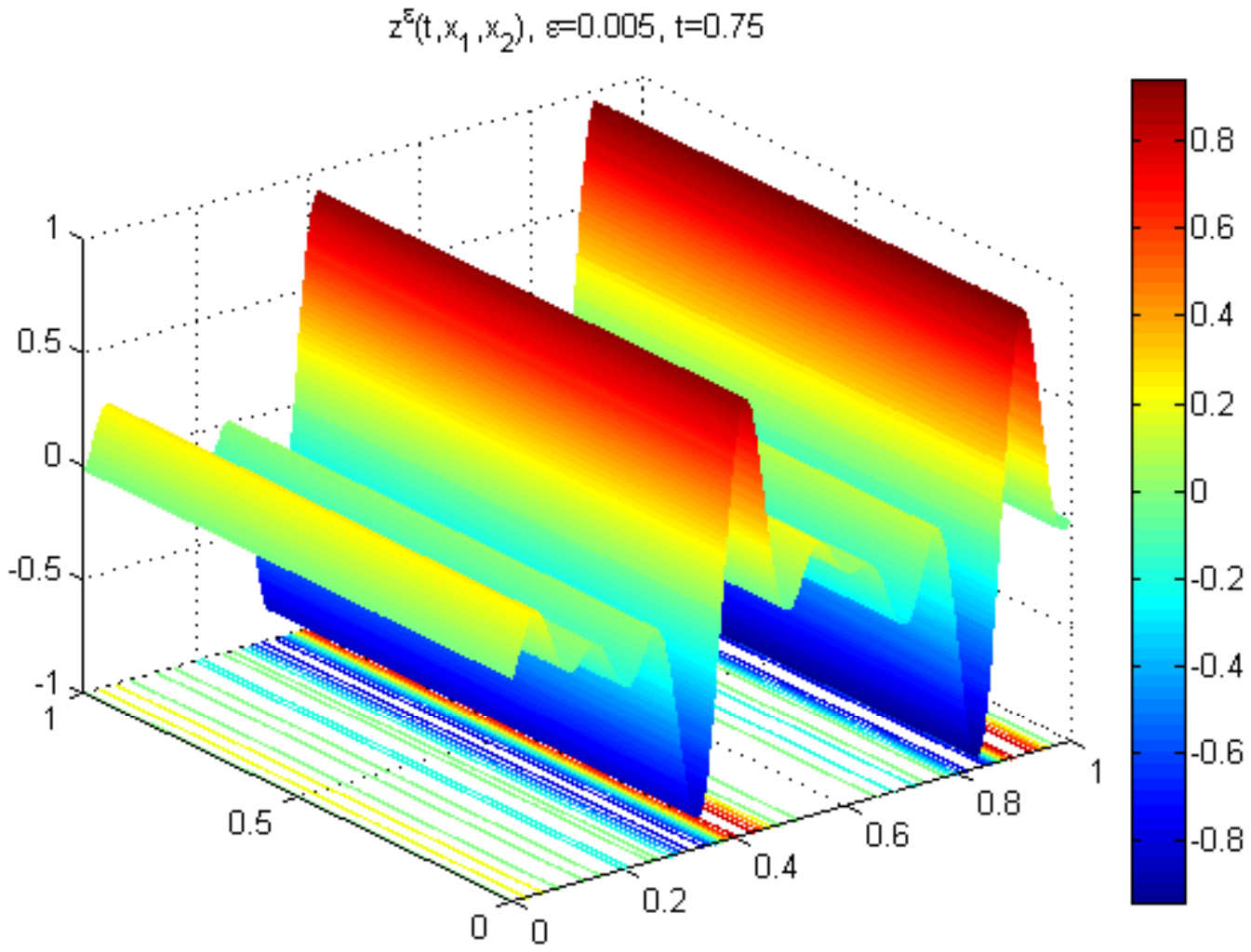}\\\end{center}
 \caption{\label{mard13}Comparison of $z_P^{\epsilon}(t,x_1,x_2)$  and $Z_P(t,\frac{t}{\epsilon},x_1,x_2)$, $ P=4$; $t=0.75,$
 $\epsilon=0.005,\,\,z_0(x_1,x_2)=\cos2\pi x_1+\cos 4\pi x_1.$ On the left $Z_P(t,\frac{t}{\epsilon},x_1,x_2)$, on the right
 $z_{P}^\epsilon(t,x_1,x_2)$.}
\begin{center}
 \includegraphics[angle=0, width=8cm]{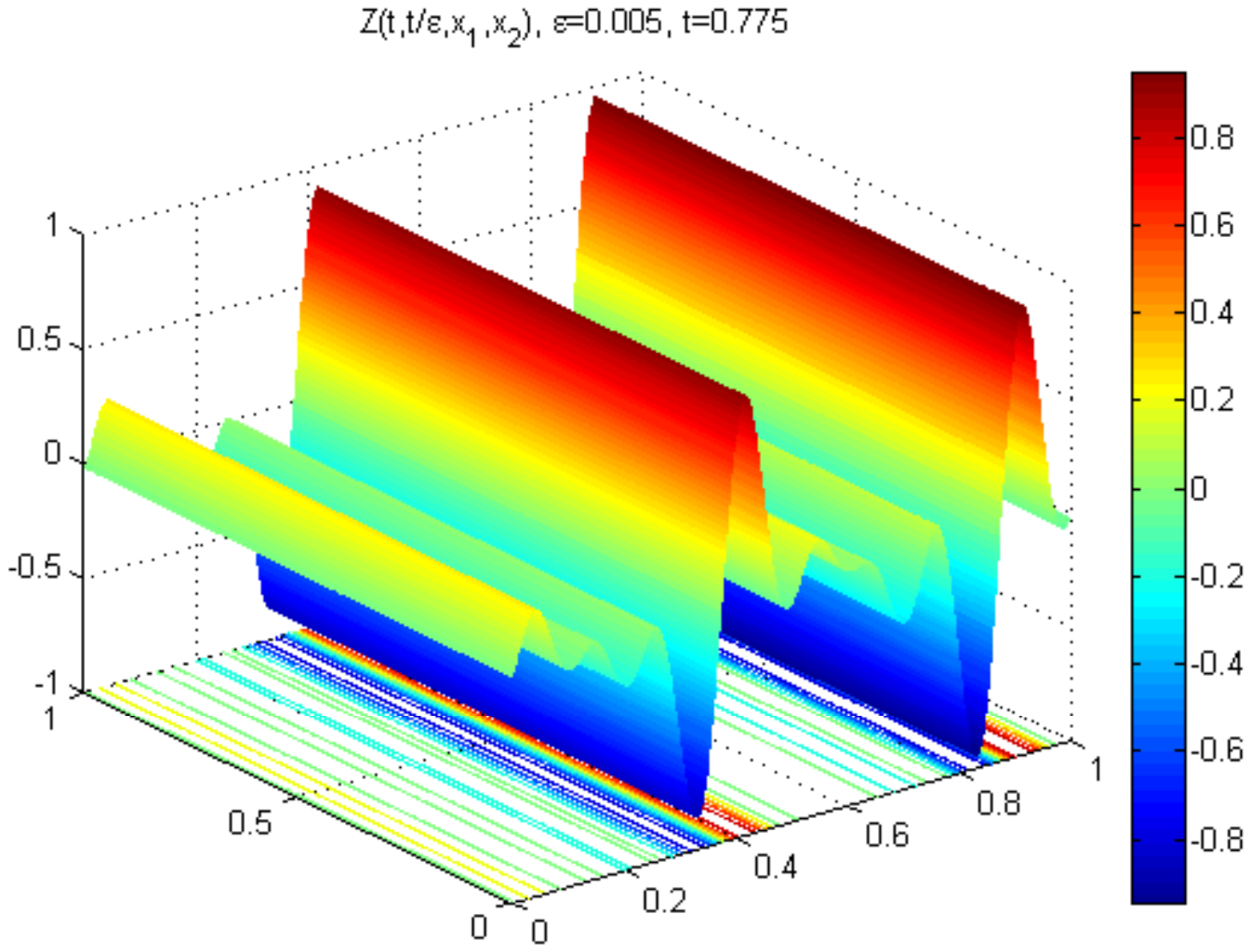}\includegraphics[angle=0, width=8cm]{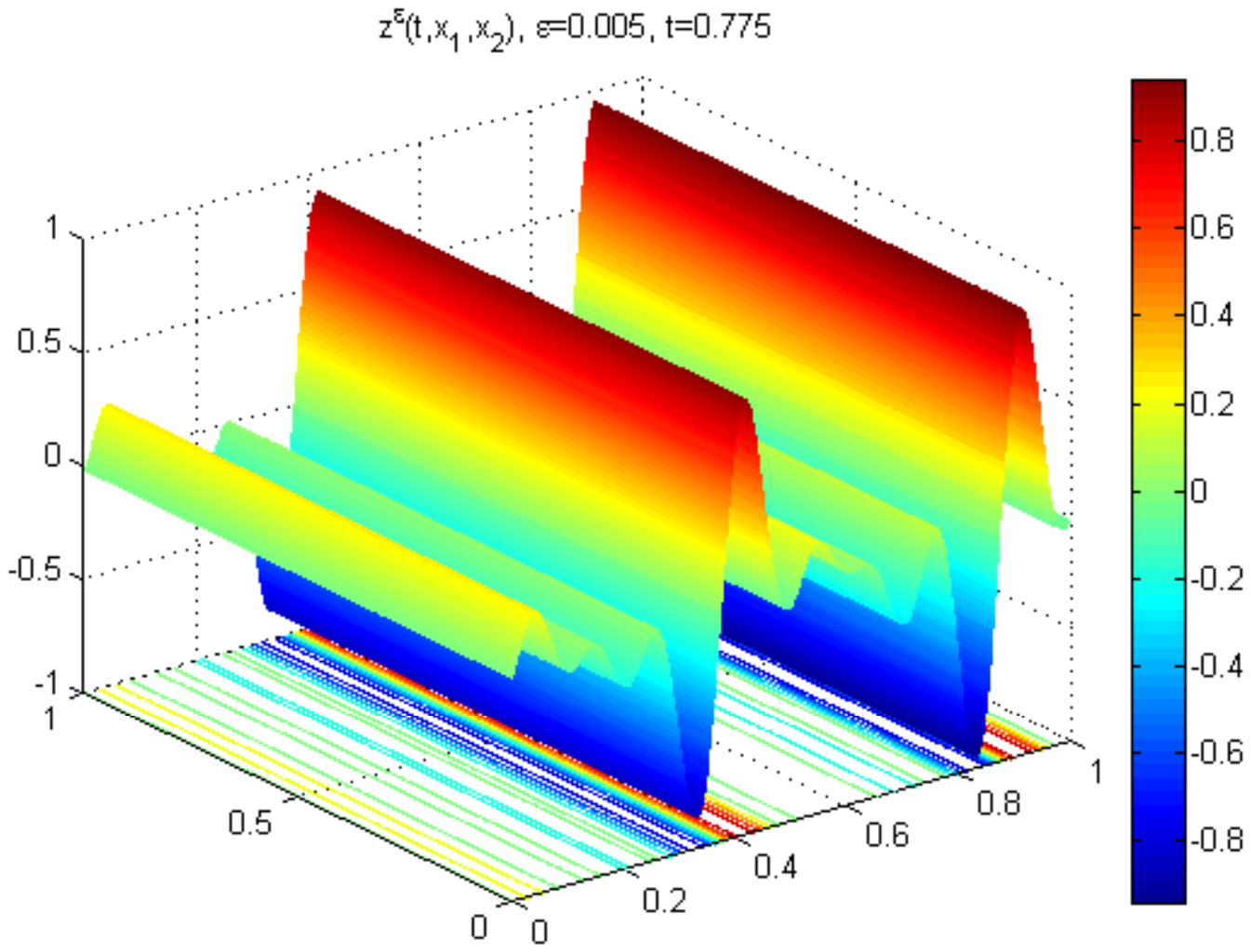}\\\end{center}
 \caption{\label{mard14}Comparison of $z_{P}^\epsilon(t,x_1,x_2)$  and $Z(t,\frac{t}{\epsilon},x_1,x_2)$, $P=4$;
 $t=0.775,$ $\epsilon=0.005,\,\,z_0(x_1,x_2)=\cos2\pi x_1+\cos 4\pi x_1.$ On the left $Z_P(t,\frac{t}{\epsilon},x_1,x_2)$,
 on the right $z_{P}^\epsilon(t,x_1,x_2)$.}
\end{figure}
We remark that for $\epsilon=0.1$  and $\epsilon=0.005,$ the solution $z_{P}^\epsilon(t,x)$ is very close
to $Z_P(t,\frac t \epsilon,x).$ But the approximation $z_{P}^\epsilon(t,x)\sim Z_P(t,\frac t\epsilon,x)$ is very good when $\epsilon$ is very small.\\
To show that $z_{P}^\epsilon $ is very close to $Z_P, $ we construct the same figures as previously but  in dimension 2
with $\epsilon=0.005 $ i.e. we construct $z_{P}^\epsilon(t,x_1,0)$ and $Z_P(t,\frac{t}{\epsilon},x_1,0)$ for $\epsilon=0.005$
at time $t=0.775$. This is given in Figure\,\ref{fg4}.
\newpage
\begin{figure}
\begin{center}
 \includegraphics[angle=0, width=8cm]{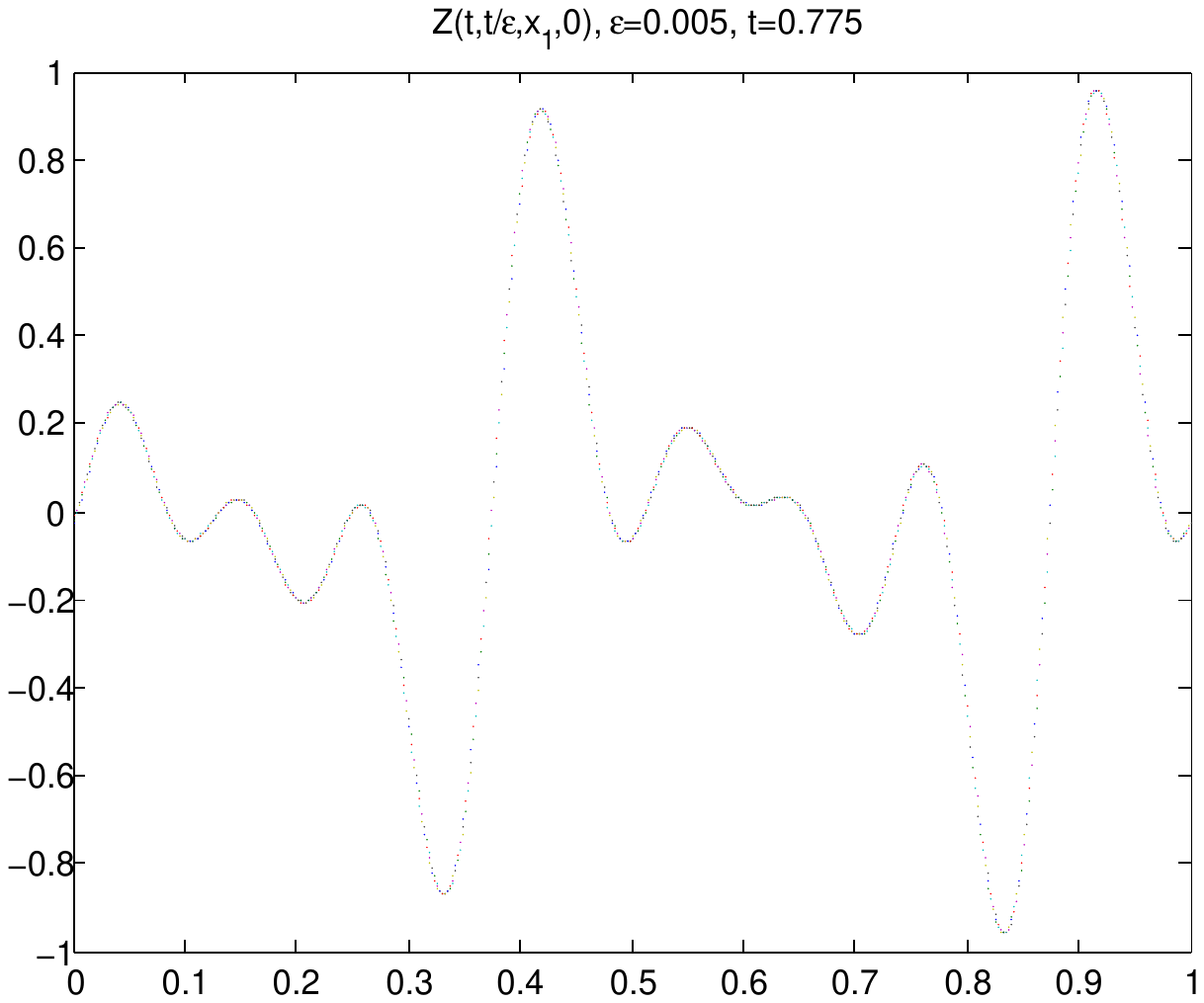}\includegraphics[angle=0, width=8cm]{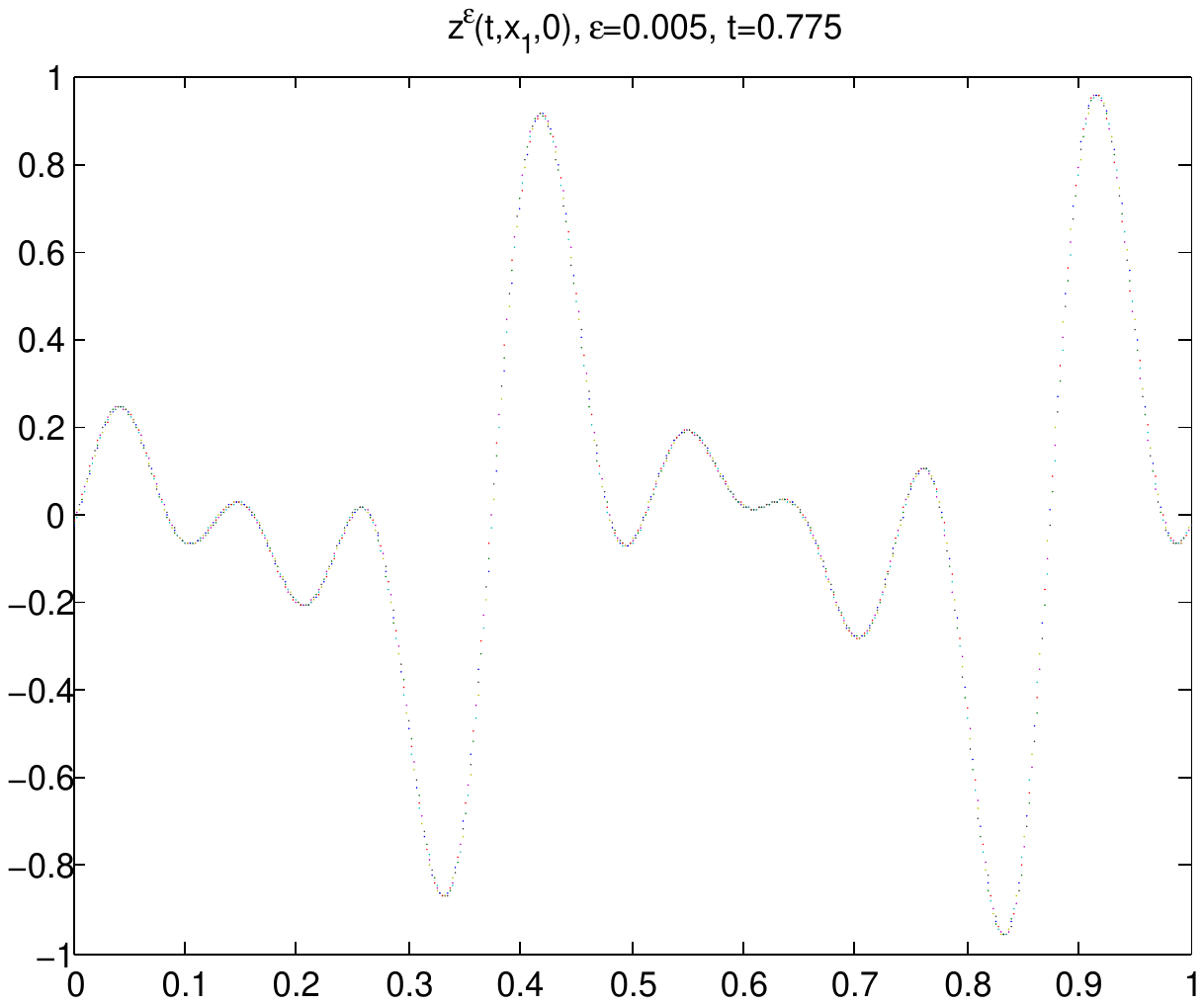}\\\end{center}
  \caption{\label{fg4} Comparison  of $z_{P}^\epsilon(t,x_1,0)$  and $Z_P(t,\frac{t}{\epsilon},x_1,0),\,\,t=0.775,\,\,
  \epsilon=0.005.$ On the left $Z_P(t,\frac{t}{\epsilon},x_1,0)$, on the right $z_{P}^\epsilon(t,x_1,0).$}
\end{figure}
~
%
\begin{figure}[htbp]
\begin{center}
 \includegraphics[angle=0, width=7cm]{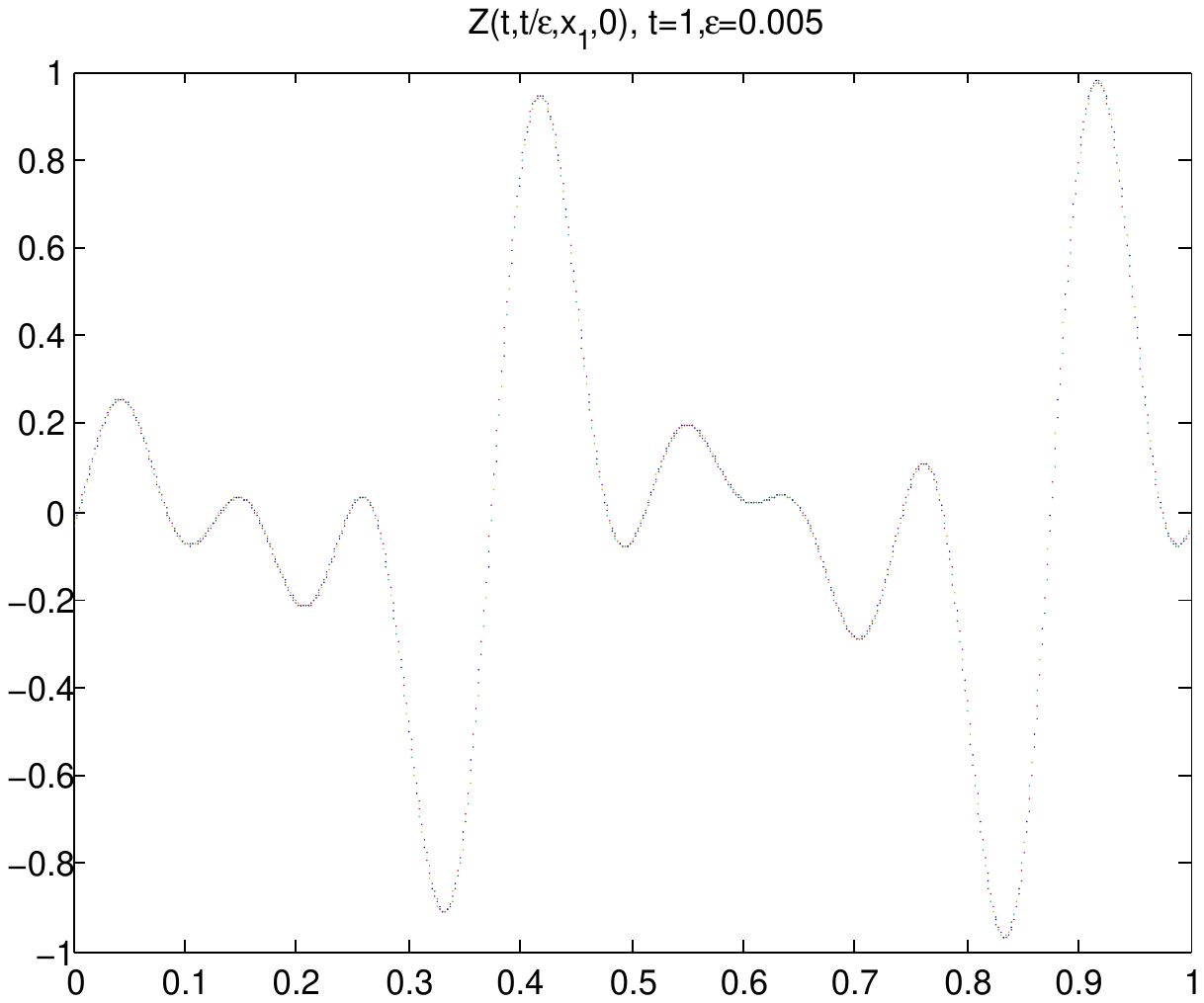}\includegraphics[angle=0, width=7cm]{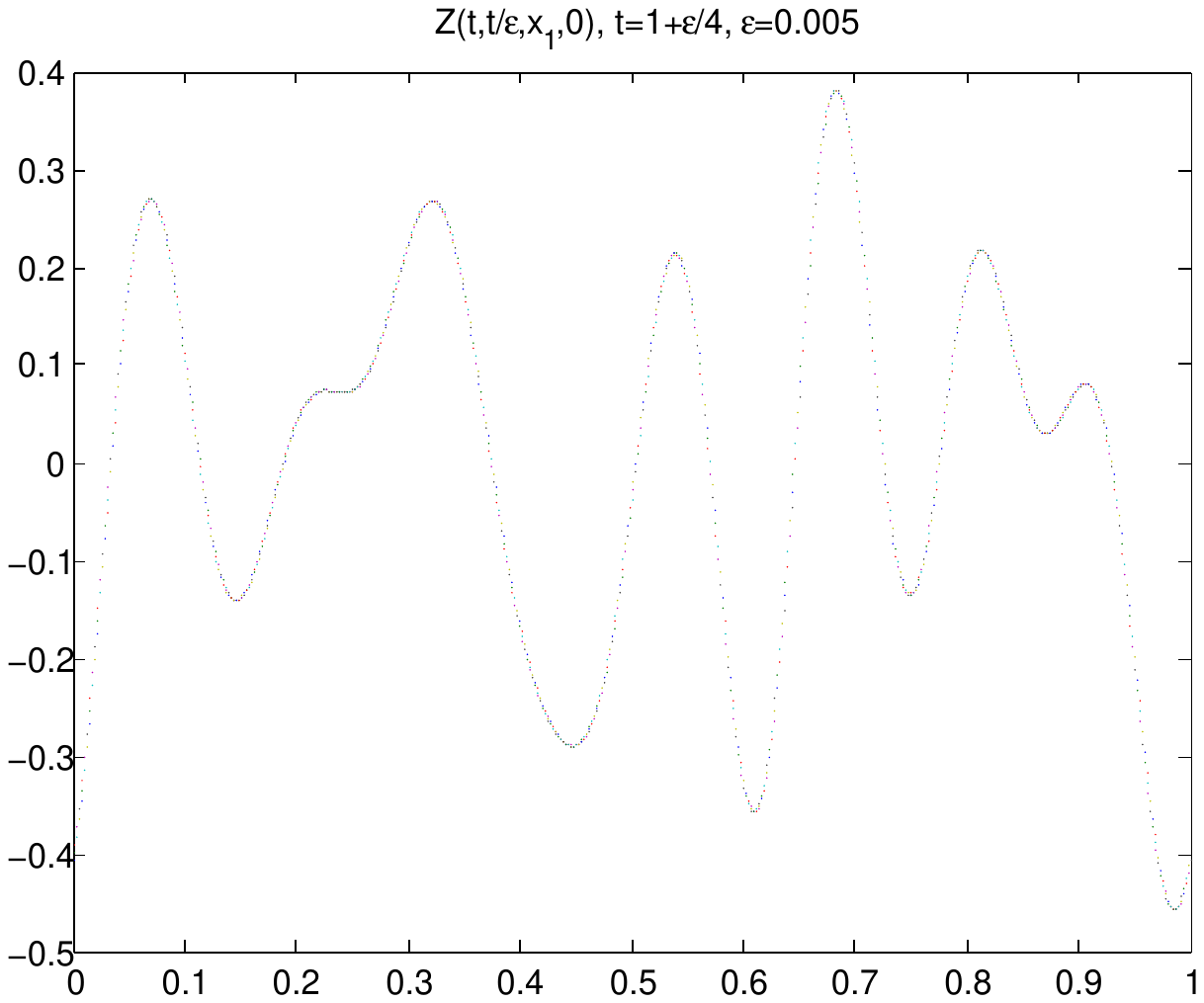}\\\includegraphics[angle=0, width=7cm]{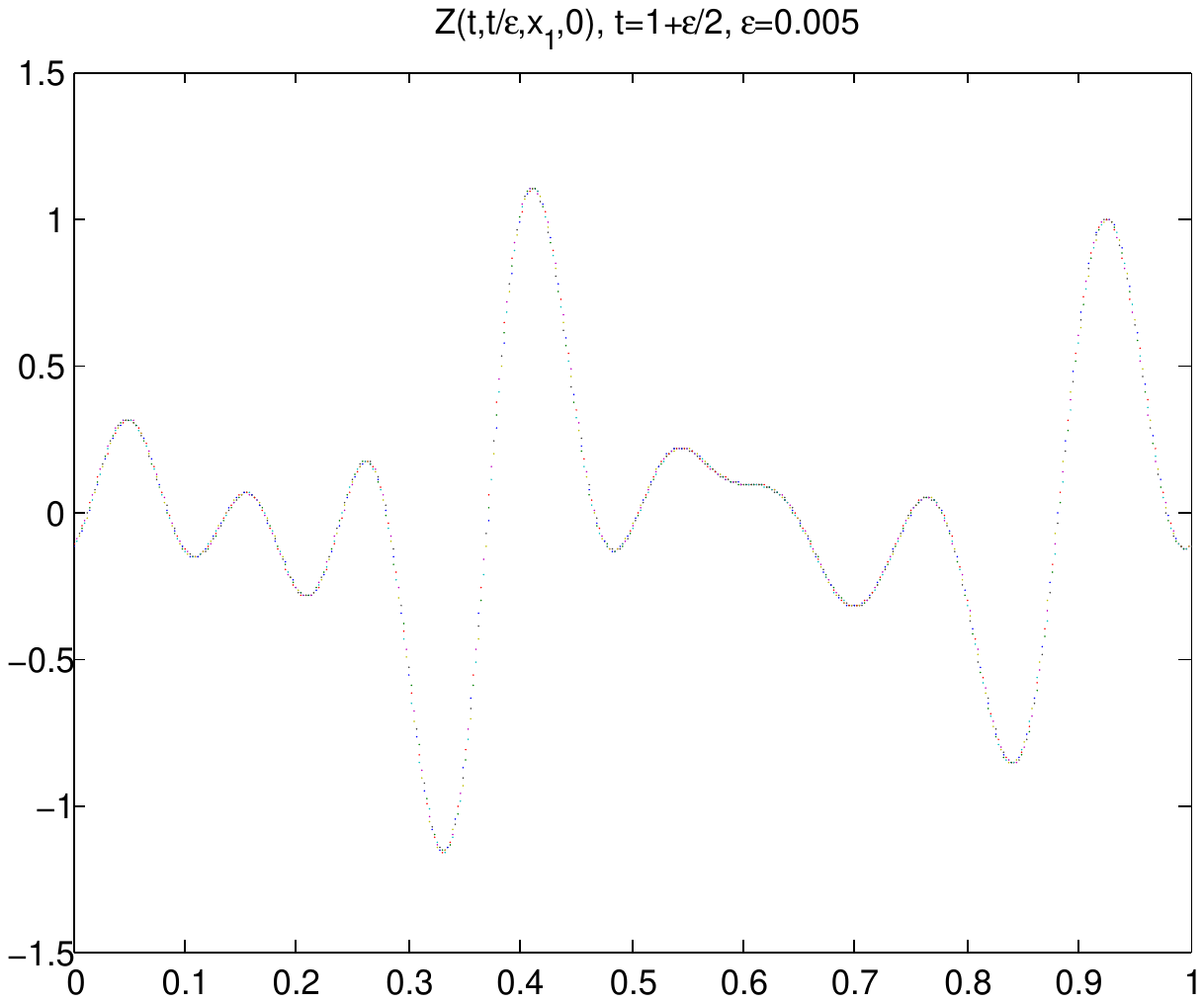}\includegraphics[angle=0, width=7cm]{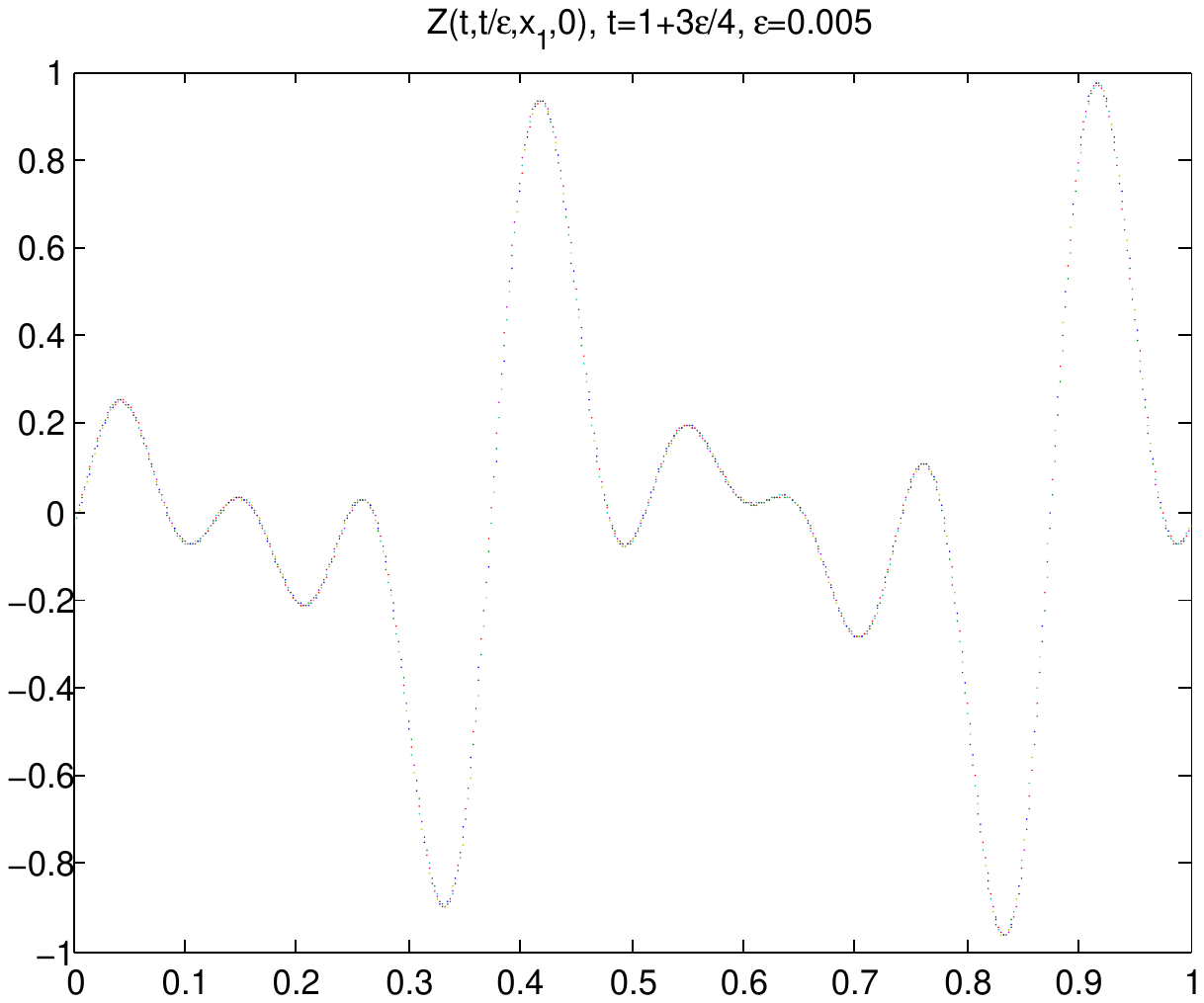}\\
 \includegraphics[angle=0, width=7cm]{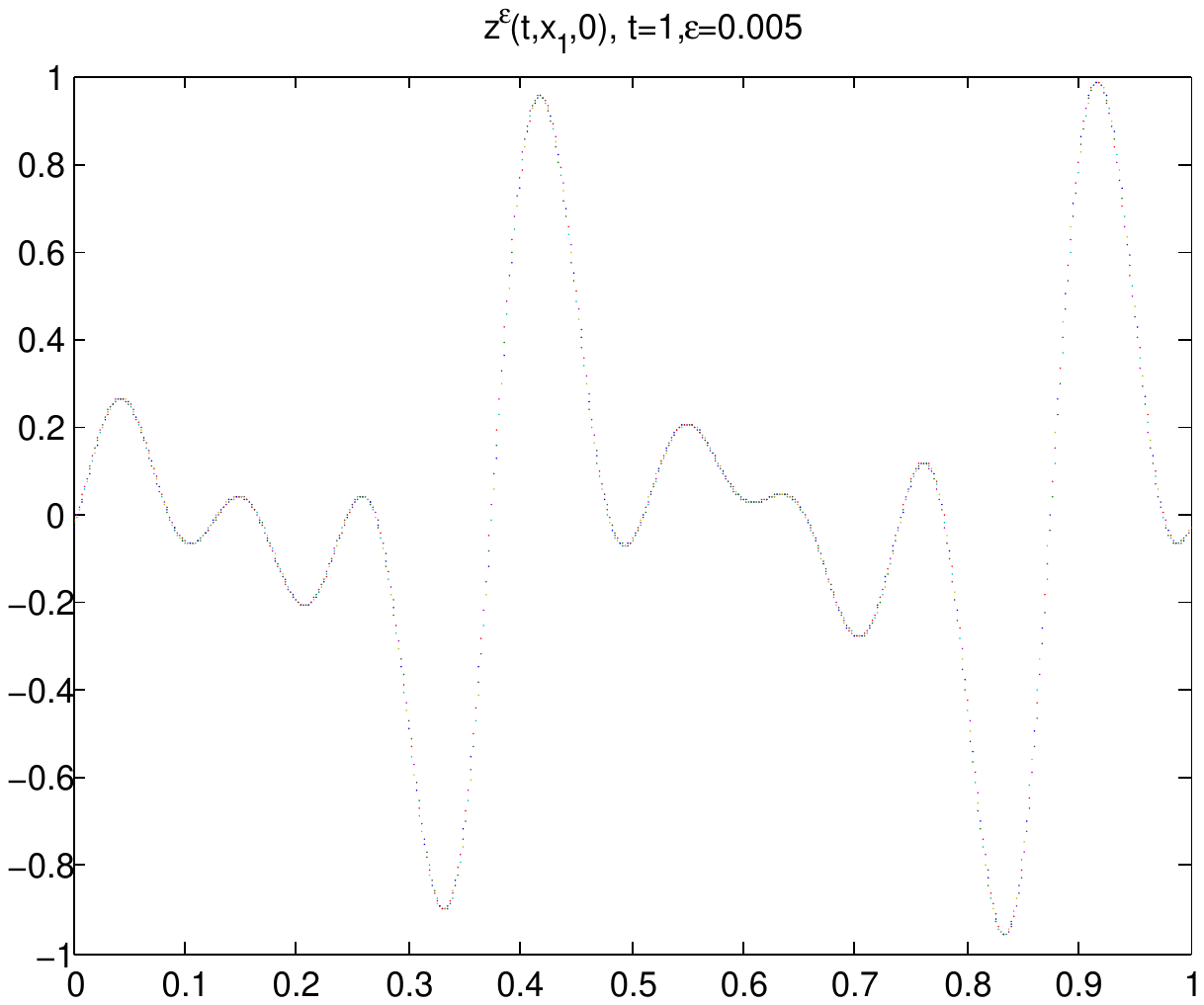}\includegraphics[angle=0, width=7cm]{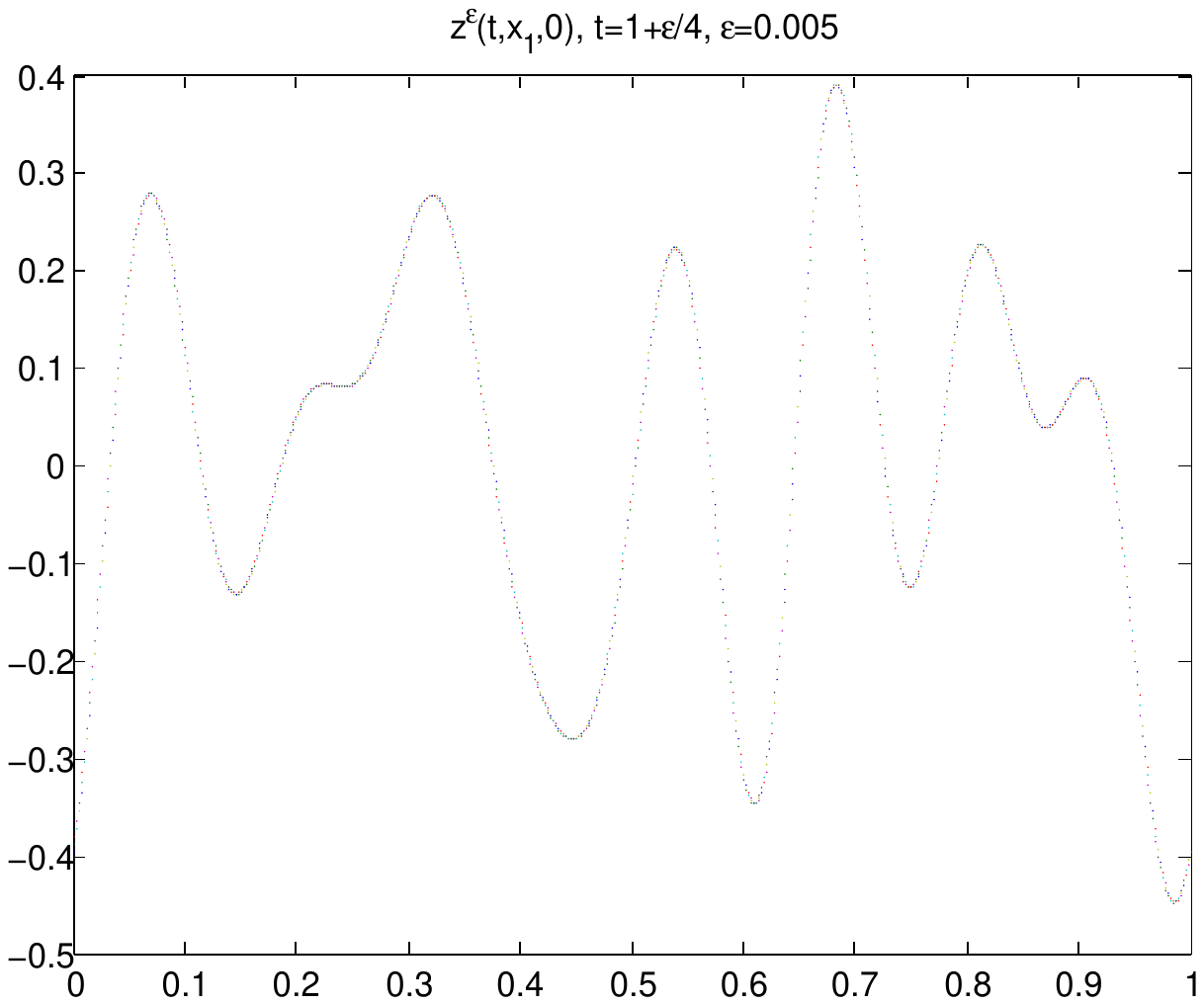}\\\includegraphics[angle=0, width=7cm]{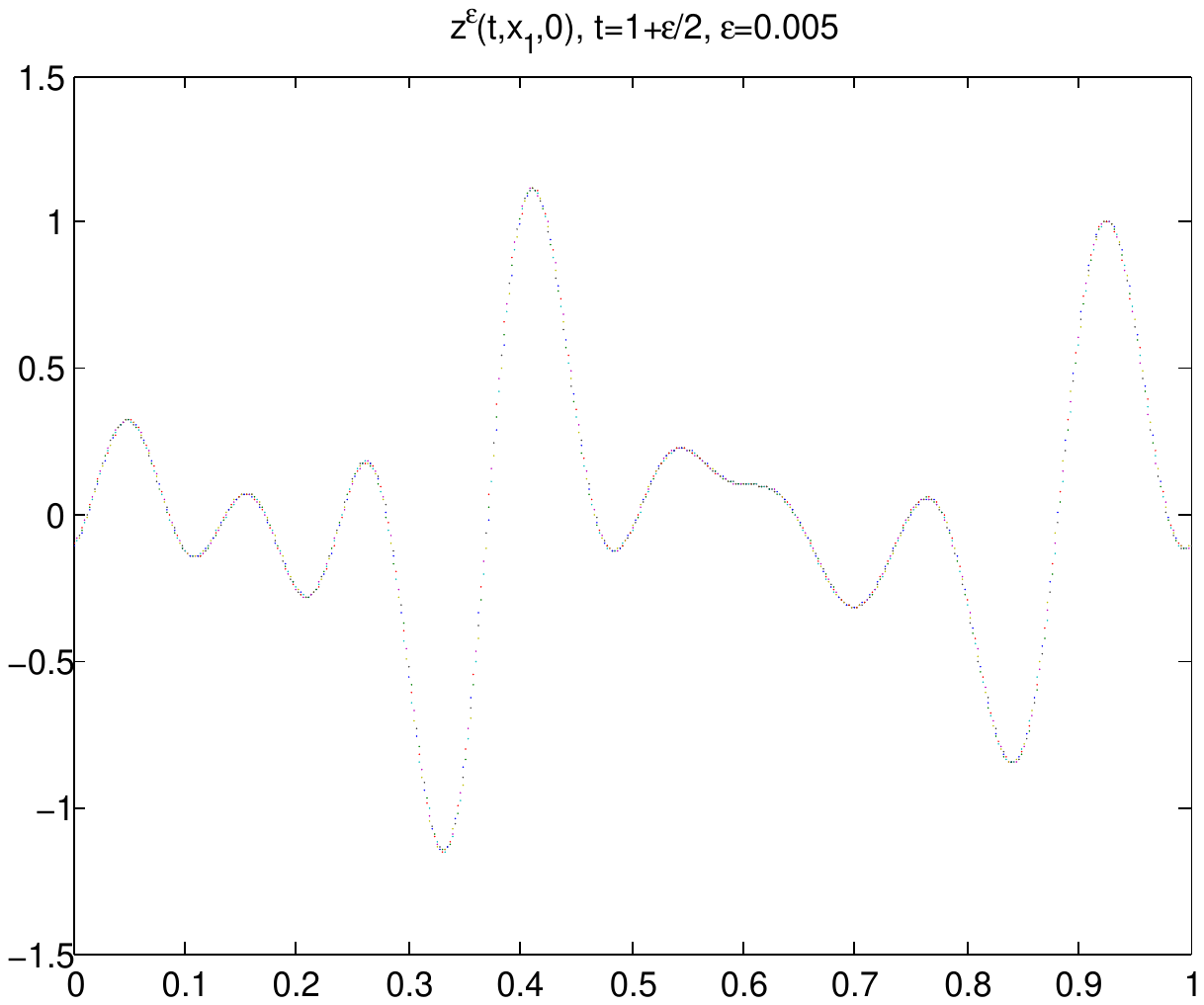}\includegraphics[angle=0, width=7cm]{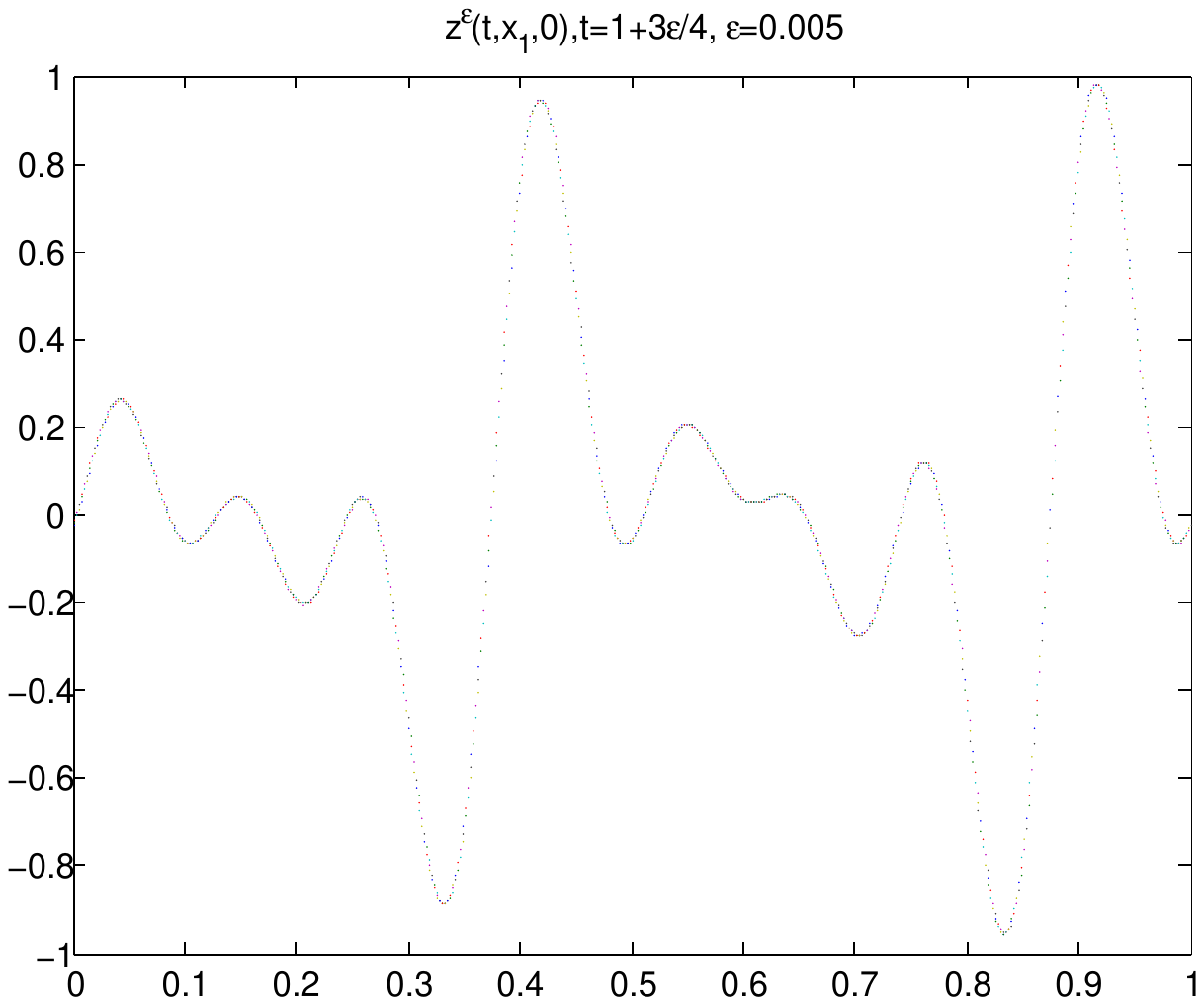}\\\end{center}
  \caption{\label{imper2}Evolution of $Z_P(t,\frac{t}{\epsilon},x_1,0)$(top) and $z_{P}^\epsilon(t,x_1,0)(\text{bottom}),
  \,\,t=1+n\epsilon/4,\,\,n=0,1,2,3.$
}
\end{figure}
\noindent The results in Figure\,\ref{imper2} show that $Z_P$ and $z_{P}^\epsilon$
have the same behavior in the same period and $Z_P$ is very close to $z_{P}^\epsilon.$
We also notice that, despite the small shifts that occur during a period, the two solutions glue together.

\end{document}